\documentclass[a4paper,12pt]{amsart}
\usepackage{setspace}
\usepackage[left=20mm,head=30mm,bottom=30mm,foot=20mm,right=20mm]{geometry}
\usepackage{amsmath,amsthm,amsfonts,amssymb,mathtools,mathrsfs,url,bm,enumitem}
\usepackage{newtxtext}
\usepackage[varvw]{newtxmath}
\usepackage[hiresbb]{graphicx}
\mathtoolsset{showonlyrefs=true}
\theoremstyle{remark}
  \newtheorem{remark}{Remark}[section]
\theoremstyle{plain}
  \newtheorem{theorem}{Theorem}[section]
  \newtheorem{corollary}{Corollary}[section]
  \newtheorem{proposition}{Proposition}[section]
  \newtheorem{lemma}{Lemma}[section]

\newcommand\restr[2]{{
  \left.\kern-\nulldelimiterspace 
  #1 
  \vphantom{\big|} 
  \right|_{#2} 
}}
\newcommand{\lam}{\lambda}
\DeclareMathOperator{\sgn}{sgn}
\DeclareMathOperator{\erfc}{erfc}
\renewcommand{\Re}{\operatorname{Re}}
\usepackage{color}

\renewcommand{\approx}{\sim}

\begin{document}

\title[Long-Time Behavior of a Point Mass in a 1D Viscous Compressible Fluid]{Long-Time Behavior of a Point Mass in a One-Dimensional Viscous Compressible Fluid and Pointwise Estimates of Solutions}
\author{Kai Koike}
\address{School of Fundamental Science and Technology, Keio University, 3-14-1 Hiyoshi, Kohoku-ku, Yokohama, Kanagawa 223-8522, Japan}
\email{koike@math.keio.ac.jp}
\address{Center for Advanced Intelligence Project, RIKEN, 1-4-1 Nihonbashi, Chuo-ku, Tokyo 103-0027, Japan}
\date{\today}

\begin{abstract}
  We consider the motion of a point mass in a one-dimensional viscous compressible barotropic fluid. The fluid--point mass system is governed by the barotropic compressible Navier--Stokes equations and Newton's equation of motion. Our main result concerns the long-time behavior of the fluid and the point mass, and it gives pointwise convergence estimates of the volume ratio and the velocity of the fluid to their equilibrium values. As a corollary, it is shown that the velocity $V(t)$ of the point mass satisfies a decay estimate $|V(t)|=O(t^{-3/2})$ --- a faster decay compared to $t^{-1/2}$ known for the motion of a point mass in the viscous Burgers fluid~[J.~L.~V{\'{a}}zquez and E.~Zuazua, Comm. Partial Differential Equations \textbf{28} (2003), 1705--1738]. The rate $-3/2$ is essentially related to the compressibility and the nonlinearity. As a consequence, it follows that the point mass is convected only a finite distance as opposed to the viscous Burgers case. The main tool used in the proof is the pointwise estimates of Green's function. It turns out that the understanding of the time-decay properties of the transmitted and reflected waves at the point mass is essential for the proof.
\end{abstract}

\maketitle

\tableofcontents

\section{Introduction}
Let us start by recalling a theorem proved by V\'{a}zquez and Zuazua~\cite[Theorem~1.2]{VZ03}. Consider a one-dimensional system consisting of a fluid and a point mass. Let $m$ be the mass of the point mass and $x=h(t)$ its location. Assuming that the fluid velocity $u=u(x,t)$ is governed by the viscous Burgers equation, the fluid--point mass system is described by the following equations:
\begin{equation}
  \begin{dcases}
    u_t+(u^2)_x=u_{xx},                        & x\in \mathbb{R}\backslash \{ h(t) \},\, t>0, \\
    u(h(t)\pm 0,t)=h'(t),                      & t>0, \\
    mh''(t)=\llbracket u_x \rrbracket(h(t),t), & t>0, \\
    h(0)=h_0,\, h'(0)=h_1;\, u(x,0)=u_0(x),    & x\in \mathbb{R}\backslash \{ h_0 \},
  \end{dcases}
\end{equation}
where $\llbracket f \rrbracket(x,t)\coloneqq f(x+0,t)-f(x-0,t)$ for a function $f=f(x,t)$. The theorem of V\'{a}zquez and Zuazua shows that the long-time behavior of $u$ is well approximated by the self-similar solution $\bar{u}$ to the viscous Burgers equation with mass $M=\int_{\mathbb{R}\backslash \{ h_0 \}}u_0(x)\, dx+mh_1$:
\begin{equation}
  t^{(1-1/p)/2}||u(t)-\bar{u}(t)||_{L^{p}(\mathbb{R}\backslash \{ h(t) \})}\to 0 \quad \text{as $t\to \infty$}
\end{equation}
for all $1\leq p\leq \infty$. Thus in particular, if $M$ is non-zero, then $||u(t)||_{L^{\infty}(\mathbb{R}\backslash \{ h(t) \})}\approx t^{-1/2}$ since $\bar{u}$ is so.\footnote{We write $|f(t)|\approx t^{-\alpha}$ ($\alpha>0$) to mean $C^{-1}(t+1)^{-\alpha}\leq |f(t)|\leq C(t+1)^{-\alpha}$ for some constant $C\geq 1$.}~Moreover, they also showed that $|h'(t)|\approx t^{-1/2}$. Thus in this case, $||u(t)||_{L^{\infty}(\mathbb{R}\backslash \{ h(t) \})}$ and $|h'(t)|=|u(h(t)\pm 0,t)|$ both decay as $t^{-1/2}$. We note in particular that the point mass is convected to spatial infinity since $\int_{0}^{\infty}(s+1)^{-1/2}\, ds=\infty$.

In this paper, this theorem is extended to the case when the fluid is governed by the barotropic compressible Navier--Stokes equations instead of the viscous Burgers equation. Contrary to the result above, the following decay estimates are obtained under certain assumptions on the initial data: The fluid velocity $U$ and the point mass velocity $h'$ decay with different rates as
\begin{equation}
  ||U(t)||_{L^{\infty}(\mathbb{R}\backslash \{ h(t) \})}\approx t^{-1/2},\quad |h'(t)|=O(t^{-3/2}).
\end{equation}
It might be of interest to note that, contrary to the Burgers case, the point mass is convected only a finite distance since $\int_{0}^{\infty}(s+1)^{-3/2}\, ds<\infty$. We also note that while $|h'(t)|\approx t^{-3/2}$ does not hold in general,\footnote{For initial data with certain symmetry, it could happen that $h'(t)=0$ for all $t\geq 0$; see the footnote in Remark~\ref{Remark:Optimality}.}~the decay rate $-3/2$ is optimal in the sense that there are several initial data for which we were able to prove that $|h'(t)|\approx t^{-3/2}$; see Remark~\ref{Remark:Optimality}. Furthermore, numerical experiments using a simple finite difference method (that are not presented in this paper) suggest that $|h'(t)|\approx t^{-3/2}$ holds more generally.

We prove these decay estimates by showing pointwise error bounds for the difference between the solution and a linear combination of self-similar solutions to Burgers equations with convection terms (Theorem~\ref{SectionI:Theorem:MainTheorem:PointwiseEstimates}). Although the $L^{\infty}$-norm of the self-similar solutions decay as $t^{-1/2}$, their decay rates at the location of the point mass, $x=h(t)$, are exponential. This is due to the compressibility of the fluid that makes the slowly decaying part of the self-similar solutions to move away from $x=h(t)$. In contrast, since there is no compressibility in the Burgers fluid, the decay rates of $||u(t)||_{L^{\infty}(\mathbb{R}\backslash \{ h(t) \})}$ and $|h'(t)|=|u(h(t)\pm 0,t)|$ are the same. Furthermore, due to the nonlinearity, error bounds are of the order of $O(t^{-3/2})$ at $x=h(t)$, which leads to $|h'(t)|=|U(h(t)\pm 0,t)|=O(t^{-3/2})$; see Remark~\ref{Remark:Nonlinearity_is_Essential}.

The difficulty in proving the result above is that we need to capture the pointwise structure of the solution in order to show that $|h'(t)|$ decays faster than $||U(t)||_{L^{\infty}(\mathbb{R}\backslash \{ h(t) \})}$. The method of Green's function developed for the Cauchy problem by Liu and Zeng~\cite{LZ97,Zeng94} is a natural tool to accomplish this. In this method, we first derive sufficiently precise pointwise estimates of Green's function (fundamental solution) by the use of Fourier transform techniques. The solution has an integral representation in terms of the initial data, the nonlinear terms (thought of as inhomogeneous terms) and Green's function. We first assume that the solution satisfies certain ansätze (the pointwise bounds we wish to prove), which are then used to evaluate the nonlinear terms; then we show, by using the pointwise estimates of Green's function, that the terms in the integral representation satisfy the same ansätze. This will show, by some arguments, that the solution indeed satisfies the desired ansätze. This method was recently extended to the half-space problem with the Robin boundary condition by Du and Wang~\cite{DW18} (see also~\cite{Deng16,DWY15,DW18}). It turns out that their method is also useful to our problem in deriving pointwise estimates of Green's function; however, a major problem is that they have not identified the leading terms of the solution, which turn out to be the self-similar solutions to the Burgers equations with convection terms. Without identifying the leading terms, it will only lead to a suboptimal decay estimate $|h'(t)|=O(t^{-1})$ for our problem. In this paper, we succeeded in identifying the leading terms and as a corollary obtained the decay estimate $|h'(t)|=O(t^{-3/2})$ (Corollary~\ref{Corollary:DecayEstimateV}). A crucial point is the understanding of the nature of transmission and reflection of waves at the point mass. By carefully looking at the structure of Green's function, we found that the reflected waves decay faster than the transmitted waves. Using this observation, we were able to obtain pointwise estimates of solutions which resemble the case of the Cauchy problem.

The organization of the rest of the paper is as follows. In the remaining of this section, we explain the formulation of the problem, first in Eulerian coordinate and then in Lagrangian mass coordinate, and then state the main theorems; relation to other works is also discussed at the end of this section. In Section~\ref{SectionII}, the theorem on the global existence of solutions (Theorem~\ref{SectionI:Theorem:MainTheorem:GlobalExistence}) is proved. The pointwise estimates of solutions (Theorem~\ref{SectionI:Theorem:MainTheorem:PointwiseEstimates}), the main objective of this paper, are proved in Section~\ref{SectionIII}.

\subsection{Motion of a Point Mass in a One-Dimensional Viscous Compressible Fluid}
\label{SectionI:Subsection:Formulation}
We now explain the formulation of the problem. Consider a one-dimensional system consisting of a viscous compressible barotropic fluid and a point mass. Let $m$ be the mass of the point mass and $X=h(t)$ its location, and let $\rho=\rho(X,t)$ and $U=U(X,t)$ be the density and the velocity of the fluid. Then the fluid--point mass system is described by the following equations:
\begin{equation}
  \label{SectionI:CNSwithPointMass}
  \begin{dcases}
    \rho_t+(\rho U)_X=0,                                    & X\in \mathbb{R}\backslash \{ h(t) \},\, t>0, \\
    (\rho U)_t+(\rho U^2)_X+P(\rho)_X=\nu U_{XX},           & X\in \mathbb{R}\backslash \{ h(t) \},\, t>0, \\
    U(h(t)\pm 0,t)=h'(t),                                   & t>0, \\
    mh''(t)=\llbracket -P(\rho)+\nu U_X \rrbracket(h(t),t), & t>0, \\
    h(0)=h_0,\, h'(0)=h_1, \\
    \rho(X,0)=\rho_0(X),\, U(X,0)=U_0(X),                   & X\in \mathbb{R}\backslash \{ h_0 \},
  \end{dcases}
\end{equation}
where the positive constant $\nu$ is the viscosity and $P=P(\rho)$ is the pressure (we assume that $P$ is smooth). The first two equations are the barotropic compressible Navier--Stokes equations; the third one is the requirement that the fluid does not penetrate through the point mass; the fourth equation is Newton's equation of motion, and the rest are the initial conditions. In what follows, we put $m=1$ for simplicity.

Let us rewrite~\eqref{SectionI:CNSwithPointMass} in the Lagrangian mass coordinate to consider the problem in a fixed domain. For $x\in \mathbb{R}_* \coloneqq \mathbb{R}\backslash \{ 0 \}$ and $t\geq 0$, define $X=X(x,t)$ by the equation
\begin{equation}
  x=\int_{h(t)}^{X(x,t)}\rho(X',t)\, dX'.
\end{equation}
Assuming that $\inf_{X\in \mathbb{R}\backslash \{ h(t) \}}\rho(X,t)>0$, this defines a bijection
\begin{equation}
  \mathbb{R}_* \ni x\mapsto X(x,t)\in \mathbb{R}\backslash \{ h(t) \}.
\end{equation}
Then define new dependent variables by $v(x,t)\coloneqq \rho^{-1}(X(x,t),t)$ and $u(x,t)\coloneqq U(X(x,t),t)$. Here, $v$ is the specific volume. Moreover, let $p(v)\coloneqq P(v^{-1})$, $v_0(x)\coloneqq \rho_{0}^{-1}(X(x,0))$ and $u_0(x)\coloneqq U_0(X(x,0))$. Then~\eqref{SectionI:CNSwithPointMass} is equivalent to:
\begin{equation}
  \label{SectionI:CNSwithPointMass:Lagrangian}
  \begin{dcases}
    v_t-u_x=0,                                               & x\in \mathbb{R}_*,\, t>0, \\
    u_t+p(v)_x=\nu \left( \frac{u_x}{v} \right)_x,           & x\in \mathbb{R}_*,\, t>0, \\
    u(\pm 0,t)=h'(t),                                        & t>0, \\
    h''(t)=\llbracket -p(v)+\nu u_x/v \rrbracket(0,t),       & t>0, \\
    h(0)=h_0,\, h'(0)=h_1;\, v(x,0)=v_0(x),\, u(x,0)=u_0(x), & x\in \mathbb{R}_*.
  \end{dcases}
\end{equation}

We shall consider perturbations around the uniform stationary state $(v,u,h')=(\bar{v},0,0)$. For simplicity, we assume that $\bar{v}=1$. Let $\tau \coloneqq v-1$, $V(t)\coloneqq h'(t)$, $\tau_0 \coloneqq v_0-1$ and $V_0 \coloneqq h_1$. Then~\eqref{SectionI:CNSwithPointMass:Lagrangian} is equivalent to:
\begin{equation}
  \label{SectionI:CNSwithPointMass:Lagrangian:Perturbation}
  \begin{dcases}
    \tau_t-u_x=0,                                                  & x\in \mathbb{R}_*,\, t>0, \\
    u_t+p(1+\tau)_x=\nu \left( \frac{u_x}{1+\tau} \right)_x,       & x\in \mathbb{R}_*,\, t>0, \\
    u(\pm 0,t)=V(t),                                               & t>0, \\
    V'(t)=\llbracket -p(1+\tau)+\nu u_x/(1+\tau) \rrbracket(0,t),  & t>0, \\
    V(0)=V_0;\, \tau(x,0)=\tau_0(x),\, u(x,0)=u_0(x),              & x\in \mathbb{R}_*.
  \end{dcases}
\end{equation}
Neglecting nonlinear terms, we obtain the linearized equations:
\begin{equation}
  \label{SectionI:CNSwithPointMass:Lagrangian:Perturbation:Linearized}
  \begin{dcases}
    \tau_t-u_x=0,                                         & x\in \mathbb{R}_*,\, t>0, \\
    u_t-c^2 \tau_x=\nu u_{xx},                            & x\in \mathbb{R}_*,\, t>0, \\
    u(\pm 0,t)=V(t),                                      & t>0, \\
    V'(t)=\llbracket c^2 \tau +\nu u_x \rrbracket(0,t), & t>0, \\
    V(0)=V_0;\, \tau(x,0)=\tau_0(x),\, u(x,0)=u_0(x),     & x\in \mathbb{R}_*,
  \end{dcases}
\end{equation}
where $c>0$ is the speed of sound defined by $c^2=-p'(1)$; we assume that $p'(1)<0$. The first two equations in~\eqref{SectionI:CNSwithPointMass:Lagrangian:Perturbation:Linearized} can be written as
\begin{equation}
  \bm{u}_t+A\bm{u}_x=B\bm{u}_{xx},
\end{equation}
where
\begin{equation}
  \bm{u}=
  \begin{pmatrix}
    \tau \\
    u
  \end{pmatrix}
  ,\quad A=
  \begin{pmatrix}
    0    & -1 \\
    -c^2 & 0
  \end{pmatrix}
  ,\quad B=
  \begin{pmatrix}
    0 & 0 \\
    0 & \nu
  \end{pmatrix}.
\end{equation}
The eigenvalues of $A$ are $\lambda_1=c$ and $\lambda_2=-c$, and the corresponding right and left eigenvectors are
\begin{equation}
  \label{right_ev}
  r_1=\frac{2c}{p''(1)}
  \begin{pmatrix}
    -1 \\
    c
  \end{pmatrix}
  ,\quad r_2=\frac{2c}{p''(1)}
  \begin{pmatrix}
    1 \\
    c
  \end{pmatrix}
\end{equation}
and
\begin{equation}
  \label{left_ev}
  l_1=\frac{p''(1)}{4c}
  \begin{pmatrix}
    -1 & 1/c
  \end{pmatrix}
  ,\quad l_2=\frac{p''(1)}{4c}
  \begin{pmatrix}
    1 & 1/c
  \end{pmatrix},
\end{equation}
respectively. Here, we assume that $p''(1)\neq 0$. Let $u_i=l_i \bm{u}$, so that $\bm{u}=u_1 r_1+u_2 r_2$. Also let $u_{0i}=\restr{u_i}{t=0}$.

\subsection{Main Theorem}
\label{SectionI:Subsection:MainTheorem}
As in the case of the Cauchy problem~\cite{LZ97,Zeng94}, the leading terms of the solution $\bm{u}$ to~\eqref{SectionI:CNSwithPointMass:Lagrangian:Perturbation} can be described by the self-similar solutions to Burgers equations with convection terms: Let
\begin{equation}
  \label{SectionI:Eq:SelfSimilarSolutionMass}
  m_i \coloneqq \int_{-\infty}^{\infty}l_i
  \begin{pmatrix}
    \tau_0 \\
    u_0
  \end{pmatrix}
  (x)\, dx,\quad m_V \coloneqq l_i
  \begin{pmatrix}
    0 \\
    V_0
  \end{pmatrix},
\end{equation}
and let $\Theta_i$ be the self-similar solution to
\begin{equation}
  \label{SectionI:Thetai}
  \partial_t \Theta_i+\lambda_i \partial_x \Theta_i+\partial_x \left( \frac{\Theta_{i}^{2}}{2} \right)=\frac{\nu}{2}\partial_{x}^{2}\Theta_i, \quad x\in \mathbb{R},\, t>0
\end{equation}
with
\begin{equation}
  \lim_{t\to +0}\Theta_i(x,t)=(m_i+m_V)\delta(x),
\end{equation}
where $\delta$ is the Dirac delta function. To avoid the singularity at $t=0$, let $\theta_i(x,t)\coloneqq \Theta_i(x,t+1)$. This $\theta_i$ is the leading term of $u_i$. We note that an explicit formula for $\theta_i$ can be derived using the Cole--Hopf transformation~\cite[p.~11]{LZ97}:
\begin{equation}
  \label{theta_explicit}
  \theta_i(x,t)=\frac{\sqrt{\nu}}{\sqrt{2(t+1)}}\left( e^{\frac{m_i+m_V}{\nu}}-1 \right) e^{-\frac{(x-\lambda_i(t+1))^2}{2\nu(t+1)}}\left[ \sqrt{\pi}+\left( e^{\frac{m_i+m_V}{\nu}}-1 \right) \int_{\frac{x-\lambda_i(t+1)}{\sqrt{2\nu(t+1)}}}^{\infty}e^{-y^2}\, dy \right]^{-1}.
\end{equation}

Let $v_i \coloneqq u_i-\theta_i$. The main theorem gives pointwise estimates of $v_i$. To state the theorem, we need:
\begin{align}
  \psi_{3/2}(x,t;\lambda_i) & =[(x-\lambda_i(t+1))^2+(t+1)]^{-3/4}, \\
  \tilde{\psi}(x,t;\lambda_i) & =[|x-\lambda_i(t+1)|^3+(t+1)^2]^{-1/2}
\end{align}
and
\begin{equation}
  \label{PTW_PSI_i}
  \Psi_i(x,t)\coloneqq \psi_{3/2}(x,t;\lambda_i)+\tilde{\psi}(x,t;\lambda_{i'}),
\end{equation}
where $i'=3-i$ ($i=1,2$). Note that $\psi_{3/2}(x,t;\lambda_i)$ and $\tilde{\psi}(x,t;\lambda_i)$ are $O(t^{-3/4})$ and $O(t^{-1})$ around the characteristic line $x=\lambda_i t$, respectively, and that they are both $O(t^{-3/2})$ on $x=0$ where the point mass lies.

We first state a theorem on the uniform boundedness of solutions in $H^4$, which is used in the proof of the pointwise estimates. In the following, we use the notation $||\cdot ||_k \coloneqq ||\cdot ||_{H^k(\mathbb{R}_*)}$ for an integer $k\geq 0$. 

\begin{theorem}
  \label{SectionI:Theorem:MainTheorem:GlobalExistence}
  Let $\tau_0,u_0 \in H^4(\mathbb{R}_*)$ and $V_0 \in \mathbb{R}$. Assume that they satisfy the following compatibility conditions:\footnote{We use in the following the notation $\llbracket f \rrbracket(x)\coloneqq f(x+0)-f(x-0)$ for a function $f=f(x)$.}
  \begin{align}
    u_0(\pm 0) & =V_0, \\
    \restr{\left[ -p(1+\tau_0)_x+\nu \left( \frac{u_{0x}}{1+\tau_0} \right)_x \right]}{x=\pm 0} & =\llbracket -p(1+\tau_0)+\nu u_{0x}/(1+\tau_0) \rrbracket(0).
  \end{align}
  Then there exist $\varepsilon_0,C>0$ such that if
  \begin{equation}
    \label{smallness}
    \varepsilon \coloneqq ||\tau_0||_4+||u_0||_4 \leq \varepsilon_0,
  \end{equation}
  then~\eqref{SectionI:CNSwithPointMass:Lagrangian:Perturbation} has a unique classical solution
  \begin{align}
    \tau & \in C([0,\infty);H^4(\mathbb{R}_*))\cap C^1([0,\infty);H^3(\mathbb{R}_*)), \\
    u    & \in C([0,\infty);H^4(\mathbb{R}_*))\cap C^1([0,\infty);H^2(\mathbb{R}_*)), \\
    u_x  & \in L^2(0,\infty;H^4(\mathbb{R}_*)), \\
    V    & \in C^2([0,\infty))
  \end{align}
  satisfying
  \begin{equation}
    \label{solution_smallness}
    ||\tau(t)||_4+||u(t)||_4+\left( \int_{0}^{\infty}||u_x(s)||_{4}^{2}\, ds \right)^{1/2}+\sum_{k=0}^{2}|\partial_{t}^{k}V(t)|\leq C\varepsilon \quad (t\geq 0).
  \end{equation}
\end{theorem}

To state the assumptions required for the pointwise estimates, we introduce:
\begin{equation}
  u_{0i}^{-}(x)\coloneqq \int_{-\infty}^{x}u_{0i}(y)\, dy,\quad u_{0i}^{+}(x)\coloneqq \int_{x}^{\infty}u_{0i}(y)\, dy.
\end{equation}

\begin{theorem}
  \label{SectionI:Theorem:MainTheorem:PointwiseEstimates}
  Let $\tau_0,u_0 \in H^4(\mathbb{R}_*)$ and $V_0 \in \mathbb{R}$. Assume that they satisfy the compatibility conditions stated in Theorem~\ref{SectionI:Theorem:MainTheorem:GlobalExistence}. Then there exist $\delta_0,C>0$ such that if
  \begin{align}
    \label{SmallnessDelta}
    \begin{aligned}
      \delta
      & \coloneqq \varepsilon+\sum_{i=1}^{2}\left[ ||u_{0i}^{-}||_{L^1(-\infty,0)}+||u_{0i}^{+}||_{L^1(0,\infty)}+\sup_{x\in \mathbb{R}_*}\left\{ (|x|+1)^{3/2}|u_{0i}(x)| \right\} \right. \\
      & \phantom{\coloneqq \delta+\sum_{i=1}^{2}[} \quad +\left. \sup_{x>0}\left\{ (|x|+1)(|u_{0i}^{-}(-x)|+|u_{0i}^{+}(x)|) \right\} \right] \leq \delta_0,
    \end{aligned}
  \end{align}
  then
  \begin{equation}
    \label{SectionI:Theorem:MainTheorem:PointwiseEstimates:PointwiseEstimates}
    |v_i(x,t)|\leq C\delta \Psi_i(x,t) \quad (x\in \mathbb{R}_*,\, t>0)
  \end{equation}
  for $i=1,2$.
\end{theorem}

By Theorem~\ref{SectionI:Theorem:MainTheorem:PointwiseEstimates}, if either $m_1+m_V$ or $m_2+m_V$ is non-zero, then $||u(t)||_{L^{\infty}(\mathbb{R}_*)}\approx t^{-1/2}$ since $||\theta_j(t)||_{L^{\infty}(\mathbb{R}_*)}\approx t^{-1/2}$ for at least one of $j=1,2$ and $||\Psi_i(t)||_{L^{\infty}(\mathbb{R}_*)}=O(t^{-3/4})$. Thus in the original Eulerian coordinate, we have $||U(t)||_{L^{\infty}(\mathbb{R}\backslash \{ h(t) \})}\approx t^{-1/2}$ as stated in the introduction. For the velocity of the point mass $V(t)$, on the other hand, we have the following.

\begin{corollary}
  \label{Corollary:DecayEstimateV}
  Let $\tau_0,u_0 \in H^4(\mathbb{R}_*)$ and $V_0 \in \mathbb{R}$. Assume that they satisfy the compatibility conditions stated in Theorem~\ref{SectionI:Theorem:MainTheorem:GlobalExistence}. Then there exist $\delta_0,C>0$ such that if~\eqref{SmallnessDelta} holds, then
  \begin{equation}
    \label{V(t)_decay_rate}
    |V(t)|\leq C\delta (t+1)^{-3/2} \quad (t>0).
  \end{equation}
\end{corollary}

\begin{proof}
  This easily follows from~\eqref{SectionI:Theorem:MainTheorem:PointwiseEstimates:PointwiseEstimates} by noting that $V(t)=u(\pm 0,t)$, $|\theta_i(0,t)|\leq C\delta e^{-c^2 t/(2\nu)}$ and $|\Psi_i(0,t)|\leq C\delta (t+1)^{-3/2}$.
\end{proof}

Note that the parts of the solution of the order of $O(t^{-1/2})$ move with velocities $\pm c$; and since the point mass is situated at $x=0$ far away from the characteristics $x=\pm ct$, its velocity $V(t)$ decays faster than $||u(t)||_{L^{\infty}(\mathbb{R}_*)}$.

\begin{remark}
  We note that the smallness condition on $\delta$ in Theorem~\ref{SectionI:Theorem:MainTheorem:PointwiseEstimates} is satisfied if for some $\alpha>0$,
  \begin{equation}
    \delta_{\alpha}\coloneqq \varepsilon+\sum_{i=1}^{2}\sup_{x\in \mathbb{R}_*}(|x|+1)^{2+\alpha}|u_{0i}(x)|
  \end{equation}
  is sufficiently small.
\end{remark}

\begin{remark}
  The pointwise estimates~\eqref{SectionI:Theorem:MainTheorem:PointwiseEstimates:PointwiseEstimates} are almost identical to those for the solutions to the Cauchy problem~\cite[Theorem~2.6]{LZ97}; the effect of the point mass appears only as the added mass $m_V$ in the definition of $\theta_i$. Physically, this is due to the fact that waves entering the point mass are mostly transmitted to the other side and the reflected waves decay faster than the transmitted waves. Thus, the point mass is almost invisible to the fluid in the long run except for its contribution to the total momentum. We also note that the situation is completely different when we consider the motion of a pendulum (cf. Section~\ref{SectionI:Subsection:RelationOtherWorks}), where waves are mostly reflected and transmission is weak.
\end{remark}

\begin{remark}
  \label{Remark:Optimality}
  We can also show a lower bound $|V(t)|\geq C^{-1}\delta (t+1)^{-3/2}$ for some special initial data.\footnote{We cannot expect this to hold in general: For example, if $\tau_0$ is symmetric and $u_0$ is anti-symmetric around $x=0$ and $V_0=0$, then $V(t)=u(\pm 0,t)=0$ for all $t\geq 0$.}~One class of such initial data is those with $m_1+m_V \neq 0$ and $m_2+m_V=0$ (or the other way around) with sufficiently fast spatial decay. We do not present the proof in this paper since it is quite lengthy, but we outline here the basic strategy of the proof: Taking advantage of the fact that $\theta_2=0$ in this case, we can replace $\tilde{\psi}(x,t;-c)$ in $\Psi_1(x,t)$ by
  \begin{equation}
    \bar{\psi}(x,t;-c)=[|x+c(t+1)|^3+(t+1)^{5/2}]^{-1/2}.
  \end{equation}
  Then we can show that $|v_1(\pm 0,t)|\leq C\delta (t+1)^{-7/4}$ and that the only term in the integral representation of $v_2(\pm 0,t)$ (see~\eqref{SectionIII:viIntegralEquation} and~\eqref{SectionIII:NonlinearTerms}) of the order of $O(t^{-3/2})$ is
  \begin{equation}
    -\frac{1}{2}\int_{0}^{t}\int_{-\infty}^{\infty}g_{2}^{*}(-y,t-s)
    \begin{pmatrix}
      0 \\
      \theta_{1}^{2}
    \end{pmatrix}_x
    (y,s)\, dyds
  \end{equation}
  for which we can obtain the lower bound. The proof requires several additional elements together with those presented in this paper, and we wish to present this in a future publication, hopefully dealing with more general initial data. We also note that numerical experiments conducted by the author suggest that $|V(t)|\approx t^{-3/2}$ holds for a broader class of initial data.
\end{remark}

The outline of the proof is as follows. We prove Theorem~\ref{SectionI:Theorem:MainTheorem:GlobalExistence} in Section~\ref{SectionII}. Local-in-time solutions are constructed by an iteration scheme (Theorem~\ref{SectionII:LocalExistence:Theorem:MainTheorem}), and these solutions are extended to global-in-time solutions by proving global energy estimates (Theorem~\ref{SectionII:GlobalEnergyEstimate:Theorem:MainTheorem}). This part is rather standard except for the complication due to the presence of the point mass, and we basically follow the idea in~\cite{MN80,MN82,MN83}. We prove Theorem~\ref{SectionI:Theorem:MainTheorem:PointwiseEstimates} in Section~\ref{SectionIII}. We first derive an integral equation satisfied by the solution (Proposition~\ref{SectionIII:Proposition:IntegralEquation}). The main idea is to write the solution in terms of the fundamental solution in the Laplace transformed side, which is the idea developed and used in~\cite{Deng16,DWY15,DW18}. Then using the pointwise estimates of the fundamental solution~\cite{LZ97,Zeng94} and Green's functions describing transmission and reflection (see~\eqref{SectionIII:Bound_of_G-G*},~\eqref{SectionIII:GbIntByParts} and~\eqref{SectionIII:Bounds_of_Gb}), we prove the pointwise estimates in Section~\ref{SectionIII:PointwiseEstimates}.

\subsection{Relation to Other Works}
\label{SectionI:Subsection:RelationOtherWorks}
The original motivation of the present work is to rigorously understand the numerical results by Tsuji and Aoki on the motion of a pendulum (a point mass attached to a linear spring) in a rarefied gas~\cite{TA13,TA14a}. They showed numerically that, in the one-dimensional case, when the gas is described by the BGK model of the Boltzmann equation, the displacement $X(t)$ of the pendulum from its equilibrium position decays as $|X(t)|\approx t^{-3/2}$. Since the BGK model and the Boltzmann equation are closely related to the compressible Navier--Stokes equations via fluid dynamic limits~\cite{Sone02}, we decided to analyze the compressible Navier--Stokes equations case first. In fact, since the gas eventually approaches thermal equilibrium, it is quite natural to expect that the long-time behavior can be understood using macroscopic fluid dynamical equations. Note that we have considered the motion of a point mass and not of a pendulum; the analysis of the latter will be presented in another paper. We also expect that by using pointwise estimates of the fundamental solution of the Boltzmann equation~\cite{LY04,LY07}, it might become possible to treat the problem discussed in~\cite{TA13,TA14a} directly based on the Boltzmann equation. Lastly, we briefly mention that for simpler models in the kinetic theory of gases, such as the free-molecular or a special Lorentz gas, the problem of the long-time behavior of a moving object has been extensively studied. It was started by the work of Caprino, Marchioro and Pulvirenti~\cite{CMP06} and extended by various authors. See~\cite{IS19,Koike18b} and the references therein.

When the fluid domain is bounded and one-dimensional (hence a bounded interval), there are several studies on the coupled system of a compressible fluid and a moving body. Note that in the references below, they prefer to call the point mass the piston instead. Then the fluid--piston system is described by~\eqref{SectionI:CNSwithPointMass} (or related equations without the barotropic assumption) posed in a finite interval with appropriate boundary conditions. Shelukhin~\cite{Shelukhin77,Shelukhin78} considered the motion of a piston in a viscous compressible barotropic fluid with no inflow conditions at the ends of the fluid domain. He showed global existence and uniqueness of solutions and that the fluid and the piston cease their motion as time grows. The rate of convergence does not seem to be given, but it is likely that it is exponential since the fluid domain is bounded. He also considered the motion of an infinitely light piston (motion of a contact discontinuity) in a barotropic or a heat conducting viscous fluid and showed similar results~\cite{Shelukhin82,Shelukhin83}. Antman and Wilber~\cite{AW07} studied the springlike motion of a heavy piston in a viscous compressible barotropic fluid. Their emphasis is on the asymptotic expansion with respect to a small parameter $\varepsilon$ characterizing the ratio of the mass of the fluid to that of the piston (heavy piston regime). Maity, Takahashi and Tucsnak~\cite{MTT17} considered the motion of a piston in a viscous compressible polytropic fluid with non-vanishing fluid velocity at both or at least one of the ends of the fluid domain, and they proved global existence and uniqueness of solutions. Feireisl et al.~\cite{FMNT18} considered the motion of a piston in a viscous heat conducting ideal gas with no inflow and no heat flux conditions at the ends of the fluid domain and no heat flux condition also at the piston (adiabatic piston). They proved global existence and uniqueness of solutions and that the gas--piston system approaches to equilibrium (again the convergence rate does not seem to be given).

There are several results on the long-time behavior of a rigid body moving in a multi-dimensional unbounded fluid domain. We cite in particular two works: Munnier and Zuazua~\cite{MZ05} considered the motion of a ball in a fluid simply modeled by the $d$-dimensional heat equation in $\mathbb{R}^d$ ($d\geq 2$), which can be thought of as a natural extension of the result by V\'{a}zquez and Zuazua~\cite{VZ03}; Ervedoza, Hillairet and Lacave~\cite{EHL14} considered, on the other hand, the motion of a disk in an incompressible viscous fluid described by the two-dimensional Navier--Stokes equations in $\mathbb{R}^2$.

We also mention that the motion of a rigid body in a viscous compressible fluid, where both the rigid body and the fluid is contained in a bounded domain in $\mathbb{R}^3$, has been studied by several authors~\cite{BG09,DE2000,Feireisl03,HMTT17,HM15}. Similar problems are considered when the rigid body is replaced by an elastic structure~\cite{Boulakia05a,Boulakia05b}. Galdi, M\'{a}cha and Ne\v{c}asov\'{a}~\cite{GMN18} considered the motion of a rigid body with a cavity filled with a compressible fluid.

\section{Global Existence of Solutions}
\label{SectionII}
Let us prove Theorem~\ref{SectionI:Theorem:MainTheorem:GlobalExistence} in this section. Global-in-time solutions are constructed by extending local-in-time solutions by energy estimates. Most of the analysis is standard, but there are several complications due to the presence of the point mass.

\subsection{Local Existence Theorem}
The construction of local-in-time solutions is based on an iteration scheme as in~\cite{MN80,MN82,MN83}. We only give the main ingredients of the proof since it is rather standard and tedious.

\subsubsection{Notations}
Let us first introduce some notations. Fix $0<T<+\infty$. Assume that $-1<\inf_{x\in \mathbb{R}_*}\tau_0(x)$, and let $\tau_1 \coloneqq (-1+\inf_{x\in \mathbb{R}_*}\tau_0(x))/2$. Let
\begin{equation}
  X_{T}\coloneqq \{ \tau\in C([0,T];H^4(\mathbb{R}_*))\cap C^{1}([0,T];H^3(\mathbb{R}_*)) \mid \tau_1 \leq \inf_{(x,t)\in \mathbb{R}_* \times [0,T]}\tau(x,t) \}
\end{equation}
and
\begin{equation}
  Y_{T}\coloneqq \{ u\in C([0,T];H^4(\mathbb{R}_*))\cap C^{1}([0,T];H^2(\mathbb{R}_*)) \mid u_x \in L^2(0,T;H^4(\mathbb{R}_*)) \}.
\end{equation}
For $(\tau,u)\in X_{T}\times Y_{T}$, let
\begin{equation}
  L_{\tau}u\coloneqq u_t-\nu\left( \frac{u_x}{1+\tau} \right)_x,\quad f_{\tau}\coloneqq -p'(1+\tau)\tau_x.
\end{equation}
Using these notations, we can rewrite~\eqref{SectionI:CNSwithPointMass:Lagrangian:Perturbation} as
\begin{equation}
  \begin{dcases}
    \tau_t-u_x=0,                                                 & x\in \mathbb{R}_*,\, t>0, \\
    L_{\tau}u=f_\tau,                                             & x\in \mathbb{R}_*,\, t>0, \\
    u(\pm 0,t)=V(t),                                              & t>0, \\
    V'(t)=\llbracket -p(1+\tau)+\nu u_x/(1+\tau) \rrbracket(0,t), & t>0, \\
    V(0)=V_0;\, \tau(x,0)=\tau_0(x),\, u(x,0)=u_0(x),               & x\in \mathbb{R}_*.
  \end{dcases}
\end{equation}

In what follows, we use $C$ to denote a large positive constant, whose value may change from place to place.

\subsubsection{Iteration Scheme}
\label{SectionII:Iteration}
Let us describe the iteration scheme we use to construct local-in-time solutions. Let $\tau^{(0)}(x,t)\coloneqq \tau_0(x)$, and let $V^{(0)}\in C^2([0,T])$ with $V^{(0)}(0)=V_0$ and $(dV^{(0)}/dt)(0)=\llbracket -p(1+\tau_0)+\nu u_{0x}/(1+\tau_0) \rrbracket(0)$. Then let $u^{(0)}$ be the solution to:
\begin{equation}
  \begin{dcases}
    L_{\tau^{(0)}}u^{(0)}=f_{\tau^{(0)}}, & x\in \mathbb{R}_*,\, 0<t\leq T, \\
    u^{(0)}(\pm 0,t)=V^{(0)}(t),          & 0<t\leq T, \\
    u^{(0)}(x,0)=u_0(x),                  & x\in \mathbb{R}_*.
  \end{dcases}
\end{equation}
Since the compatibility conditions stated in Theorem~\ref{SectionI:Theorem:MainTheorem:GlobalExistence} are satisfied, this parabolic initial-boundary value problem possesses a unique classical solution $u^{(0)}$ and we have $(\tau^{(0)},u^{(0)},V^{(0)})\in X_{T}\times Y_{T}\times C^2([0,T])$ by~\cite[Chapter~4\, \S 6.4]{LM72}.

Suppose now that we are given $(\tau^{(n)},u^{(n)},V^{(n)})\in X_{T}\times Y_{T}\times C^2([0,T])$ satisfying
\begin{equation}
  \label{SectionII:Eq:n-CompatibilityConditions}
  \restr{\tau^{(n)}}{t=0}=\tau_0, \quad \restr{u^{(n)}}{t=0}=u_0, \quad V^{(n)}(0)=V_0, \quad \frac{dV^{(n)}}{dt}(0)=\llbracket -p(1+\tau_0)+\nu u_{0x}/(1+\tau_0) \rrbracket(0).
\end{equation}
Then we define
\begin{equation}
 (\tau^{(n+1)},u^{(n+1)},V^{(n+1)})\in X_{T}\times Y_{T}\times C^2([0,T]) 
\end{equation}
satisfying~\eqref{SectionII:Eq:n-CompatibilityConditions} with $n$ replaced by $n+1$ as follows. First, solve the following parabolic initial-boundary value problem to define $u^{(n+1)}$:
\begin{equation}
  \label{SectionII:LocalExistence:IterationScheme:u}
  \begin{dcases}
    L_{\tau^{(n)}}u^{(n+1)}=f_{\tau^{(n)}}, & x\in \mathbb{R}_*,\, 0<t\leq T, \\
    u^{(n+1)}(\pm 0,t)=V^{(n)}(t),          & 0<t\leq T, \\
    u^{(n+1)}(x,0)=u_0(x),                  & x\in \mathbb{R}_*.
  \end{dcases}
\end{equation}
Since the compatibility conditions stated in Theorem~\ref{SectionI:Theorem:MainTheorem:GlobalExistence} are satisfied,~\eqref{SectionII:LocalExistence:IterationScheme:u} possesses a unique classical solution $u^{(n+1)}\in Y_T$ by~\cite[Chapter~4\, \S 6.4]{LM72}. Next, let
\begin{equation}
  \label{SectionII:LocalExistence:IterationScheme:tau}
  \tau^{(n+1)}\coloneqq \tau_0+\int_{0}^{t}u^{(n)}_{x}\, ds
\end{equation}
and
\begin{equation}
  \label{SectionII:LocalExistence:IterationScheme:V}
  V^{(n+1)}\coloneqq V_0+\int_{0}^{t}\llbracket -p(1+\tau^{(n)})+\nu u^{(n)}_{x}/(1+\tau^{(n)}) \rrbracket(0,s)\, ds.
\end{equation}
In this way, we define a sequence of approximate solutions
\begin{equation}
  \{ (\tau^{(n)},u^{(n)},V^{(n)}) \}_{n=0}^{\infty} \subset X_{T}\times Y_{T}\times C^2([0,T]).
\end{equation}

\subsubsection{Convergence of the Sequence of Approximate Solutions}
We explain here how the convergence of the approximate solutions is proved.

To give energy estimates for the solution $u^{(n+1)}$ to~\eqref{SectionII:LocalExistence:IterationScheme:u}, we change the dependent variables as follows. Let $\phi \colon \mathbb{R}\to \mathbb{R}$ be a smooth compactly supported function satisfying $\phi(0)=1$. Then let $\bar{V}^{(n)}(x,t)\coloneqq V^{(n)}(t)\phi(x)\colon \mathbb{R}_* \times [0,T]\to \mathbb{R}$, $v^{(n+1)}\coloneqq u^{(n+1)}-\bar{V}^{(n)}$, $v_0\coloneqq u_0-V_0 \phi$ and $g_{\tau^{(n)},V^{(n)}}\coloneqq f_{\tau^{(n)}}-L_{\tau^{(n)}}\bar{V}^{(n)}$. By~\eqref{SectionII:LocalExistence:IterationScheme:u}, $v^{(n+1)}$ satisfies
\begin{equation}
  \label{SectionII:LocalExistence:Def_of_v}
  \begin{dcases}
    L_{\tau^{(n)}}v^{(n+1)}=g_{\tau^{(n)},V^{(n)}}, & x\in \mathbb{R}_*,\, 0<t\leq T, \\
    v^{(n+1)}(\pm 0,t)=0,                           & 0<t\leq T, \\
    v^{(n+1)}(x,0)=v_0(x),                          & x\in \mathbb{R}_*.
  \end{dcases}
\end{equation}
This is now a homogeneous initial-boundary value problem, and we can obtain energy estimates in a standard way (cf. \cite[Proposition~3.3]{MN80}).

\begin{proposition}
  \label{SectionII:LocalExistence:Proposition:Energy_Estimate_of_v}
  Let $\tau \in X_{T}$, $g\in C([0,T];H^3(\mathbb{R}_*))$ and $v_0 \in H^4(\mathbb{R}_*)$, and suppose that $v\in Y_T$ satisfies
  \begin{equation}
    \begin{dcases}
      L_{\tau}v=g,   & x\in \mathbb{R}_*,\, 0<t\leq T, \\
      v(\pm 0,t)=0,  & 0<t\leq T, \\
      v(x,0)=v_0(x), & x\in \mathbb{R}_*.
    \end{dcases}
  \end{equation} 
   If $\sup_{0\leq t\leq T}||\tau(t)||_4 \leq E$ for some positive constant $E$, then there exist positive constants $c$, $C$ and $\bar{C}$ depending only on $\nu$, $\tau_1$ and $E$ such that $v$ satisfies the following estimate:
  \begin{equation}
    ||v(t)||_{4}^{2}+c\int_{0}^{t}||v_x(s)||_{4}^{2}\, ds\leq e^{\bar{C}t}\left( ||v_0||_{4}^{2}+C\int_{0}^{t}||g(s)||_{3}^{2}\, ds \right)
  \end{equation}
  for $0\leq t\leq T$.
\end{proposition}

The following proposition is needed to give estimates for $\tau^{(n+1)}$ defined by~\eqref{SectionII:LocalExistence:IterationScheme:tau}. We write $||\cdot ||$ to mean $||\cdot ||_0$ in the following.

\begin{proposition}
  \label{SectionII:LocalExistence:Proposition:EstimatesApproximateSolutions:tau}
  Let $\tau_0 \in H^4(\mathbb{R}_*)$ and $u\in Y_{T}$, and assume that $-1<\inf_{x\in \mathbb{R}_*}\tau_0(x)$. If $T$ satisfies
  \begin{equation}
    \label{SectionII:LocalExistence:EstimatesApproximateSolutions:tau:Assumption_on_T}
    2T^{1/2}\left( \int_{0}^{T}||u_x(s)||^2\, ds \right)^{1/2}\leq 1+\inf_{x\in \mathbb{R}_*}\tau_0(x),
  \end{equation}
  then $\tau$ defined by
  \begin{equation}
    \label{SectionII:LocalExistence:EstimatesApproximateSolutions:tau:Def_of_tau}
    \tau \coloneqq \tau_0+\int_{0}^{t}u_x\, ds
  \end{equation}
  satisfies $\tau \in X_{T}$ and
  \begin{equation}
    \label{SectionII:LocalExistence:EstimatesApproximateSolutions:tau:Estimate_of_tau}
    ||\partial_{x}^{k}\tau(t)||\leq ||\partial_{x}^{k}\tau_0||+t^{1/2}\left( \int_{0}^{t}||\partial_{x}^{k+1}u(s)||^2\, ds \right)^{1/2} \quad (0\leq k\leq 4)
  \end{equation}
  for $0\leq t\leq T$.
\end{proposition}

\begin{proof}
  That $\tau \in C([0,T];H^4(\mathbb{R}_*))\cap C^1([0,T];H^3(\mathbb{R}_*))$ follows from $u\in Y_{T}$. The estimates~\eqref{SectionII:LocalExistence:EstimatesApproximateSolutions:tau:Estimate_of_tau} follow by applying the Cauchy--Schwarz inequality to~\eqref{SectionII:LocalExistence:EstimatesApproximateSolutions:tau:Def_of_tau} differentiated $k$ times. Now, by~\eqref{SectionII:LocalExistence:EstimatesApproximateSolutions:tau:Assumption_on_T} and~\eqref{SectionII:LocalExistence:EstimatesApproximateSolutions:tau:Def_of_tau}, we have
  \begin{equation}
    \inf_{(x,t)\in \mathbb{R}_* \times [0,T]}\tau(x,t)\geq \inf_{x\in \mathbb{R}_*}\tau_0(x)-T^{1/2}\left( \int_{0}^{T}||u_x(s)||^2\, ds \right)^{1/2}\geq \tau_1.
  \end{equation}
  Thus $\tau \in X_{T}$.
\end{proof}

The following proposition is needed to give estimates for $V^{(n+1)}$ defined by~\eqref{SectionII:LocalExistence:IterationScheme:V}.

\begin{proposition}
  \label{SectionII:LocalExistence:Proposition:EstimatesApproximateSolutions:V}
  Let $(\tau,u)\in X_{T}\times Y_{T}$. If $\sup_{0\leq t\leq T}||\tau(t)||\leq E$ for some positive constant $E$, then $V$ defined by
  \begin{equation}
    V\coloneqq V_0+\int_{0}^{t}\llbracket -p(1+\tau)+\nu u_x/(1+\tau) \rrbracket(0,s)\, ds
  \end{equation}
  satisfies $V\in C^2([0,T])$ and
  \begin{align}
    |V(t)| & \leq |V_0|+Ct\sup_{0\leq t\leq T}||\tau(t)||_1+\frac{2\nu}{1+\tau_1}t^{1/2}\left( \int_{0}^{t}||u_x(s)||_{1}^{2}\, ds \right)^{1/2}, \\
    |V'(t)| & \leq C||\tau(t)||_1+\frac{2\nu}{1+\tau_1}||u_x(t)||_1
  \end{align}
  for $0\leq t\leq T$, where $C$ is a positive constant depending only on $\tau_1$ and $E$.
\end{proposition}

\begin{proof}
  It should be sufficient to note that
  \begin{equation}
    \llbracket -p(1+\tau) \rrbracket(0,t)=\llbracket -p(1+\tau)+p(1) \rrbracket(0,t).
  \end{equation}
  Using this, we can obtain the desired bounds.
\end{proof}

The following proposition follows easily from Propositions~\ref{SectionII:LocalExistence:Proposition:Energy_Estimate_of_v}--\ref{SectionII:LocalExistence:Proposition:EstimatesApproximateSolutions:V}.

\begin{proposition}
  \label{SectionII:LocalExistence:Proposition:UniformBoundedness}
  Suppose that $\tau_0,u_0 \in H^4(\mathbb{R}_*)$ and $V_0 \in \mathbb{R}$ satisfy the compatibility conditions stated in Theorem~\ref{SectionI:Theorem:MainTheorem:GlobalExistence}. Assume that $-1<\inf_{x\in \mathbb{R}_*}\tau_0(x)$. Define the sequence of approximate solutions $\{ (\tau^{(n)},u^{(n)},V^{(n)}) \}_{n=0}^{\infty}$ as described in Section~\ref{SectionII:Iteration}. Let $E\coloneqq 2(||\tau_0||_4+||u_0||_4)$. Then if $T$ is sufficiently small (depending only on $\nu$, $\tau_1$ and $E$), the approximate solutions satisfy:
  \begin{equation}
    \label{SectionII:LocalExistence:Proposition:UniformBoundedness:UniformBoundedness}
    ||\tau^{(n)}(t)||_4, \quad ||u^{(n)}(t)||_4, \quad \left( \int_{0}^{t}||u_{x}^{(n)}(s)||_{4}^{2}\, ds \right)^{1/2}, \quad |V^{(n)}(t)|\leq E
  \end{equation}
  for $0\leq t\leq T$.
\end{proposition}

Using this uniform boundedness, the convergence of the approximate solutions can be proved as in~\cite{MN80}. We omit the proof.

\begin{theorem}
  \label{SectionII:LocalExistence:Theorem:MainTheorem}
  Suppose that $\tau_0,u_0 \in H^4(\mathbb{R}_*)$ and $V_0 \in \mathbb{R}$ satisfy the compatibility conditions stated in Theorem~\ref{SectionI:Theorem:MainTheorem:GlobalExistence}. Assume that $-1<\inf_{x\in \mathbb{R}_*}\tau_0(x)$. Define the sequence of approximate solutions $\{ (\tau^{(n)},u^{(n)},V^{(n)}) \}_{n=0}^{\infty}$ as described in Section~\ref{SectionII:Iteration}. Let $E\coloneqq 2(||\tau_0||_4+||u_0||_4)$. Then if $T$ is sufficiently small (depending only on $\nu$, $\tau_1$ and $E$), there exists $(\tau,u,V)\in X_T \times Y_T \times C^2([0,T])$ such that
  \begin{equation}
    (\tau^{(n)},u^{(n)},V^{(n)})\to (\tau,u,V) \quad \text{as $n\to \infty$}
  \end{equation}
  in the space
  \begin{equation}
    C([0,T];H^3(\mathbb{R}_*))\times C([0,T];H^3(\mathbb{R}_*))\times C^1([0,T]).
  \end{equation}
  The limit $(\tau,u,V)$ is the unique classical solution to~\eqref{SectionI:CNSwithPointMass:Lagrangian:Perturbation} and satisfies
  \begin{equation}
    \label{SectionII:LocalExistence:Theorem:MainTheorem:UniformBounds}
    ||\tau(t)||_4, \quad ||u(t)||_4, \quad \left( \int_{0}^{t}||u_x(s)||_{4}^{2}\, ds \right)^{1/2}, \quad |V(t)|\leq E
  \end{equation}
  for $0\leq t\leq T$.
\end{theorem}

\subsection{Global Existence Theorem}
To extend the local-in-time solutions obtained in the previous section, we next derive global-in-time energy estimates. The first energy estimate is the conservation of physical energy.

\begin{proposition}
  \label{SectionII:GlobalEnergyEstimate:Proposition:EnergyConservation}
  Let $(\tau,u,V)\in X_{T}\times Y_{T}\times C^2([0,T])$ be the solution to~\eqref{SectionI:CNSwithPointMass:Lagrangian:Perturbation}. Let
  \begin{equation}
    P(\tau)\coloneqq  -\int_{0}^{\tau}(p(1+\eta)-p(1))\, d\eta.
  \end{equation}
  Then we have
  \begin{align}
    \label{SectionII:GlobalEnergyEstimate:Proposition:EnergyConservation:EnergyConservation}
    \begin{aligned}
      & 2\int_{\mathbb{R}_*}P(\tau)\, dx+||u(t)||^2+|V(t)|^2+2\nu \int_{0}^{t}\int_{\mathbb{R}_*}\frac{u_{x}^{2}}{1+\tau}\, dxds \\
      & =2\int_{\mathbb{R}_*}P(\tau_0)\, dx+||u_0||^2+|V_0|^2.
    \end{aligned}
  \end{align}
  Moreover, there exist $\varepsilon_0>0$ and $C\geq 1$ such that if  $\sup_{0\leq t\leq T}|\tau(t)|_{\infty}\leq \varepsilon_0$,\footnote{We write $|\cdot |_{\infty}$ to mean $||\cdot ||_{L^{\infty}(\mathbb{R}_*)}$.}~then
  \begin{equation}
    \label{SectionII:GlobalEnergyEstimate:Proposition:EnergyConservation:P_is_Equivalent_to_tauL2}
    C^{-1}||\tau(t)||^2 \leq \int_{\mathbb{R}_*}P(\tau)\, dx\leq C||\tau(t)||^2.
  \end{equation}
\end{proposition}

\begin{proof}
  First, multiply $\tau_t-u_x=0$ and $L_{\tau}u-f_{\tau}=0$ by $-p(1+\tau)+p(1)$ and $u$, respectively, and integrate the resulting equations with respect to $x$; then multiply $V'(t)-\llbracket -p(1+\tau)+\nu u_x/(1+\tau) \rrbracket(0,t)=0$ by $V(t)$. Adding the obtained equations and applying integration by parts, we obtain
  \begin{equation}
    \label{SectionII:GlobalEnergyEstimate:Proposition:EnergyConservation:Proof:DifferentialForm}
    \frac{d}{dt}\left( \int_{\mathbb{R}_*}P(\tau)\, dx+\frac{1}{2}||u(t)||^2+\frac{1}{2}|V(t)|^2 \right) +\nu \int_{\mathbb{R}_*}\frac{u_{x}^{2}}{1+\tau}\, dx=0.
  \end{equation}
  Note that we used $u(\pm 0,t)=V(t)$ here. Integrating~\eqref{SectionII:GlobalEnergyEstimate:Proposition:EnergyConservation:Proof:DifferentialForm} with respect to $t$ gives~\eqref{SectionII:GlobalEnergyEstimate:Proposition:EnergyConservation:EnergyConservation}. By taking the Taylor expansion,
  \begin{equation}
    P(\tau)=-\frac{1}{2}p'(1)\tau^2+O(\tau^3) \quad \text{as $\tau \to 0$},
  \end{equation}
  and~\eqref{SectionII:GlobalEnergyEstimate:Proposition:EnergyConservation:P_is_Equivalent_to_tauL2} follows from this (note that $p'(1)<0$).
\end{proof}

We continue to prove higher order energy estimates in the following.

\begin{proposition}
  \label{SectionII:GlobalEnergyEstimate:Proposition:Energy_Estimate_of_ut}
  Let $(\tau,u,V)\in X_{T}\times Y_{T}\times C^2([0,T])$ be the solution to~\eqref{SectionI:CNSwithPointMass:Lagrangian:Perturbation}. Then there exist $\varepsilon_0>0$ and $C\geq 1$ such that if
  \begin{equation}
    \label{AposterioriSmallness}
    \sup_{0\leq t\leq T}|\tau(t)|_{\infty}+\sup_{0\leq t\leq T}|u_x(t)|_{\infty}\leq \varepsilon_0,
  \end{equation}
   then
  \begin{equation}
    \label{SectionII:GlobalEnergyEstimate:Proposition:Energy_Estimate_of_ut:Energy_Estimate_of_ut}
    ||u_x(t)||^2+||u_t(t)||^2+|V'(t)|^2+C^{-1}\int_{0}^{t}||u_{xt}(s)||^2 \, ds\leq C\left( ||\tau_0||_{1}^{2}+||u_0||_{2}^{2} \right).
  \end{equation}
\end{proposition}

\begin{proof}
  First, differentiate $\tau_t-u_x=0$, $L_{\tau}u-f_{\tau}=0$ and $V'(t)-\llbracket -p(1+\tau)+\nu u_x/(1+\tau) \rrbracket(0,t)=0$ with respect to $t$, then multiply the resulting equations by $-p(1+\tau)_t$, $u_t$ and $V'(t)$, respectively. Adding and integrating these equations, we obtain
  \begin{align}
    \label{SectionII:GlobalEnergyEstimate:Proposition:Energy_Estimate_of_ut:Proof:AddedIntegratedEq}
    \begin{aligned}
      & \frac{1}{2}\left( ||u_t(t)||^2+|V'(t)|^2 \right)-\int_{0}^{t}\int_{\mathbb{R}_*}p(1+\tau)_t \tau_{tt}\, dxds+\nu \int_{0}^{t}\int_{\mathbb{R}_*}\frac{u_{xt}^{2}}{1+\tau}\, dxds \\
      & =\frac{1}{2}\left( ||u_t(0)||^2+|V'(0)|^2 \right)+\nu \int_{0}^{t}\int_{\mathbb{R}_*}\frac{\tau_t u_x}{(1+\tau)^2}u_{xt}\, dxds.
    \end{aligned}
  \end{align}
  Next, note that
  \begin{equation}
    \label{SectionII:GlobalEnergyEstimate:Proposition:Energy_Estimate_of_ut:Proof:a(b_t)=(ab)_t-(a_t)b}
    \int_{\mathbb{R}_*}p(1+\tau)_t \tau_{tt}\, dx=\frac{1}{2}\frac{d}{dt}\int_{\mathbb{R}_*}p'(1+\tau)\tau_{t}^{2}-\frac{1}{2}\int_{\mathbb{R}_*}p''(1+\tau)\tau_{t}^{3}\, dx.
  \end{equation}
  By~\eqref{SectionII:GlobalEnergyEstimate:Proposition:Energy_Estimate_of_ut:Proof:AddedIntegratedEq} and~\eqref{SectionII:GlobalEnergyEstimate:Proposition:Energy_Estimate_of_ut:Proof:a(b_t)=(ab)_t-(a_t)b}, we obtain
  \begin{align}
    & \frac{1}{2}\left( -\int_{\mathbb{R}_*}p'(1+\tau)u_{x}^{2}\, dx+||u_t(t)||^2+|V'(t)|^2 \right)+\nu \int_{0}^{t}\int_{\mathbb{R}_*}\frac{u_{xt}^{2}}{1+\tau}\, dxds \\
    & =\frac{1}{2}\left( -\int_{\mathbb{R}_*}p'(1+\tau_0)u_{0x}^{2}\, dx+||u_t(0)||^2+|V'(0)|^2 \right) \\
    & \quad +\nu \int_{0}^{t}\int_{\mathbb{R}_*}\frac{u_{x}^{2}}{(1+\tau)^2}u_{xt}\, dxds-\frac{1}{2}\int_{0}^{t}\int_{\mathbb{R}_*}p''(1+\tau)u_{x}^{3}\, dxds,
  \end{align}
  from which~\eqref{SectionII:GlobalEnergyEstimate:Proposition:Energy_Estimate_of_ut:Energy_Estimate_of_ut} follows by using~\eqref{SectionI:CNSwithPointMass:Lagrangian:Perturbation} and Proposition~\ref{SectionII:GlobalEnergyEstimate:Proposition:EnergyConservation}.
\end{proof}

\begin{proposition}
  \label{SectionII:GlobalEnergyEstimate:Proposition:Energy_Estimate_of_ux}
  Let $(\tau,u,V)\in X_{T}\times Y_{T}\times C^2([0,T])$ be the solution to~\eqref{SectionI:CNSwithPointMass:Lagrangian:Perturbation}. Then there exist $\varepsilon_0>0$ and $C\geq 1$ such that if~\eqref{AposterioriSmallness} holds, then
  \begin{equation}
    \label{SectionII:GlobalEnergyEstimate:Proposition:Energy_Estimate_of_ux:Energy_Estimate_of_ux}
    ||u_x(t)||^2+C^{-1}\int_{0}^{t}\left( ||u_t(s)||^2+|V'(s)|^2 \right)\, ds\leq C\left( ||\tau_0||_{1}^{2}+||u_0||_{2}^{2} \right).
  \end{equation}
\end{proposition}

\begin{proof}
  First, multiply $L_{\tau}u-f_{\tau}=0$ by $u_t$, and integrate the resulting equation with respect to $x$ and $t$; then multiply $V'(t)-\llbracket -p(1+\tau)+\nu u_x/(1+\tau) \rrbracket(0,t)=0$ by $V'(t)$, and integrate the resulting equation with respect to $t$. Adding the obtained equations and applying integration by parts, we obtain
  \begin{align}
    & \frac{\nu}{2}||u_x(t)||^2+\int_{0}^{t}\left( ||u_t(s)||^2+|V'(s)|^2 \right)\, ds \\
    & =\frac{\nu}{2}||u_{0x}||^2+\int_{0}^{t}\int_{\mathbb{R}_*}(p(1+\tau)-p(1))u_{xt}\, dxds+\nu \int_{0}^{t}\int_{\mathbb{R}_*}\frac{\tau}{1+\tau}u_x u_{xt}\, dxds.
  \end{align}
  Next, note that
  \begin{align}
    \int_{\mathbb{R}_*}(p(1+\tau)-p(1))u_{xt}\, dx
    & =\frac{d}{dt}\int_{\mathbb{R}_*}(p(1+\tau)-p(1))u_x \, dx-\int_{\mathbb{R}_*}p'(1+\tau)\tau_t u_x \, dx \\
    & =\frac{d}{dt}\int_{\mathbb{R}_*}(p(1+\tau)-p(1))u_x \, dx-\int_{\mathbb{R}_*}p'(1+\tau)u_{x}^{2} \, dx.
  \end{align}
  By the Taylor expansion, we have
  \begin{equation}
    \left| \int_{\mathbb{R}_*}(p(1+\tau)-p(1))u_x \, dx \right| \leq C\int_{\mathbb{R}_*}|\tau u_x|\, dx\leq C||\tau(t)||^2+\frac{\nu}{4}||u_x(t)||^2.
  \end{equation}
  Combining these calculations and using Propositions~\ref{SectionII:GlobalEnergyEstimate:Proposition:EnergyConservation} and~\ref{SectionII:GlobalEnergyEstimate:Proposition:Energy_Estimate_of_ut}, we obtain~\eqref{SectionII:GlobalEnergyEstimate:Proposition:Energy_Estimate_of_ux:Energy_Estimate_of_ux}.
\end{proof}

\begin{proposition}
  \label{SectionII:GlobalEnergyEstimate:Proposition:Energy_Estimate_of_taux}
  Let $(\tau,u,V)\in X_{T}\times Y_{T}\times C^2([0,T])$ be the solution to~\eqref{SectionI:CNSwithPointMass:Lagrangian:Perturbation}. Then there exist $\varepsilon_0>0$ and $C\geq 1$ such that if~\eqref{AposterioriSmallness} holds, then
  \begin{equation}
    \label{SectionII:GlobalEnergyEstimate:Proposition:Energy_Estimate_of_taux:Energy_Estimate_of_taux}
    ||\tau_x(t)||^2+C^{-1}\int_{0}^{t}||\tau_x(s)||^2 \, ds\leq C\left( ||\tau_0||_{1}^{2}+||u_0||_{2}^{2} \right).
  \end{equation}
\end{proposition}

\begin{proof}
  First, note that by $L_{\tau}u=f_{\tau}$, we have
  \begin{equation}
    \label{SectionII:GlobalEnergyEstimate:Proposition:Energy_Estimate_of_taux:Proof:uxx}
    u_{xx}=\frac{1+\tau}{\nu}p'(1+\tau)\tau_x+\frac{\tau_x}{1+\tau}u_x+\frac{1+\tau}{\nu}u_t.
  \end{equation}
  Then differentiating $\tau_t-u_x=0$ with respect to $x$ and using~\eqref{SectionII:GlobalEnergyEstimate:Proposition:Energy_Estimate_of_taux:Proof:uxx}, we obtain
  \begin{equation}
    \label{SectionII:GlobalEnergyEstimate:Proposition:Energy_Estimate_of_taux:Proof:tauxt}
    \tau_{xt}=\frac{1+\tau}{\nu}p'(1+\tau)\tau_x+\frac{\tau_x}{1+\tau}u_x+\frac{1+\tau}{\nu}u_t.
  \end{equation}
  Multiplying this equation by $\tau_x$ and integrating with respect to $x$ and $t$, we obtain
  \begin{align}
    \label{SectionII:GlobalEnergyEstimate:Proposition:Energy_Estimate_of_taux:Proof:ApplicantCauchySchwarz}
    \begin{aligned}
      & \frac{1}{2}||\tau_x(t)||^2-\frac{1}{\nu}\int_{0}^{t}\int_{\mathbb{R}_*}(1+\tau)p'(1+\tau)\tau_{x}^{2}\, dxds \\
      & =\frac{1}{2}||\tau_{0x}||^2+\int_{0}^{t}\int_{\mathbb{R}_*}\frac{\tau_{x}^{2}}{1+\tau}u_x \, dxds+\frac{1}{\nu}\int_{0}^{t}\int_{\mathbb{R}_*}(1+\tau)\tau_x u_t \, dxds.
    \end{aligned}
  \end{align}
  Then~\eqref{SectionII:GlobalEnergyEstimate:Proposition:Energy_Estimate_of_taux:Energy_Estimate_of_taux} is obtained by applying the Cauchy--Schwarz inequality to~\eqref{SectionII:GlobalEnergyEstimate:Proposition:Energy_Estimate_of_taux:Proof:ApplicantCauchySchwarz} and using Proposition~\ref{SectionII:GlobalEnergyEstimate:Proposition:Energy_Estimate_of_ux}.
\end{proof}

\begin{proposition}
  \label{SectionII:GlobalEnergyEstimate:Proposition:Energy_Estimate_of_tauxx}
  Let $(\tau,u,V)\in X_{T}\times Y_{T}\times C^2([0,T])$ be the solution to~\eqref{SectionI:CNSwithPointMass:Lagrangian:Perturbation}. Then there exist $\varepsilon_0>0$ and $C\geq 1$ such that if~\eqref{AposterioriSmallness} holds, then
  \begin{equation}
    \label{SectionII:GlobalEnergyEstimate:Proposition:Energy_Estimate_of_tauxx:Energy_Estimate_of_tauxx}
    ||\tau_{xx}(t)||^2+C^{-1}\int_{0}^{t}||\tau_{xx}(s)||^2 \, ds\leq C\left( ||\tau_0||_{2}^{2}+||u_0||_{2}^{2} \right).
  \end{equation}
\end{proposition}

\begin{proof}
  First, differentiate~\eqref{SectionII:GlobalEnergyEstimate:Proposition:Energy_Estimate_of_taux:Proof:tauxt} with respect to $x$, and multiply the resulting equation by $\tau_{xx}$; then integrating the obtained equation with respect to $x$ and $t$, we have
  \begin{align}
    & \frac{1}{2}||\tau_{xx}(t)||^2-\frac{1}{\nu}\int_{0}^{t}\int_{\mathbb{R}_*}(1+\tau)p'(1+\tau)\tau_{xx}^{2}\, dxds \\
    & =\frac{1}{2}||\tau_{0xx}||^2+\frac{1}{\nu}\int_{0}^{t}\int_{\mathbb{R}_*} \left\{ p'(1+\tau)+(1+\tau)p''(1+\tau) \right\} \tau_{x}^{2}\tau_{xx}\, dxds \\
    & \quad +\int_{0}^{t}\int_{\mathbb{R}_*}\frac{\tau_{xx}^{2}}{1+\tau}u_x \, dxds-\int_{0}^{t}\int_{\mathbb{R}_*}\frac{\tau_{x}^{2}}{(1+\tau)^2}\tau_{xx}u_x \, dxds \\
    & \quad +\int_{0}^{t}\int_{\mathbb{R}_*}\frac{\tau_x}{1+\tau}\tau_{xx}u_{xx}\, dxds+\frac{1}{\nu}\int_{0}^{t}\int_{\mathbb{R}_*}\tau_x \tau_{xx}u_t \, dxds \\
    & \quad +\frac{1}{\nu}\int_{0}^{t}\int_{\mathbb{R}_*}(1+\tau)\tau_{xx}u_{xt}\, dxds,
  \end{align}
  from which~\eqref{SectionII:GlobalEnergyEstimate:Proposition:Energy_Estimate_of_tauxx:Energy_Estimate_of_tauxx} follows by using~\eqref{SectionII:GlobalEnergyEstimate:Proposition:Energy_Estimate_of_taux:Proof:uxx} and Propositions~\ref{SectionII:GlobalEnergyEstimate:Proposition:Energy_Estimate_of_ut}--\ref{SectionII:GlobalEnergyEstimate:Proposition:Energy_Estimate_of_taux}.
\end{proof}

\begin{proposition}
  \label{SectionII:GlobalEnergyEstimate:Proposition:Energy_Estimate_of_utt}
  Let $(\tau,u,V)\in X_{T}\times Y_{T}\times C^2([0,T])$ be the solution to~\eqref{SectionI:CNSwithPointMass:Lagrangian:Perturbation}. Then there exist $\varepsilon_0>0$ and $C\geq 1$ such that if~\eqref{AposterioriSmallness} holds, then
  \begin{equation}
    \label{SectionII:GlobalEnergyEstimate:Proposition:Energy_Estimate_of_utt:Energy_Estimate_of_utt}
    ||u_{xt}(t)||^2+||u_{tt}(t)||^2+|V''(t)|^2+C^{-1}\int_{0}^{t}||u_{xtt}(s)||^2 \, ds\leq C\left( ||\tau_0||_{3}^{2}+||u_0||_{4}^{2} \right).
  \end{equation}
\end{proposition}

\begin{proof}
  First, differentiate $\tau_t-u_x=0$ and $L_{\tau}u-f_{\tau}=0$ twice with respect to $t$, and multiply the resulting equations by $-p(1+\tau)_{tt}$ and $u_{tt}$, respectively; then integrate the obtained equations with respect to $x$ and $t$. Next, differentiate $V'(t)-\llbracket -p(1+\tau)+\nu u_x/(1+\tau) \rrbracket(0,t)=0$ twice with respect to $t$, and multiply the resulting equation by $V''(t)$ and integrate the obtained equation with respect to $t$. Now, add the resulting equations to obtain
  \begin{align}
    \label{SectionII:GlobalEnergyEstimate:Proposition:Energy_Estimate_of_utt:Proof:AddedEq}
    \begin{aligned}
      & \frac{1}{2}\left( ||u_{tt}(t)||^2+|V''(t)|^2 \right) -\int_{0}^{t}\int_{\mathbb{R}_*}p(1+\tau)_{tt}\tau_{ttt}\, dyds+\nu \int_{0}^{t}\int_{\mathbb{R}_*}\frac{u_{xtt}^{2}}{1+\tau}\, dxds \\
      & =\frac{1}{2}\left( ||u_{tt}(0)||^2+|V''(0)|^2 \right) \\
      & \quad -2\nu \int_{0}^{t}\int_{\mathbb{R}_*}\frac{u_{x}^{3}}{(1+\tau)^3}u_{xtt}\, dxds+3\nu \int_{0}^{t}\int_{\mathbb{R}_*}\frac{u_x}{(1+\tau)^2}u_{xt}u_{xtt}\, dxds.
    \end{aligned}
  \end{align}
  Note that
  \begin{equation}
    \int_{\mathbb{R}_*}p(1+\tau)_{tt}\tau_{ttt}\, dx=\frac{1}{2}\frac{d}{dt}\int_{\mathbb{R}_*}p'(1+\tau)\tau_{tt}^{2}\, dx-\frac{1}{2}\int_{\mathbb{R}_*}p''(1+\tau)\tau_t \tau_{tt}^{2}\, dx+\int_{\mathbb{R}_*}p''(1+\tau)\tau_{tt}\tau_{ttt}\, dx.
  \end{equation}
  Now~\eqref{SectionII:GlobalEnergyEstimate:Proposition:Energy_Estimate_of_utt:Energy_Estimate_of_utt} follows from~\eqref{SectionII:GlobalEnergyEstimate:Proposition:Energy_Estimate_of_utt:Proof:AddedEq} by using~\eqref{SectionI:CNSwithPointMass:Lagrangian:Perturbation}, Propositions~\ref{SectionII:GlobalEnergyEstimate:Proposition:EnergyConservation} and \ref{SectionII:GlobalEnergyEstimate:Proposition:Energy_Estimate_of_ut}.
\end{proof}

\begin{proposition}
  \label{SectionII:GlobalEnergyEstimate:Proposition:Energy_Estimate_of_uxxx}
  Let $(\tau,u,V)\in X_{T}\times Y_{T}\times C^2([0,T])$ be the solution to~\eqref{SectionI:CNSwithPointMass:Lagrangian:Perturbation}. Then there exist $\varepsilon_0>0$ and $C\geq 1$ such that if~\eqref{AposterioriSmallness} holds, then
  \begin{equation}
    \label{SectionII:GlobalEnergyEstimate:Proposition:Energy_Estimate_of_uxxx:Energy_Estimate_of_uxxx}
    ||u_{xt}(t)||^2+C^{-1}\int_{0}^{t}\left( ||u_{tt}(s)||^2+|V''(s)|^2 \right) \, ds\leq C\left( ||\tau_0||_{3}^{2}+||u_0||_{4}^{2} \right).
  \end{equation}
\end{proposition}

\begin{proof}
  First, differentiate $L_{\tau}u-f_{\tau}=0$ with respect to $t$, and multiply the resulting equation by $u_{tt}$; then integrate the obtained equation with respect to $x$ and $t$. Next, differentiate $V'(t)-\llbracket -p(1+\tau)+\nu u_x/(1+\tau) \rrbracket(0,t)=0$ with respect to $t$, and multiply the resulting equation by $V''(t)$ and then integrate the obtained equation with respect to $t$. Now, add the resulting two equations to obtain
  \begin{align}
    & \frac{\nu}{2}||u_{xt}(t)||^2+\int_{0}^{t}\left( ||u_{tt}(s)||^2+|V''(s)|^2 \right) \, ds \\
    & =\frac{\nu}{2}||u_{xt}(0)||^2+\nu \int_{0}^{t}\int_{\mathbb{R}_*}\frac{\tau}{1+\tau}u_{xt}u_{xtt}\, dxds \\
    & \quad +\nu \int_{0}^{t}\int_{\mathbb{R}_*}\frac{u_{x}^{2}}{(1+\tau)^2}u_{xtt}\, dxds+\int_{0}^{t}\int_{\mathbb{R}_*}p'(1+\tau)u_x u_{xtt}\, dxds,
  \end{align}
  from which~\eqref{SectionII:GlobalEnergyEstimate:Proposition:Energy_Estimate_of_uxxx:Energy_Estimate_of_uxxx} follows by using~\eqref{SectionI:CNSwithPointMass:Lagrangian:Perturbation}, Propositions~\ref{SectionII:GlobalEnergyEstimate:Proposition:EnergyConservation}, \ref{SectionII:GlobalEnergyEstimate:Proposition:Energy_Estimate_of_ut} and~\ref{SectionII:GlobalEnergyEstimate:Proposition:Energy_Estimate_of_utt}.
\end{proof}

\begin{proposition}
  \label{SectionII:GlobalEnergyEstimate:Proposition:Energy_Estimate_of_tauxxx}
  Let $(\tau,u,V)\in X_{T}\times Y_{T}\times C^2([0,T])$ be the solution to~\eqref{SectionI:CNSwithPointMass:Lagrangian:Perturbation}. Then there exist $\varepsilon_0>0$ and $C\geq 1$ such that if~\eqref{AposterioriSmallness} holds, then
  \begin{equation}
    \label{SectionII:GlobalEnergyEstimate:Proposition:Energy_Estimate_of_tauxxx:Energy_Estimate_of_tauxxx}
    ||\tau_{xxx}(t)||^2+C^{-1}\int_{0}^{t}||\tau_{xxx}(s)||^2 \, ds\leq C\left( ||\tau_0||_{3}^{2}+||u_0||_{4}^{2} \right).
  \end{equation}
\end{proposition}

\begin{proof}
  The proof is similar to that of Proposition~\ref{SectionII:GlobalEnergyEstimate:Proposition:Energy_Estimate_of_tauxx}: We first differentiate~\eqref{SectionII:GlobalEnergyEstimate:Proposition:Energy_Estimate_of_taux:Proof:tauxt} twice with respect to $x$, and multiply the resulting equation by $\tau_{xxx}$; then integrate the obtained equation with respect to $x$ and $t$. By a simple but somewhat lengthy computation, we obtain~\eqref{SectionII:GlobalEnergyEstimate:Proposition:Energy_Estimate_of_tauxxx:Energy_Estimate_of_tauxxx}.
\end{proof}

\begin{proposition}
  \label{SectionII:GlobalEnergyEstimate:Proposition:Energy_Estimate_of_tauxxxx}
  Let $(\tau,u,V)\in X_{T}\times Y_{T}\times C^2([0,T])$ be the solution to~\eqref{SectionI:CNSwithPointMass:Lagrangian:Perturbation}. Then there exist $\varepsilon_0>0$ and $C\geq 1$ such that if~\eqref{AposterioriSmallness} holds, then
  \begin{equation}
    \label{SectionII:GlobalEnergyEstimate:Proposition:Energy_Estimate_of_tauxxxx:Energy_Estimate_of_tauxxxx}
    ||\tau_{xxxx}(t)||^2+C^{-1}\int_{0}^{t}||\tau_{xxxx}(s)||^2 \, ds\leq C\left( ||\tau_0||_{4}^{2}+||u_0||_{4}^{2} \right).
  \end{equation}
\end{proposition}

\begin{proof}
  The proof is again similar to that of Propositions~\ref{SectionII:GlobalEnergyEstimate:Proposition:Energy_Estimate_of_tauxx}: We first differentiate~\eqref{SectionII:GlobalEnergyEstimate:Proposition:Energy_Estimate_of_taux:Proof:tauxt} three times with respect to $x$, and multiply the resulting equation by $\tau_{xxxx}$; then integrate the obtained equation with respect to $x$ and $t$. By a simple but quite lengthy computation, we obtain~\eqref{SectionII:GlobalEnergyEstimate:Proposition:Energy_Estimate_of_tauxxxx:Energy_Estimate_of_tauxxxx}.
\end{proof}

Now the following global-in-time energy estimate is an immediate consequence of Propositions~\ref{SectionII:GlobalEnergyEstimate:Proposition:EnergyConservation}--\ref{SectionII:GlobalEnergyEstimate:Proposition:Energy_Estimate_of_tauxxxx}.

\begin{theorem}
  \label{SectionII:GlobalEnergyEstimate:Theorem:MainTheorem}
  Let $(\tau,u,V)\in X_{T}\times Y_{T}\times C^2([0,T])$ be the solution to~\eqref{SectionI:CNSwithPointMass:Lagrangian:Perturbation}. Let
  \begin{align}
    M(t)
    & \coloneqq ||\tau(t)||_{4}^{2}+||u(t)||_{4}^{2}+\sum_{k=0}^{2}|\partial_{t}^{k}V(t)|^2 \\
    & \quad +\int_{0}^{t}||\tau_x(s)||_{3}^{2}\, ds+\int_{0}^{t}||u_x(s)||_{4}^{2}\, ds+\int_{0}^{t}\left( |V'(s)|^2+|V''(s)|^2 \right) \, ds.
  \end{align}
  There exist $\varepsilon_0,C>0$ independent of $T$ such that if
  \begin{equation}
    \varepsilon \coloneqq \sup_{0\leq t\leq T}||\tau(t)||_4+\sup_{0\leq t\leq T}||u(t)||_4 \leq \varepsilon_0,
  \end{equation}
  then
  \begin{equation}
    M(t)\leq C\varepsilon \quad (0\leq t\leq T).
  \end{equation}
\end{theorem}

Combining the local-in-time existence theorem (Theorem~\ref{SectionII:LocalExistence:Theorem:MainTheorem}) and the global-in-time energy estimate (Theorem~\ref{SectionII:GlobalEnergyEstimate:Theorem:MainTheorem}), we obtain the global-in-time existence theorem (Theorem~\ref{SectionI:Theorem:MainTheorem:GlobalExistence}) by the usual argument of continuation; see~\cite[Theorem~7.1]{MN80}.

\section{Pointwise Estimates of Solutions}
\label{SectionIII}
In this section, we prove Theorem~\ref{SectionI:Theorem:MainTheorem:PointwiseEstimates} on the pointwise estimates of solutions to~\eqref{SectionI:CNSwithPointMass:Lagrangian:Perturbation}.

\subsection{Preliminaries}
Denote by $G=G(x,t)\in \mathbb{R}^{2\times 2}$ the fundamental solution of the Cauchy problem corresponding to~\eqref{SectionI:CNSwithPointMass:Lagrangian:Perturbation:Linearized}:
\begin{equation}
  \label{SectionIII:CNS_Lag_per_Fundamental_Solution}
  \begin{dcases}
    \partial_t G+
    \begin{pmatrix}
      0    & -1 \\
      -c^2 & 0
    \end{pmatrix}
    \partial_x G=
    \begin{pmatrix}
      0 & 0 \\
      0 & \nu
    \end{pmatrix}
    \partial_{x}^{2}G,   & x\in \mathbb{R},\, t>0, \\
    G(x,0)=\delta(x)I_2, & x\in \mathbb{R},
  \end{dcases}
\end{equation}
where $\delta$ is the Dirac delta function and $I_2$ is the $2\times 2$ identify matrix. Let us recall the pointwise estimates of $G$ proved in~\cite{LZ97,Zeng94}. First, let $G^*=G^*(x,t)\in \mathbb{R}^{2\times 2}$ be the modified fundamental solution defined by the following equations:
\begin{equation}
  \label{SectionIII:CNS_Lag_per_Fundamental_Solution_Modified}
  \begin{dcases}
    \partial_t G^*+
    \begin{pmatrix}
      0    & -1 \\
      -c^2 & 0
    \end{pmatrix}
    \partial_x G^* =\frac{\nu}{2}\partial_{x}^{2}G^*, & x\in \mathbb{R},\, t>0, \\
    G^*(x,0)=\delta(x)I_2,                            & x\in \mathbb{R}.
  \end{dcases}
\end{equation}
According to~\cite[p.~1060]{Zeng94}, $G^*$ has the following form:\footnote{The formula for $G^*$ in~\cite[p.~47]{LZ97} needs to be divided by $(2\pi)^{1/2}$.}
\begin{equation}
  \label{SectionIII:CNS_Lag_per_Fundamental_Solution_Modified_Explicit_Formula}
  G^*(x,t)=\frac{1}{2(2\pi \nu t)^{1/2}}e^{-\frac{(x-ct)^2}{2\nu t}}
  \begin{pmatrix}
    1  & -\frac{1}{c} \\
    -c & 1
  \end{pmatrix}
  +\frac{1}{2(2\pi \nu t)^{1/2}}e^{-\frac{(x+ct)^2}{2\nu t}}
  \begin{pmatrix}
    1 & \frac{1}{c} \\
    c & 1
  \end{pmatrix}.
\end{equation}
The difference $G-G^*$ has the following estimates~\cite[Theorem~5.8]{LZ97}: for any integer $l\geq 0$,
\begin{equation}
  \label{SectionIII:Bound_of_G-G*}
  \left| \partial_{x}^{l}G(x,t)-\partial_{x}^{l}G^*(x,t)-e^{-\frac{c^2}{\nu}t}\sum_{j=0}^{l}\delta^{(l-j)}(x)Q_j(t) \right| \leq C(t+1)^{-\frac{1}{2}}t^{-\frac{l+1}{2}}\left( e^{-\frac{(x-ct)^2}{Ct}}+e^{-\frac{(x+ct)^2}{Ct}} \right),
\end{equation}
where $\delta^{(k)}$ is the $k$-th derivative of the Dirac delta function and $Q_j=Q_j(t)$ is a $2\times 2$ polynomial matrix. Additionally, we have
\begin{equation}
  \label{SectionIII:Q0Q1}
  Q_0=
  \begin{pmatrix}
    1 & 0 \\
    0 & 0
  \end{pmatrix}
  , \quad Q_1=
  \begin{pmatrix}
    0                & -\frac{1}{\nu} \\
    -\frac{c^2}{\nu} & 0
  \end{pmatrix}.
\end{equation}
Note that the form of $Q_1$ is not explicitly stated in~\cite[Theorem~5.8]{LZ97}, but a closer look at the proof reveals that $Q_1$ has the form given above ($Q_1$ is exactly the matrix $C_M$ given in the proof of~\cite[Lemma~5.4]{LZ97}).

Next, we give a remark on the Laplace transform $\tilde{G}(x,s)\coloneqq \mathcal{L}[G](x,s)$.\footnote{We use $s$ to denote the Laplace variable; to save the symbol, we also use $s$ as the time integration variable as in $\int_{0}^{t}f(t-s)\, ds$.}~Let $\lambda=s/\sqrt{\nu s+c^2}$; we take the branch such that $\Re \lambda>0$ when $\Re s>0$. Then we have
\begin{equation}
  \label{SectionIII:Laplace_Transform_of_G}
  \tilde{G}(x,s)=\frac{1}{\nu s+c^2}
  \begin{pmatrix}
    \nu \delta(x)+\frac{c^2}{2\sqrt{\nu s+c^2}}e^{-\lambda |x|} & -\frac{\sgn(x)}{2}e^{-\lambda |x|} \\
    -\frac{c^2 \sgn(x)}{2}e^{-\lambda |x|}                      & \frac{s}{2\lambda}e^{-\lambda |x|}
  \end{pmatrix}.
\end{equation}
Note that a similar formula is derived in~\cite{DW18}. To show~\eqref{SectionIII:Laplace_Transform_of_G}, first, take the Fourier--Laplace transform of~\eqref{SectionIII:CNS_Lag_per_Fundamental_Solution} to obtain
\begin{equation}
  \label{SectionIII:Fourier_Lapalace_Transform_of_G}
  \mathcal{F}[\tilde{G}](\xi,s)=
  \begin{pmatrix}
    s         & -i\xi \\
    -ic^2 \xi & s+\nu \xi^2 
  \end{pmatrix}^{-1}
  =\frac{1}{s^2+(\nu s+c^2)\xi^2}
  \begin{pmatrix}
    s+\nu \xi^2 & i\xi \\
    ic^2 \xi    & s
  \end{pmatrix}.
\end{equation}
By the residue theorem, we can calculate the following integrals: when $\Re s>0$ and $x\neq 0$,
\begin{align}
  \frac{1}{2\pi}\int_{-\infty}^{\infty}\frac{e^{i\xi x}}{\xi^2+\lam^2}\, d\xi & =\frac{e^{-\lam |x|}}{2\lam}, \\
  \frac{1}{2\pi}\int_{-\infty}^{\infty}\frac{\xi^2}{\xi^2+\lam^2}e^{i\xi x}\, d\xi & =\delta(x)-\frac{\lam e^{-\lam |x|}}{2}, \\
  \frac{1}{2\pi}\int_{-\infty}^{\infty}\frac{i\xi}{\xi^2+\lam^2}e^{i\xi x}\, d\xi & =\frac{d}{dx}\frac{e^{-\lam |x|}}{2\lam}=-\frac{\sgn(x)}{2}e^{-\lam |x|}.
\end{align}
Applying these formulae, we can compute the inverse Fourier transform of~\eqref{SectionIII:Fourier_Lapalace_Transform_of_G} to obtain~\eqref{SectionIII:Laplace_Transform_of_G}.

Next, we introduce some notations. Let
\begin{equation}
  N(x,t)\coloneqq -p(1+\tau)+p(1)-c^2 \tau-\nu \frac{\tau}{1+\tau}u_x
\end{equation}
and
\begin{equation}
  \label{SectionIII:Gb}
  G_T(x,t)\coloneqq \mathcal{L}^{-1}\left[ \frac{2}{\lambda+2}\tilde{G} \right](x,t), \quad G_R(x,t)\coloneqq (G-G_T)(x,t)
  \begin{pmatrix}
    1 & 0 \\
    0 & -1
  \end{pmatrix},
\end{equation}
where $\mathcal{L}^{-1}$ is the inverse Laplace transform. The subscripts ``T'' and ``R'' stand for ``transmission'' and ``reflection'' since $G_T$ and $G_R$ can be understood as Green's functions describing transmission and reflection at the point mass; see Remark~\ref{Remark:GT_GR}.

Next, we give a formula for $G_T$. To show this, we use the differential equation method as in~\cite{Deng16,DWY15,DW18}. By~\eqref{SectionIII:Laplace_Transform_of_G}, for $x>0$, $G_T$ satisfies
\begin{equation}
  \label{SectionIII:Eq:derGT}
  \partial_x G_T(x,t)=\mathcal{L}\left[ \frac{-2\lambda}{\lambda+2}\tilde{G} \right](x,t)=-2G(x,t)+2G_T(x,t).
\end{equation}
Solving  this equation, we obtain
\begin{equation}
  \label{SectionIII:GbFormula}
  G_T(x,t)=2\int_{-\infty}^{0}e^{2z}G(x-z,t)\, dz \quad (x>0).
\end{equation}
By~\eqref{SectionIII:Eq:derGT}, we also have
\begin{equation}
  \label{SectionIII:GbIntByParts}
  G_R(x,t)=-\frac{1}{2}\partial_x G_T(x,t)
  \begin{pmatrix}
    1 & 0 \\
    0 & -1
  \end{pmatrix}
  \quad (x>0).
\end{equation}
A similar formula can be derived for $x<0$. By using~\eqref{SectionIII:GbFormula}, we can show that $G_T$ satisfies the following bounds:
\begin{align}
  \label{SectionIII:Bounds_of_Gb}
  |\partial_{x}^{k}G_T(x,t)|\leq C(t+1)^{-1/2}t^{-k/2}\left( e^{-\frac{(x-ct)^2}{Ct}}+e^{-\frac{(x+ct)^2}{Ct}} \right)+Ce^{-\frac{|x|+t}{C}}
\end{align}
for $(x,t)\in \mathbb{R}_*\times (0,\infty)$ and for any integer $k\geq 0$. See Appendix~\ref{AppendixA} for the proof. We note that $\partial_{x}^{k}G_T$ has weaker singularity at $t=0$ than $\partial_{x}^{k}G$ and contains no singularity for $t>0$.

\subsection{Derivation of an Integral Equation}
The following proposition gives an integral equation satisfied by the solution $(\tau,u,V)$ to~\eqref{SectionI:CNSwithPointMass:Lagrangian:Perturbation}. The integral equation is derived by solving~\eqref{SectionI:CNSwithPointMass:Lagrangian:Perturbation} using the Laplace transform (treating the nonlinear terms as source terms) and~\eqref{SectionIII:Laplace_Transform_of_G} to rewrite the obtained formula in terms of $G$, which is a technique developed and used in~\cite{Deng16,DWY15,DW18}.

\begin{proposition}
  \label{SectionIII:Proposition:IntegralEquation}
  Let $(\tau,u,V)$ be the solution to~\eqref{SectionI:CNSwithPointMass:Lagrangian:Perturbation} guaranteed to exist under the assumptions of Theorem~\ref{SectionI:Theorem:MainTheorem:GlobalExistence}. Then $(\tau,u)$ satisfies the following integral equation:\footnote{In what follows, $V$ only appears implicitly through the relation $V(t)=u(\pm 0,t)$.}~for $x>0$,
  \begin{align}
    \label{SectionIII:Proposition:IntegralEquation:IntegralEquation}
    \begin{aligned}
      \begin{pmatrix}
        \tau \\
        u
      \end{pmatrix}
      (x,t) & =\int_{0}^{\infty}G(x-y,t)
      \begin{pmatrix}
        \tau_0 \\
        u_0
      \end{pmatrix}
      (y)\, dy+\int_{0}^{\infty}G_R(x+y,t)
      \begin{pmatrix}
        \tau_0 \\
        u_0
      \end{pmatrix}
      (y)\, dy \\
      & \quad +\int_{-\infty}^{0}G_T(x-y,t)
      \begin{pmatrix}
        \tau_0 \\
        u_0
      \end{pmatrix}
      (y)\, dy+G_T(x,t)
      \begin{pmatrix}
        0 \\
        V_0
      \end{pmatrix} \\
      & \quad +\int_{0}^{t}\int_{0}^{\infty}G(x-y,t-s)
      \begin{pmatrix}
        0 \\
        N_x
      \end{pmatrix}
      (y,s)\, dyds \\
      & \quad +\int_{0}^{t}\int_{0}^{\infty}G_R(x+y,t-s)
      \begin{pmatrix}
        0 \\
        N_x
      \end{pmatrix}
      (y,s)\, dyds \\
      & \quad +\int_{0}^{t}\int_{-\infty}^{0}G_T(x-y,t-s)
      \begin{pmatrix}
        0 \\
        N_x
      \end{pmatrix}
      (y,s)\, dyds \\
      & \quad +\int_{0}^{t}G_T(x,t-s)
      \begin{pmatrix}
        0 \\
        \llbracket N \rrbracket
      \end{pmatrix}
      (0,s)\, ds.
    \end{aligned}
  \end{align}
  A similar formula holds for $x<0$.
\end{proposition}

\begin{proof}
  First, let $(\tau_1,u_1)$ be the generalized solution to the following Cauchy problem:\footnote{Since the initial data $(\tau_0,u_0)$ may have discontinuity across $x=0$, the solution $(\tau_1,u_1)$ should be sought in the distributional sense (cf.~\cite{Kim83}).}
  \begin{equation}
    \begin{dcases}
      \partial_t \tau_1-\partial_x u_1=0,                                        & x\in \mathbb{R},\, t>0, \\
      \partial_t u_1-c^2 \partial_x \tau_1=\nu \partial_{x}^{2}u_1+\partial_x N, & x\in \mathbb{R},\, t>0, \\
      \tau_1(x,0)=\tau_0(x),\, u_1(x,0)=u_0(x),                                  & x\in \mathbb{R}.
    \end{dcases}
  \end{equation}
  Here, $\tau$ and $u$, hence also $N$ are considered to be known functions. By using the fundamental solution $G$, the solution $(\tau_1,u_1)$ can be expressed as follows:
  \begin{equation}
    \label{SectionIII:Proposition:IntegralEquation:Proof:tau1_u1_Integral_Representation}
    \begin{pmatrix}
      \tau_1 \\
      u_1
    \end{pmatrix}
    (x,t)=\int_{-\infty}^{\infty}G(x-y,t)
    \begin{pmatrix}
      \tau_0 \\
      u_0
    \end{pmatrix}
    (y)\, dy+\int_{0}^{t}\int_{-\infty}^{\infty}G(x-y,t-s)
    \begin{pmatrix}
      0 \\
      N_x
    \end{pmatrix}
    (y,s)\, dyds.
  \end{equation}
  Next, let $(\tau_2,u_2)$ be the solution to the following system:
  \begin{equation}
    \label{SectionIII:Proposition:IntegralEquation:Proof:tau2_u2}
    \begin{dcases}
      \partial_t \tau_2-\partial_x u_2=0,                                                                                       & x\in \mathbb{R}_*,\, t>0, \\
      \partial_t u_2-c^2 \partial_x \tau_2=\nu \partial_{x}^{2}u_2,                                                             & x\in \mathbb{R}_*,\, t>0, \\
      \partial_t u_2(\pm 0,t)=\llbracket c^2 \tau+\nu u_x \rrbracket(0,t)-\partial_t u_1(\pm 0,t)+\llbracket N \rrbracket(0,t), & t>0, \\
      \tau_2(x,0)=0,\, u_2(x,0)=0,                                                                                              & x\in \mathbb{R}_*.
    \end{dcases}
  \end{equation}
  Then we have $(\tau,u)=(\tau_1+\tau_2,u_1+u_2)$ by the uniqueness of solutions to~\eqref{SectionIII:Proposition:IntegralEquation:Proof:tau2_u2}.\footnote{Here, we do not attempt to prove the unique existence theorem for~\eqref{SectionIII:Proposition:IntegralEquation:Proof:tau2_u2} since it is rather standard.}

  By the first equation in~\eqref{SectionIII:Proposition:IntegralEquation:Proof:tau2_u2}, the third equation in~\eqref{SectionIII:Proposition:IntegralEquation:Proof:tau2_u2} can be written as
  \begin{equation}
    \partial_t u_2(\pm 0,t)-\llbracket c^2 \tau_2+\nu \partial_t \tau_2 \rrbracket(0,t)=\llbracket c^2 \tau_1+\nu \partial_x u_1 \rrbracket(0,t)-\partial_t u_1(\pm 0,t)+\llbracket N \rrbracket(0,t).
  \end{equation}
  By using~\eqref{SectionIII:CNS_Lag_per_Fundamental_Solution_Modified_Explicit_Formula}, \eqref{SectionIII:Bound_of_G-G*} and~\eqref{SectionIII:Q0Q1}, we see that
  \begin{equation}
    \label{SectionIII:Proposition:IntegralEquation:Proof:Structure_of_Singularity}
    \llbracket c^2 \tau_1+\nu \partial_x u_1 \rrbracket(0,t)=0, \quad u_1(+0,t)=u_1(-0,t).\footnote{We use these equalities in the following to simplify the calculations; but we can carry out similar computations without using these to arrive at the same conclusion.}
  \end{equation}
   Now, let
  \begin{equation}
    \label{SectionIII:Proposition:IntegralEquation:Proof:Psi}
    \Psi(s)\coloneqq -s\tilde{u}_1(+0,s)+V_0+\llbracket \tilde{N} \rrbracket(0,s),
  \end{equation}
  where the Laplace transformed variables are denoted with tildes. Taking the Laplace transform of~\eqref{SectionIII:Proposition:IntegralEquation:Proof:tau2_u2} and using~\eqref{SectionIII:Proposition:IntegralEquation:Proof:Structure_of_Singularity}, we obtain
  \begin{align}
    \label{SectionIII:Proposition:IntegralEquation:Proof:tau2_u2_Laplace_Transformed}
    \begin{dcases}
      s\tilde{\tau}_2-\partial_x \tilde{u}_2=0,                                           & x\in \mathbb{R}_*,\, \Re s>0, \\
      s\tilde{u}_2-c^2 \partial_x \tilde{\tau}_2=\nu \partial_{x}^{2}\tilde{u}_2,         & x\in \mathbb{R}_*,\, \Re s>0, \\
      s\tilde{u}_2(\pm 0,s)-(\nu s+c^2)\llbracket \tilde{\tau}_2 \rrbracket(0,s)=\Psi(s), & \Re s>0.
    \end{dcases}
  \end{align}
  The general solutions to~\eqref{SectionIII:Proposition:IntegralEquation:Proof:tau2_u2_Laplace_Transformed} are given by
  \begin{equation}
    \label{SectionIII:Proposition:IntegralEquation:Proof:tau2_u2_General_Solutions}
    \begin{pmatrix}
      \tilde{\tau}_2 \\
      \tilde{u}_2
    \end{pmatrix}
    (x,s)=
    \begin{dcases}
      C_+
      \begin{pmatrix}
        -\lambda/s \\
        1
      \end{pmatrix}
      e^{-\lambda x} & (x>0), \\
      C_-
      \begin{pmatrix}
        \lambda/s \\
        1
      \end{pmatrix}
      e^{\lambda x} & (x<0),
    \end{dcases}
  \end{equation}
  and the constants $C_+$ and $C_-$ are determined by the third equation in~\eqref{SectionIII:Proposition:IntegralEquation:Proof:tau2_u2_Laplace_Transformed}:
  \begin{equation}
    \begin{dcases}
      C_+ s+(C_+ +C_-)\lambda(\nu s+c^2)/s=\Psi(s), \\
      C_- s+(C_+ +C_-)\lambda(\nu s+c^2)/s=\Psi(s).
    \end{dcases}
  \end{equation}
  Solving these equations, we obtain
  \begin{equation}
    C_+=C_-=\frac{s\Psi(s)}{s^2+2\lambda(\nu s+c^2)}.
  \end{equation}
  Substituting these into~\eqref{SectionIII:Proposition:IntegralEquation:Proof:tau2_u2_General_Solutions}, we obtain, for $x>0$,
  \begin{equation}
    \label{SectionIII:Proposition:IntegralEquation:Proof:tau2_u2_Specific_Solution}
    \begin{pmatrix}
      \tilde{\tau}_2 \\
      \tilde{u}_2
    \end{pmatrix}
    (x,s)=\frac{s}{s^2+2\lambda(\nu s+c^2)}
    \begin{pmatrix}
      -\lambda/s \\
      1
    \end{pmatrix}
    \Psi(s)e^{-\lambda x}.
  \end{equation}
  Substituting~\eqref{SectionIII:Proposition:IntegralEquation:Proof:tau1_u1_Integral_Representation} into~\eqref{SectionIII:Proposition:IntegralEquation:Proof:Psi} and using~\eqref{SectionIII:Laplace_Transform_of_G}, we obtain, for $x>0$,
  \begin{align}
    \begin{pmatrix}
      -\lambda/s \\
      1
    \end{pmatrix}
    \Psi(s)e^{-\lambda x}
    & =
    \begin{pmatrix}
      -\lambda/s \\
      1
    \end{pmatrix}
    \begin{pmatrix}
      0 & -s
    \end{pmatrix}
    e^{-\lambda x}\int_{-\infty}^{\infty}\tilde{G}(+0-y,s)
    \begin{pmatrix}
      \tau_0(y) \\
      u_0(y)+\tilde{N}_x(y,s)
    \end{pmatrix}
    \, dy \\
    & \quad +
    \begin{pmatrix}
      -\lambda/s \\
      1
    \end{pmatrix}
    e^{-\lambda x}\left( V_0+\llbracket \tilde{N} \rrbracket(0,s) \right) \\
    & =\int_{-\infty}^{\infty}\frac{e^{-\lambda(x+|y|)}}{2(\nu s+c^2)}
    \begin{pmatrix}
      c^2 \lambda \sgn(y) & s \\
      -c^2 s\sgn(y)       & -s^2/\lambda
    \end{pmatrix}
    \begin{pmatrix}
      \tau_0(y) \\
      u_0(y)+\tilde{N}_x(y,s)
    \end{pmatrix}
    \, dy \\
    & \quad +
    \begin{pmatrix}
      -\lambda/s \\
      1
    \end{pmatrix}
    e^{-\lambda x}\left( V_0+\llbracket \tilde{N} \rrbracket(0,s) \right) \\
    & =s\int_{0}^{\infty}\tilde{G}(x+y,s)
    \begin{pmatrix}
      1 & 0 \\
      0 & -1
    \end{pmatrix}
    \begin{pmatrix}
      \tau_0(y) \\
      u_0(y)+\tilde{N}_x(y,s)
    \end{pmatrix}
    \, dy \\
    & \quad -s\int_{-\infty}^{0}\tilde{G}(x-y,s)
    \begin{pmatrix}
      \tau_0(y) \\
      u_0(y)+\tilde{N}_x(y,s)
    \end{pmatrix}
    \, dy \\
    & \quad +2\sqrt{\nu s+c^2}\tilde{G}(x,s)
    \begin{pmatrix}
      0 \\
      V_0+\llbracket \tilde{N} \rrbracket(0,s)
    \end{pmatrix}.
  \end{align}
  Plugging this into~\eqref{SectionIII:Proposition:IntegralEquation:Proof:tau2_u2_Specific_Solution}, we obtain
  \begin{align}
    \begin{pmatrix}
      \tilde{\tau}_2 \\
      \tilde{u}_2
    \end{pmatrix}
    (x,s)
    & =\frac{\lambda}{\lambda+2}\int_{0}^{\infty}\tilde{G}(x+y,s)
    \begin{pmatrix}
      1 & 0 \\
      0 & -1
    \end{pmatrix}
    \begin{pmatrix}
      \tau_0(y) \\
      u_0(y)+\tilde{N}_x(y,s)
    \end{pmatrix}
    \, dy \\
    & \quad -\frac{\lambda}{\lambda+2}\int_{-\infty}^{0}\tilde{G}(x-y,s)
    \begin{pmatrix}
      \tau_0(y) \\
      u_0(y)+\tilde{N}_x(y,s)
    \end{pmatrix}
    \, dy \\
    & \quad +\frac{2}{\lambda+2}\tilde{G}(x,s)
    \begin{pmatrix}
      0 \\
      V_0+\llbracket \tilde{N} \rrbracket(0,s)
    \end{pmatrix}.
  \end{align}
  Taking the inverse Laplace transform of this equation and noting that $G_T$ and $G_R$ are defined by~\eqref{SectionIII:Gb}, we obtain~\eqref{SectionIII:Proposition:IntegralEquation:IntegralEquation}.
\end{proof}

\begin{remark}
  \label{Remark:GT_GR}
  Looking at~\eqref{SectionIII:Proposition:IntegralEquation:IntegralEquation}, we can interpret $G_T$ and $G_R$ as Green's functions describing transmission and reflection of waves at the point mass. To see this, fix $x>0$. Then the first term on the right-hand side of~\eqref{SectionIII:Proposition:IntegralEquation:IntegralEquation} is the contribution of the initial data coming directly from $y>0$ to $x$; the second term is the contribution from $y>0$ to $x$ that has travelled the distance $x+y$, which can be interpreted as having travelled from $y$ to the origin and then reflected back to $x$; the third term is the contribution from $y<0$ to $x$ that has travelled the distance $x-y$, which can be interpreted as having travelled from $y$ to the origin and transmitted to reach $x$. For the forth term, consider an initial condition of the form $(\tau_0,u_0)=(0,V_0 \delta_{x_*})$, where $\delta_{x_*}$ is the delta function concentrating at $x_*>0$; then by the interpretation given above, its contribution is
  \begin{equation}
    G(x-x_*,t)
    \begin{pmatrix}
      0 \\
      V_0
    \end{pmatrix}
    +G_R(x+x_*,t)
    \begin{pmatrix}
      0 \\
      V_0
    \end{pmatrix}.
  \end{equation}
  Thus the forth term is obtained by taking the limit $x_* \to +0$ in this expression, and it can be interpreted as the influence of giving a momentum of $mV_0$ ($m=1$) to the fluid at $x=0$ at the initial time. The remaining nonlinear terms can be interpreted similarly. Note that by~\eqref{SectionIII:GbIntByParts} and~\eqref{SectionIII:Bounds_of_Gb}, we can see that the reflected waves decay faster than the transmitted waves.
\end{remark}

\subsection{Proof of the Pointwise Estimates of Solutions}
\label{SectionIII:PointwiseEstimates}
Recall the definition of $v_i$ and $\Psi_i$ given in Section~\ref{SectionI:Subsection:MainTheorem}, and let
\begin{equation}
  P(t)\coloneqq \sup_{0\leq s\leq t}\sum_{i=1}^{2}\left\{ |v_i(\cdot,s)\Psi_i(\cdot,s)^{-1}|_{\infty}+|v_{ix}(\cdot,s)|_{\infty}(s+1) \right\}+\sup_{0\leq s\leq t}|u_{xx}(\cdot,s)|_{\infty}(s+1)^{1/2}.
\end{equation}
It should be noted that we do not know a priori that $P(t)$ is finite. In what follows, as in the previous works~\cite{DW18,LZ97}, we shall tacitly assume that $P(t)$ is already known to be finite. One way to justify this assumption, as was done in~\cite[p.~296]{IK02}, is to first derive estimates for suitably weighted versions of $v_i$ (for which the corresponding $P(t)$ is trivially finite) and then take the limit to the original $v_i$ afterwards. We omit the details for brevity.

In order to prove Theorem~\ref{SectionI:Theorem:MainTheorem:PointwiseEstimates}, it suffices to prove that there exists a positive constant $C$ such that $P(t)\leq C\delta$ for all $t\geq 0$. See Theorem~\ref{SectionI:Theorem:MainTheorem:PointwiseEstimates} for the definition of $\delta$. To show this, we aim to prove that there exists a positive constant $C$ such that
\begin{equation}
  \label{SectionIII:GoalInequality}
  P(t)\leq C\delta +C(\delta+P(t))^2 \quad (t\geq 0).
\end{equation}
Then by taking $\delta$ sufficiently small, we can conclude that $P(t)\leq C\delta$ for all $t\geq 0$.

To show~\eqref{SectionIII:GoalInequality}, we first rewrite~\eqref{SectionIII:Proposition:IntegralEquation:IntegralEquation} to get an integral equation for $v_i$. Let
\begin{equation}
  g_i \coloneqq l_i G
  \begin{pmatrix}
    r_1 & r_2
  \end{pmatrix}
  , \quad g_{i}^{*}\coloneqq l_i G^*
  \begin{pmatrix}
    r_1 & r_2
  \end{pmatrix}
  , \quad g_{T,i}\coloneqq l_i G_T
  \begin{pmatrix}
    r_1 & r_2
  \end{pmatrix}
  , \quad g_{R,i}\coloneqq l_i G_R
  \begin{pmatrix}
    r_1 & r_2
  \end{pmatrix}.
\end{equation}
See Section~\ref{SectionI:Subsection:Formulation} for the definition of $r_i$ and $l_i$. We note that $(r_1\, r_2)(l_1\, l_2)^{T}=I_2$. Next, let
\begin{equation}
  N_i\coloneqq l_i
  \begin{pmatrix}
    0 \\
    N
  \end{pmatrix}
  , \quad N_{i}^{*}\coloneqq -\frac{1}{2}\theta_{i}^{2}.
\end{equation}
In fact, $N_i$ does not depend on the index $i$, but we shall keep the index $i$ just to discriminate it from $N$.

Multiplying $l_i$ to~\eqref{SectionIII:Proposition:IntegralEquation:IntegralEquation} from the left, we obtain, for $x>0$,
\begin{align}
  \label{SectionIII:uiIntegralEquation}
  \begin{aligned}
    u_i(x,t)
    & =\int_{0}^{\infty}g_i(x-y,t)
    \begin{pmatrix}
      u_{01} \\
      u_{02}
    \end{pmatrix}
    (y)\, dy+\int_{0}^{\infty}g_{R,i}(x+y,t)
    \begin{pmatrix}
      u_{01} \\
      u_{02}
    \end{pmatrix}
    (y)\, dy \\
    & \quad +\int_{-\infty}^{0}g_{T,i}(x-y,t)
    \begin{pmatrix}
      u_{01} \\
      u_{02}
    \end{pmatrix}
    (y)\, dy+m_V g_{T,i}(x,t)\bm{1} \\
    & \quad +\int_{0}^{t}\int_{0}^{\infty}g_i(x-y,t-s)
    \begin{pmatrix}
      N_1 \\
      N_2
    \end{pmatrix}_x
    (y,s)\, dyds \\
    & \quad +\int_{0}^{t}\int_{0}^{\infty}g_{R,i}(x+y,t-s)
    \begin{pmatrix}
      N_1 \\
      N_2
    \end{pmatrix}_x
    (y,s)\, dyds \\
    & \quad +\int_{0}^{t}\int_{-\infty}^{0}g_{T,i}(x-y,t-s)
    \begin{pmatrix}
      N_1 \\
      N_2
    \end{pmatrix}_x
    (y,s)\, dyds \\
    & \quad +\int_{0}^{t}g_{T,i}(x,t-s)
    \begin{pmatrix}
      \llbracket N_1 \rrbracket \\
      \llbracket N_2 \rrbracket
    \end{pmatrix}
    (0,s)\, ds,
  \end{aligned}
\end{align}
where $\bm{1}=(1\, 1)^{T}$. A similar formula holds for $x<0$.

Next, note that $(\tau^* \, u^*)^{T}\coloneqq (r_1 \, r_2)(\theta_1 \, \theta_2)^{T}$ solves
\begin{equation}
  \label{SectionIII:CNS_Lag_per_Modified}
  \begin{pmatrix}
    \tau^* \\
    u^*
  \end{pmatrix}_t
  +
  \begin{pmatrix}
    0    & -1 \\
    -c^2 & 0
  \end{pmatrix}
  \begin{pmatrix}
    \tau^* \\
    u^*
  \end{pmatrix}_x
  =\frac{\nu}{2}
  \begin{pmatrix}
    \tau^* \\
    u^*
  \end{pmatrix}_{xx}
  +
  \begin{pmatrix}
    r_1 & r_2
  \end{pmatrix}
  \begin{pmatrix}
    N_{1}^{*} \\
    N_{2}^{*}
  \end{pmatrix}_x
\end{equation}
since multiplying $l_i$ to~\eqref{SectionIII:CNS_Lag_per_Modified} from the left yields
\begin{equation}
  \label{SectionIII:Burgers_for_thetai}
  \partial_t \theta_i+\lambda_i \partial_x \theta_i=\frac{\nu}{2}\partial_{x}^{2}\theta_i+\partial_x N_{i}^{*},
\end{equation}
which is just~\eqref{SectionI:Thetai} with $\Theta_i$ replaced by $\theta_i$. Hence, by~\eqref{SectionIII:CNS_Lag_per_Fundamental_Solution_Modified}, we have
\begin{equation}
  \label{SectionIII:thetaiIntegralEquation}
  \theta_i(x,t)=\int_{-\infty}^{\infty}g_{i}^{*}(x-y,t)
  \begin{pmatrix}
    \theta_1 \\
    \theta_2
  \end{pmatrix}
  (y,0)\, dy+\int_{0}^{t}\int_{-\infty}^{\infty}g_{i}^{*}(x-y,t-s)
  \begin{pmatrix}
    N_{1}^{*} \\
    N_{2}^{*}
  \end{pmatrix}_x
  (y,s)\, dyds.
\end{equation}
Note that
\begin{equation}
  \label{SectionIII:gi*_Explicit_Formula}
  g_{1}^{*}=\frac{1}{(2\pi \nu t)^{1/2}}e^{-\frac{(x-ct)^2}{2\nu t}}
  \begin{pmatrix}
    1 & 0
  \end{pmatrix}
  , \quad  g_{2}^{*}=\frac{1}{(2\pi \nu t)^{1/2}}e^{-\frac{(x+ct)^2}{2\nu t}}
  \begin{pmatrix}
    0 & 1
  \end{pmatrix},
\end{equation}
which are natural since $\theta_i$ solves~\eqref{SectionIII:Burgers_for_thetai}.

Now, combining~\eqref{SectionIII:uiIntegralEquation} and~\eqref{SectionIII:thetaiIntegralEquation}, we obtain, for $x>0$,
\begin{align}
  \label{SectionIII:viIntegralEquation}
  \begin{aligned}
    v_i(x,t)
    & =\int_{-\infty}^{\infty}g_i(x-y,t)
    \begin{pmatrix}
      u_{01} \\
      u_{02}
    \end{pmatrix}
    (y)\, dy-\int_{-\infty}^{\infty}g_{i}^{*}(x-y,t)
    \begin{pmatrix}
      \theta_1 \\
      \theta_2
    \end{pmatrix}
    (y,0)\, dy \\
    & \quad +\int_{0}^{\infty}g_{R,i}(x+y,t)
    \begin{pmatrix}
      u_{01} \\
      u_{02}
    \end{pmatrix}
    (y)\, dy \\
    & \quad +\int_{-\infty}^{0}g_{R,i}(x-y,t)
    \begin{pmatrix}
      u_{02} \\
      u_{01}
    \end{pmatrix}
    (y)\, dy+m_V g_{T,i}(x,t)\bm{1} \\
    & \quad +\int_{0}^{t}\int_{-\infty}^{\infty}g_{i}^{*}(x-y,t-s)
    \begin{pmatrix}
      N_1-N_{1}^{*} \\
      N_2-N_{2}^{*}
    \end{pmatrix}_x
    (y,s)\, dyds \\
    & \quad +\int_{0}^{t}\int_{-\infty}^{\infty}(g_i-g_{i}^{*})(x-y,t-s)
    \begin{pmatrix}
      N_1 \\
      N_2
    \end{pmatrix}_x
    (y,s)\, dyds \\
    & \quad +\int_{0}^{t}\int_{0}^{\infty}g_{R,i}(x+y,t-s)
    \begin{pmatrix}
      N_1 \\
      N_2
    \end{pmatrix}_x
    (y,s)\, dyds \\
    & \quad +\int_{0}^{t}\int_{-\infty}^{0}g_{R,i}(x-y,t-s)
    \begin{pmatrix}
      N_1 \\
      N_2
    \end{pmatrix}_x
    (y,s)\, dyds \\
    & \quad +\int_{0}^{t}g_{T,i}(x,t-s)
    \begin{pmatrix}
      \llbracket N_1 \rrbracket \\
      \llbracket N_2 \rrbracket
    \end{pmatrix}
    (0,s)\, ds.
  \end{aligned}
\end{align}
Note that we used $g_{R,i}=l_i(G_T-G)(r_2\, r_1)$ and $N_1=N_2$. A similar formula holds for $x<0$. Denote the sum of the first five terms on the right-hand side by $\mathcal{I}_i(x,t)$ and the sum of the next five terms by $\mathcal{N}_i(x,t)$.

We first give a bound for $\mathcal{I}_i(x,t)$.

\begin{lemma}
  \label{SectionIII:Lemma:Bound_of_Ii}
  Under the assumptions of Theorem~\ref{SectionI:Theorem:MainTheorem:PointwiseEstimates}, there exists a positive constant $C$ such that
  \begin{equation}
    \label{SectionIII:Lemma:Bound_of_Ii:Bound_of_Ii}
    |\mathcal{I}_i(x,t)|\leq C\delta \Psi_i(x,t)
  \end{equation}
  for $(x,t)\in \mathbb{R}_* \times (0,\infty)$.
\end{lemma}

\begin{proof}
  Since the case when $t\leq 1$ can be handled easily by using the assumptions of Theorem~\ref{SectionI:Theorem:MainTheorem:PointwiseEstimates}, \eqref{SectionIII:CNS_Lag_per_Fundamental_Solution_Modified_Explicit_Formula}, \eqref{SectionIII:Bound_of_G-G*}, \eqref{SectionIII:GbIntByParts} and~\eqref{SectionIII:Bounds_of_Gb}, we assume that $t\geq 1$ in the following. We also assume that $x>0$; the case when $x<0$ is similar. Let $v_{0i}\coloneqq \restr{v_i}{t=0}$. First, we rewrite $\mathcal{I}_{i}(x,t)$ as follows:
  \begin{align}
    \mathcal{I}_i(x,t)
    & =\int_{-\infty}^{\infty}g_{i}^{*}(x-y,t)
    \begin{pmatrix}
      v_{01} \\
      v_{02}
    \end{pmatrix}
    (y)\, dy+m_V g_{i}^{*}(x,t)\bm{1} \\
    & \quad +\int_{-\infty}^{\infty}(g_i-g_{i}^{*})(x-y,t)
    \begin{pmatrix}
      u_{01} \\
      u_{02}
    \end{pmatrix}
    (y)\, dy \\
    & \quad +\int_{0}^{\infty}g_{R,i}(x+y,t)
    \begin{pmatrix}
      u_{01} \\
      u_{02}
    \end{pmatrix}
    (y)\, dy \\
    & \quad +\int_{-\infty}^{0}g_{R,i}(x-y,t)
    \begin{pmatrix}
      u_{02} \\
      u_{01}
    \end{pmatrix}
    (y)\, dy \\
    & \quad +m_V (g_{T,i}-g_{i}^{*})(x,t)\bm{1}.
  \end{align}
  Let
  \begin{align}
    \mathcal{I}_{i,1}(x,t) & \coloneqq \int_{-\infty}^{\infty}g_{i}^{*}(x-y,t)
    \begin{pmatrix}
      v_{01} \\
      v_{02}
    \end{pmatrix}
    (y)\, dy+m_V g_{i}^{*}(x,t)\bm{1}, \\
    \mathcal{I}_{i,2}(x,t) & \coloneqq \int_{-\infty}^{\infty}(g_i-g_{i}^{*})(x-y,t)
    \begin{pmatrix}
      u_{01} \\
      u_{02}
    \end{pmatrix}
    (y)\, dy, \\
    \mathcal{I}_{i,3}(x,t) & \coloneqq \int_{0}^{\infty}g_{R,i}(x+y,t)
    \begin{pmatrix}
      u_{01} \\
      u_{02}
    \end{pmatrix}
    (y)\, dy+\int_{-\infty}^{0}g_{R,i}(x-y,t)
    \begin{pmatrix}
      u_{02} \\
      u_{01}
    \end{pmatrix}
    (y)\, dy, \\
    \mathcal{I}_{i,4}(x,t) & \coloneqq m_V (g_{T,i}-g_{i}^{*})(x,t)\bm{1}.
  \end{align}

  We first give a bound for $\mathcal{I}_{i,1}(x,t)$. Define $\eta_j$ by
  \begin{equation}
    \eta_j(x)\coloneqq \int_{-\infty}^{x}v_{0j}(y)\, dy+m_V H(x),
  \end{equation}
  where $H(x)$ is the Heaviside function. Let $\eta=(\eta_1 \, \eta_2)^{T}$. Then we have
  \begin{equation}
    \mathcal{I}_{i,1}(x,t)=\int_{-\infty}^{\infty}g_{i}^{*}(x-y,t)\partial_x \eta(y)\, dy.
  \end{equation}
  Note that by~\eqref{SectionI:Eq:SelfSimilarSolutionMass}, we have
  \begin{equation}
    \eta_j(x)=
    \begin{dcases}
      \int_{-\infty}^{x}v_{0j}(y)\, dy & (x<0), \\
      -\int_{x}^{\infty}v_{0j}(y)\, dy & (x>0).
    \end{dcases}
  \end{equation}
  Thus, by the assumptions of Theorem~\ref{SectionI:Theorem:MainTheorem:PointwiseEstimates}, we have
  \begin{equation}
    \label{SectionIII:Lemma:Bound_of_Ii:Proof:Bound_of_etaj}
    ||\eta_j||_{L^1(\mathbb{R}_*)}\leq C\delta,\quad |\eta_j(x)|\leq C\delta(|x|+1)^{-1}.
  \end{equation}
  We prove the following bound case-by-case:
  \begin{equation}
    \label{SectionIII:Lemma:Bound_of_Ii:Proof:Bound_of_Ii1}
    |\mathcal{I}_{i,1}(x,t)|\leq C\delta \Psi_i(x,t).
  \end{equation}
  Case (i): $|x-\lambda_i t|\leq (t+1)^{1/2}$. By integration by parts, \eqref{SectionIII:CNS_Lag_per_Fundamental_Solution_Modified_Explicit_Formula} and~\eqref{SectionIII:Lemma:Bound_of_Ii:Proof:Bound_of_etaj}, we have
  \begin{align}
    |\mathcal{I}_{i,1}(x,t)|
    & =\left| \int_{-\infty}^{\infty}\partial_x g_{i}^{*}(x-y,t)\eta(y)\, dy \right| \\
    & \leq C(t+1)^{-1}\int_{-\infty}^{\infty}|\eta(x)|\, dx \\
    & \leq C\delta (t+1)^{-1}\leq C\delta \Psi_i(x,t).
  \end{align}
  Case (ii): $(t+1)^{1/2}<|x-\lambda_i t|<t+1$. Suppose that $x-\lambda_i t>0$; the case when $x-\lambda_i t \leq 0$ can be treated in a similar manner. By integration by parts, \eqref{SectionIII:CNS_Lag_per_Fundamental_Solution_Modified_Explicit_Formula} and~\eqref{SectionIII:Lemma:Bound_of_Ii:Proof:Bound_of_etaj}, we have
  \begin{align}
    |\mathcal{I}_{i,1}(x,t)|
    & \leq C(t+1)^{-1}\int_{-\infty}^{(x-\lambda_i t)/2}e^{-\frac{(x-\lambda_i t)^2}{Ct}}|\eta_i(y)|\, dy \\
    & \quad +C\delta (t+1)^{-1}\int_{(x-\lambda_i t)/2}^{\infty}e^{-\frac{(x-y-\lambda_i t)^2}{Ct}}(y+1)^{-1}dy \\
    & \leq C\delta(t+1)^{-1}e^{-\frac{(x-\lambda_i t)^2}{Ct}}+C\delta (|x-\lambda_i t|+1)^{-1}(t+1)^{-1/2} \\
    & \leq C\delta(t+1)^{-1}e^{-\frac{(x-\lambda_i t)^2}{Ct}}+C\delta (|x-\lambda_i t|+1)^{-3/2} \\
    & \leq C\delta \Psi_i(x,t).
  \end{align}
  Case (iii): $|x-\lambda_i t|\geq t+1$. Again, let us only consider the case when $x-\lambda_i t>0$. By the assumptions of Theorem~\ref{SectionI:Theorem:MainTheorem:PointwiseEstimates} and~\eqref{SectionIII:CNS_Lag_per_Fundamental_Solution_Modified_Explicit_Formula}, we have
  \begin{align}
    |\mathcal{I}_{i,1}(x,t)|
    & \leq C(t+1)^{-1/2}\int_{-\infty}^{(x-\lambda_i t)/2}e^{-\frac{(x-\lambda_i t)^2}{Ct}}|v_{0i}(y)|\, dy \\
    & \quad +C\delta (t+1)^{-1/2}\int_{(x-\lambda_i t)/2}^{\infty}e^{-\frac{(x-y-\lambda_i t)^2}{Ct}}(y+1)^{-3/2}dy \\
    & \quad +C\delta (t+1)^{-1/2}e^{-\frac{(x-\lambda_i t)^2}{2\nu t}} \\
    & \leq C\delta (t+1)^{-1/2}e^{-\frac{(x-\lambda_i t)^2}{Ct}}+C\delta (|x-\lambda_i t|+1)^{-3/2}.
  \end{align}
  Since
  \begin{equation}
    e^{-\frac{(x-\lambda_i t)^2}{Ct}}\leq e^{-\frac{|x-\lambda_i t|}{C}},
  \end{equation}
  we obtain~\eqref{SectionIII:Lemma:Bound_of_Ii:Proof:Bound_of_Ii1}.

  We next show that
  \begin{equation}
    \label{SectionIII:Lemma:Bound_of_Ii:Proof:Bound_of_Ii2}
    |\mathcal{I}_{i,2}(x,t)|\leq C\delta \Psi_i(x,t).
  \end{equation}
  By the assumptions of Theorem~\ref{SectionI:Theorem:MainTheorem:PointwiseEstimates}, \eqref{SectionIII:Bound_of_G-G*} and~\eqref{SectionIII:Q0Q1}, we have
  \begin{align}
    |\mathcal{I}_{i,2}(x,t)|
    & \leq C\sum_{j=1}^{2}(t+1)^{-1}\int_{-\infty}^{\infty}e^{-\frac{(x-y-\lambda_j t)^2}{Ct}}\left|
    \begin{pmatrix}
      u_{01} \\
      u_{02}
    \end{pmatrix}
    \right| (y)\, dy+Ce^{-\frac{c^2}{\nu}t}|u_{01}(x)| \\
    & \leq C\sum_{j=1}^{2}(t+1)^{-1}\int_{-\infty}^{\infty}e^{-\frac{(x-y-\lambda_j t)^2}{Ct}}\left|
    \begin{pmatrix}
      u_{01} \\
      u_{02}
    \end{pmatrix}
    \right| (y)\, dy+C\delta e^{-\frac{c^2}{\nu}t}(|x|+1)^{-3/2}.
  \end{align}
  Let
  \begin{equation}
    I_j(x,t)\coloneqq (t+1)^{-1}\int_{-\infty}^{\infty}e^{-\frac{(x-y-\lambda_j t)^2}{Ct}}\left|
    \begin{pmatrix}
      u_{01} \\
      u_{02}
    \end{pmatrix}
    \right| (y)\, dy.
  \end{equation}
  To show~\eqref{SectionIII:Lemma:Bound_of_Ii:Proof:Bound_of_Ii2}, it suffices to show the following case-by-case:
  \begin{equation}
    |I_j(x,t)|\leq C\delta \tilde{\psi}(x,t;\lambda_j).
  \end{equation}
  Case (i): $|x-\lambda_j t|\leq (t+1)^{1/2}$. By the assumptions of Theorem~\ref{SectionI:Theorem:MainTheorem:PointwiseEstimates}, we have
  \begin{equation}
    |I_j(x,t)|\leq C\delta(t+1)^{-1}\leq C\delta \tilde{\psi}(x,t;\lambda_j).
  \end{equation}
  Case (ii): $(t+1)^{1/2}<|x-\lambda_j t|$. Suppose that $x-\lambda_j t>0$; the case when $x-\lambda_j t\leq 0$ can be treated in a similar manner. By the assumptions of Theorem~\ref{SectionI:Theorem:MainTheorem:PointwiseEstimates}, we have
  \begin{align}
    |I_j(x,t)|
    & \leq C(t+1)^{-1}\int_{-\infty}^{(x-\lambda_j t)/2}e^{-\frac{(x-\lambda_j t)^2}{Ct}}\left|
    \begin{pmatrix}
      u_{01} \\
      u_{02}
    \end{pmatrix}
    \right| (y)\, dy \\
    & \quad +C\delta(t+1)^{-1}\int_{(x-\lambda_j t)/2}^{\infty}e^{-\frac{(x-y-\lambda_j t)^2}{Ct}}(y+1)^{-3/2}dy \\
    & \leq C\delta(t+1)^{-1}e^{-\frac{(x-\lambda_j t)^2}{Ct}}+C\delta(|x-\lambda_j t|+1)^{-3/2}(t+1)^{-1/2} \\
    & \leq C\delta \tilde{\psi}(x,t;\lambda_j).
  \end{align}
  This proves~\eqref{SectionIII:Lemma:Bound_of_Ii:Proof:Bound_of_Ii2}.

  By using~\eqref{SectionIII:GbIntByParts} and~\eqref{SectionIII:Bounds_of_Gb}, we can show that
  \begin{equation}
    |\mathcal{I}_{i,3}(x,t)|\leq C\delta \Psi_i(x,t).
  \end{equation}
  The proof is similar to that of~\eqref{SectionIII:Lemma:Bound_of_Ii:Proof:Bound_of_Ii2} except that we also need to handle integrals of the form
  \begin{equation}
    \int_{0}^{\infty}e^{-\frac{|x+y|+t}{C}}|f(y)|\, dy, \quad \int_{-\infty}^{0}e^{-\frac{|x-y|+t}{C}}|f(y)|\, dy,
  \end{equation}
  where $f$ is a function satisfying $|f(x)|\leq C\delta$. These are easily seen to be bounded by
  \begin{equation}
    C\delta e^{-\frac{|x|+t}{C}}\int_{-\infty}^{\infty}e^{-\frac{|y|}{C}}\, dy\leq C\delta e^{-\frac{|x|+t}{C}}\leq C\delta \Psi_i(x,t).
  \end{equation}

  Finally, we need to show that
  \begin{equation}
    \label{SectionIII:Lemma:Bound_of_Ii:Proof:Bound_of_Ii3}
    |\mathcal{I}_{i,4}(x,t)|\leq C\delta \Psi_i(x,t).
  \end{equation}
  Since
  \begin{equation}
    g_{T,i}-g_{i}^{*}=(g_{T,i}-g_i)+(g_i-g_{i}^{*}),
  \end{equation}
  \eqref{SectionIII:Lemma:Bound_of_Ii:Proof:Bound_of_Ii3} follows from~\eqref{SectionIII:Bound_of_G-G*}, \eqref{SectionIII:GbIntByParts} and~\eqref{SectionIII:Bounds_of_Gb}.
\end{proof}

\begin{remark}
  \label{Remark:Nonlinearity_is_Essential}
  A careful look at the proof above reveals the following: If we impose stronger and stronger spatial decay conditions on $u_{0i}$ ($i=1,2$), then $\mathcal{I}_i(\pm 0,t)$ ($i=1,2$) decay faster and faster. For example, if we assume that $u_{0i}$ ($i=1,2$) decay exponentially, then $\mathcal{I}_i(\pm 0,t)$ ($i=1,2$) also decay exponentially; therefore, the decay rate $-3/2$ in the decay estimate $|V(t)|=O(t^{-3/2})$ (Corollary~\ref{Corollary:DecayEstimateV}) comes from the nonlinear contribution $\mathcal{N}_i(\pm 0,t)$; see Lemma~\ref{SectionIII:Lemma:Bound_of_Ni}.
\end{remark}

We next give a bound for the nonlinear term $\mathcal{N}_i(x,t)$. By integration by parts, we have
\begin{align}
  \mathcal{N}_i(x,t)
  & =\int_{0}^{t}\int_{-\infty}^{\infty}g_{i}^{*}(x-y,t-s)
  \begin{pmatrix}
    N_1-N_{1}^{*} \\
    N_2-N_{2}^{*}
  \end{pmatrix}_x
  (y,s)\, dyds \\
  & \quad +\int_{0}^{t}\int_{-\infty}^{\infty}\partial_x(g_i-g_{i}^{*})(x-y,t-s)
  \begin{pmatrix}
    N_1 \\
    N_2
  \end{pmatrix}
  (y,s)\, dyds \\
  & \quad -\int_{0}^{t}\int_{0}^{\infty}\partial_x g_{R,i}(x+y,t-s)
  \begin{pmatrix}
    N_1 \\
    N_2
  \end{pmatrix}
  (y,s)\, dyds \\
  & \quad +\int_{0}^{t}\int_{-\infty}^{0}\partial_x g_{R,i}(x-y,t-s)
  \begin{pmatrix}
    N_1 \\
    N_2
  \end{pmatrix}
  (y,s)\, dyds \\
  & \quad +\int_{0}^{t}g_{i}^{*}(x,t-s)
  \begin{pmatrix}
    \llbracket N_1 \rrbracket \\
    \llbracket N_2 \rrbracket
  \end{pmatrix}
  (0,s)\, ds.
\end{align}
We note that
\begin{equation}
  N_i=-\frac{1}{2}\left( \theta_{1}^{2}+\theta_{2}^{2} \right)+\frac{\nu c}{p''(1)}\left( \theta_1 \theta_{1x}-\theta_2 \theta_{2x} \right)+R_i,
\end{equation}
where
\begin{align}
  \label{SectionIII:Bound_of_Ri}
  \begin{aligned}
    & |R_i-\left( -v_1 \theta_1+v_1 \theta_2+v_2 \theta_1-v_2 \theta_2  \right)| \\
    & \leq C\sum_{j,k=1}^{2}|v_j|\left( |v_k|+|v_{kx}|+|\theta_{kx}| \right)+C\sum_{j\neq k}|\theta_j|\left( |\theta_k|+|\theta_{kx}| \right) \\
    & \quad +C\sum_{j,k=1}^{2}|v_{jx}||\theta_k|+C\sum_{j,k,l=1}^{2}\left( |v_j|+|\theta_j|+|\theta_{jx}| \right) |\theta_k||\theta_l|.
  \end{aligned}
\end{align}
Here, we used the formulae $\tau=(2c/p''(1))(-\theta_1-v_1+\theta_2+v_2)$ and $u=(2c^2/p''(1))(\theta_1+v_1+\theta_2+v_2)$. Since $\theta_i$ is smooth, we have $\llbracket N_i-N_{i}^{*} \rrbracket(0,t)=\llbracket N_i \rrbracket(0,t)=\llbracket R_i \rrbracket(0,t)$. Thus, by integration by parts,
\begin{align}
  \label{SectionIII:NonlinearTerms}
  \begin{aligned}
    \mathcal{N}_i(x,t)
    & =-\frac{1}{2}\int_{0}^{t}\int_{-\infty}^{\infty}g_{i}^{*}(x-y,t-s)
    \begin{pmatrix}
      \theta_{2}^{2} \\
      \theta_{1}^{2}
    \end{pmatrix}_x
    (y,s)\, dyds \\
    & \quad +\frac{\nu c}{2p''(1)}\int_{0}^{t}\int_{-\infty}^{\infty}g_{i}^{*}(x-y,t-s)
    \begin{pmatrix}
      \theta_{1}^{2}-\theta_{2}^{2} \\
      \theta_{1}^{2}-\theta_{2}^{2}
    \end{pmatrix}_{xx}
    (y,s)\, dyds \\
    & \quad +\int_{0}^{t}\int_{-\infty}^{\infty}\partial_x g_{i}^{*}(x-y,t-s)
    \begin{pmatrix}
      R_1 \\
      R_2
    \end{pmatrix}
    (y,s)\, dyds \\
    & \quad +\int_{0}^{t}\int_{-\infty}^{\infty}\partial_x(g_i-g_{i}^{*})(x-y,t-s)
    \begin{pmatrix}
      N_1 \\
      N_2
    \end{pmatrix}
    (y,s)\, dyds \\
    & \quad -\int_{0}^{t}\int_{0}^{\infty}\partial_x g_{R,i}(x+y,t-s)
    \begin{pmatrix}
      N_1 \\
      N_2
    \end{pmatrix}
    (y,s)\, dyds \\
    & \quad +\int_{0}^{t}\int_{-\infty}^{0}\partial_x g_{R,i}(x-y,t-s)
    \begin{pmatrix}
      N_1 \\
      N_2
    \end{pmatrix}
    (y,s)\, dyds.
  \end{aligned}
\end{align}

\begin{lemma}
  \label{SectionIII:Lemma:Bound_of_Ni}
  Under the assumptions of Theorem~\ref{SectionI:Theorem:MainTheorem:PointwiseEstimates}, there exists a positive constant $C$ such that
  \begin{equation}
    \label{SectionIII:Lemma:Bound_of_Ni:Bound_of_Ni}
    |\mathcal{N}_i(x,t)|\leq C(\delta+P(t))^2 \Psi_i(x,t)
  \end{equation}
  for $(x,t)\in \mathbb{R}_* \times (0,\infty)$.
\end{lemma}

\begin{proof}
  We treat the case when $x>0$; the case when $x<0$ is similar. Let $\mathcal{N}_{i,1}(x,t)$ be the first term on the right-hand side of~\eqref{SectionIII:NonlinearTerms}. We show that
  \begin{equation}
    \label{SectionIII:Lemma:Bound_of_Ni1}
    |\mathcal{N}_{i,1}(x,t)|\leq C\delta^2 \Psi_i(x,t).
  \end{equation}
  By~Lemma~\ref{LZlemma:3.4} ($\alpha=2$, $k=0$), we have for $j\neq i$,
  \begin{align}
    & \left| \int_{0}^{t}\int_{-\infty}^{\infty}(t-s)^{-1/2}e^{-\frac{(x-y-\lambda_i(t-s))^2}{2\nu(t-s)}}(\theta_{j}^{2})_x(y,s)\, dyds \right| \\
    & \leq C\delta^2 [\psi_{3/2}(x,t;\lambda_i)+\Theta_2(x,t;\lambda_j,\nu^*) \\
    & \phantom{\leq C\delta^2 [} \quad +|x-\lambda_i(t+1)|^{-1}|x-\lambda_j(t+1)|^{-1/2}\chi_K(x,t)],
  \end{align}
  where $\nu^*$, $K$ are large positive constants, $\Theta_{\alpha}(x,t;\lambda,\mu)$ is defined by~\eqref{Def:Theta} and
  \begin{equation}
    \chi_K(x,t)\coloneqq \mathrm{char}\left\{ -c(t+1)+K(t+1)^{1/2}\leq x\leq c(t+1)-K(t+1)^{1/2} \right\}.
  \end{equation}
  Here, $\mathrm{char}\{ S \}$ is the characteristic function of a set $S$. We note that
  \begin{equation}
    |x-\lambda_i(t+1)|^{-1}|x-\lambda_j(t+1)|^{-1/2}\chi_K(x,t)\leq C\psi_{3/2}(x,t;\lambda_i)+C[|x-\lambda_j(t+1)|^3+(t+1)^{5/2}]^{-1/2}\leq C\Psi_i(x,t),
  \end{equation}
  which is easily shown by considering the cases $|x-\lambda_j(t+1)|\lesseqgtr c(t+1)$. The bound~\eqref{SectionIII:Lemma:Bound_of_Ni1} then follows from these and~\eqref{SectionIII:gi*_Explicit_Formula}.

  Next, let $\mathcal{N}_{i,2}(x,t)$ be the second term on the right-hand side of~\eqref{SectionIII:NonlinearTerms}. We show that
  \begin{equation}
    \label{SectionIII:Lemma:Bound_of_Ni2}
    |\mathcal{N}_{i,2}(x,t)|\leq C\delta^2 \Psi_i(x,t).
  \end{equation}
  By~Lemma~\ref{LZlemma:3.4} ($\alpha=2$, $k=1$), we have for $j\neq i$,
  \begin{align}
    & \left| \int_{0}^{t}\int_{-\infty}^{\infty}(t-s)^{-1/2}e^{-\frac{(x-y-\lambda_i(t-s))^2}{2\nu(t-s)}}(\theta_{j}^{2})_{xx}(y,s)\, dyds \right| \\
    & \leq C\delta^2(t+1)^{-1/2}[\psi_{3/2}(x,t;\lambda_i)+\Theta_2(x,t;\lambda_j,\nu^*)+|x-\lambda_i(t+1)|^{-1}|x-\lambda_j(t+1)|^{-1/2}\chi_K(x,t)] \\
    & \leq C\delta^2(t+1)^{-1/2}\Psi_i(x,t).
  \end{align}
  By~Lemma~\ref{LZlemma:3.2} ($\alpha=0$, $\beta=3$),
  \begin{align}
    & \left| \int_{0}^{t}\int_{-\infty}^{\infty}(t-s)^{-1/2}e^{-\frac{(x-y-\lambda_i(t-s))^2}{2\nu(t-s)}}(\theta_{i}^{2})_{xx}(y,s)\, dyds \right| \\
    & \leq C\int_{0}^{t}\int_{-\infty}^{\infty}(t-s)^{-1}e^{-\frac{(x-y-\lambda_i(t-s))^2}{C(t-s)}}|(\theta_i \theta_{ix})(y,s)|\, dyds \\
    & \leq C\delta^2 \Theta_2(x,t;\lambda_i,\nu^*)\log(t+2)\leq C\delta^2 \Psi_i(x,t).
  \end{align}
  These and~\eqref{SectionIII:gi*_Explicit_Formula} show~\eqref{SectionIII:Lemma:Bound_of_Ni2}.
  
  Next, let $\mathcal{N}_{i,3}(x,t)$ be the third term on the right-hand side of~\eqref{SectionIII:NonlinearTerms}. We show that
  \begin{equation}
    \label{SectionIII:Lemma:Bound_of_Ni3}
    |\mathcal{N}_{i,3}(x,t)|\leq C(\delta+P(t))^2 \Psi_i(x,t).
  \end{equation}
  Let us first consider the terms involving $v_j \theta_k$ appearing on the left-hand side of~\eqref{SectionIII:Bound_of_Ri}. Since $|v_i(x,t)|\leq CP(t)(t+1)^{-3/4}$, by~Lemma~\ref{LZlemma:3.2} ($\alpha=0$, $\beta=5/2$), we have for any $j$,
  \begin{equation}
    \left| \int_{0}^{t}\int_{-\infty}^{\infty}(t-s)^{-1}e^{-\frac{(x-y-\lambda_i(t-s))^2}{C(t-s)}}(\theta_i v_j)(y,s)\, dyds \right| \leq C\delta P(t)\Theta_{3/2}(x,t;\lambda_i,\nu^*) \leq C\delta P(t)\Psi_i(x,t).
  \end{equation}
  Next, note that $|\tilde{\psi}(x,t;\lambda_j)|\leq C(t+1)^{-1}$. By Lemmas~\ref{LZlemma:3.3} ($\alpha=0$, $\beta=3$) and~\ref{LZlemma:3.5} ($\alpha=0$, $\beta=1$), we have for $j\neq i$,
  \begin{align}
    & \left| \int_{0}^{t}\int_{-\infty}^{\infty}(t-s)^{-1}e^{-\frac{(x-y-\lambda_i(t-s))^2}{C(t-s)}}(v_i \theta_j)(y,s)\, dyds \right| \\
    & \leq C\delta P(t)\int_{0}^{t}\int_{-\infty}^{\infty}(t-s)^{-1}e^{-\frac{(x-y-\lambda_i(t-s))^2}{C(t-s)}}(s+1)^{-1/2}\psi_{3/2}(y,s;\lambda_i)\, dyds \\
    & \quad +C\delta P(t)\int_{0}^{t}\int_{-\infty}^{\infty}(t-s)^{-1}e^{-\frac{(x-y-\lambda_i(t-s))^2}{C(t-s)}}\Theta_{3}(y,s;\lambda_j,2\nu)\, dyds \\
    & \leq C\delta P(t)[\psi_{3/2}(x,t;\lambda_i)+\Theta_2(x,t;\lambda_i,\nu^*)\log(t+2)+\Theta_2(x,t;\lambda_j,\nu^*) \\
    & \phantom{\leq C\delta P(t)[} \quad +|x-\lambda_i(t+1)|^{-1}|x-\lambda_j(t+1)|^{-1/2}\chi_K(x,t)] \\
    & \leq C\delta P(t)\Psi_i(x,t).
  \end{align}
  Next, let $L_j \coloneqq \partial_t+\lambda_j \partial_x-(\nu/2)\partial_{x}^{2}$. Then
  \begin{equation}
    L_j\left( \theta_j v_j \right)=v_j L_j \theta_j+\theta_j L_j u_j-\theta_j L_j \theta_j -\nu \theta_{jx}v_{jx}.
  \end{equation}
  Note that $L_j \theta_j=-(\theta_{j}^{2}/2)_x$ and
  \begin{align}
    \theta_j L_j u_j
    & =(\nu/2)\theta_j u_{j'xx}+\theta_j N_{jx} \\
    & =(\nu/2)(\theta_j u_{j'x})_x-(\nu/2)\theta_{jx}u_{j'x}+(\theta_j N_j)_x-\theta_{jx}N_j,
  \end{align}
  where $j'=3-j$. Since $|N_j(x,t)|\leq C(\delta+P(t))^2(t+1)^{-1}$, we can apply~Lemma~\ref{LZlemma:3.4} ($\alpha=2$, $k=0$) to obtain, for $j\neq i$,
  \begin{align}
    & \left| \int_{0}^{t}\int_{-\infty}^{\infty}\partial_x \left\{ (t-s)^{-1/2}e^{-\frac{(x-y-\lambda_i(t-s))^2}{2\nu(t-s)}} \right\} (v_j \theta_j)(y,s)\, dyds \right| \\
    & \leq C(\delta+P(t))^2 [\psi_{3/2}(x,t;\lambda_i)+\Theta_2(x,t;\lambda_j,\nu^*)+|x-\lambda_i(t+1)|^{-1}|x-\lambda_j(t+1)|^{-1/2}\chi_K(x,t)] \\
    & \leq C(\delta+P(t))^2 \Psi_i(x,t).
  \end{align}
  Next, let us consider the terms involving $v_j v_k$ appearing on the right-hand side of~\eqref{SectionIII:Bound_of_Ri}. Since
  \begin{align}
    & \left| \int_{0}^{t}\int_{-\infty}^{\infty}(t-s)^{-1}e^{-\frac{(x-y-\lambda_i(t-s))^2}{C(t-s)}}(v_j v_k)(y,s)\, dyds \right| \\
    & \leq C(\delta+P(t))^2 \sum_{l,m=1}^{2}\int_{0}^{t}\int_{-\infty}^{\infty}(t-s)^{-1}e^{-\frac{(x-y-\lambda_i(t-s))^2}{C(t-s)}}\psi_{3/2}(y,s;\lambda_l)\psi_{3/2}(y,s;\lambda_m)\, dyds,
  \end{align}
  we can apply~Lemmas~\ref{LZlemma:3.5} and~\ref{LZlemma:3.6} ($\alpha=0$, $\beta=3/2$) to obtain
  \begin{align}
    & \left| \int_{0}^{t}\int_{-\infty}^{\infty}(t-s)^{-1}e^{-\frac{(x-y-\lambda_i(t-s))^2}{C(t-s)}}(v_j v_k)(y,s)\, dyds \right| \\
    & \leq C(\delta+P(t))^2(t+1)^{-1/4}[\log(t+2)\psi_{3/2}(x,t;\lambda_i)+\psi_{3/2}(x,t;\lambda_{i'})] \\
    & \quad +C(\delta+P(t))^2 |x-\lambda_i(t+1)|^{-1}|x-\lambda_{i'}(t+1)|^{-1/2}\chi_K(x,t) \\
    & \leq C(\delta+P(t))^2 \Psi_i(x,t),
  \end{align}
  where $i'=3-i$. By conducting similar calculations, we can treat other integrals involving the terms appearing on the right-hand side of~\eqref{SectionIII:Bound_of_Ri}. Thus, we obtain~\eqref{SectionIII:Lemma:Bound_of_Ni3}.
  
  Next, let $\mathcal{N}_{i,4}(x,t)$ be the fourth term on the right-hand side of~\eqref{SectionIII:NonlinearTerms}. We show that
  \begin{equation}
    \label{SectionIII:Lemma:Bound_of_Ni4}
    |\mathcal{N}_{i,4}(x,t)|\leq C(\delta+P(t))^2 \Psi_i(x,t).
  \end{equation}
  Note first that
  \begin{equation}
    \label{SectionIII:Lemma:Q0N}
    Q_0
    \begin{pmatrix}
      r_1 & r_2
    \end{pmatrix}
    \begin{pmatrix}
      N_1 \\
      N_2
    \end{pmatrix}
    =Q_0
    \begin{pmatrix}
      r_1 & r_2
    \end{pmatrix}
    \begin{pmatrix}
      l_1 \\
      l_2
    \end{pmatrix}
    \begin{pmatrix}
      0 \\
      N
    \end{pmatrix}
    =Q_0
    \begin{pmatrix}
      0 \\
      N
    \end{pmatrix}
    =0,
  \end{equation}
  where $Q_0$ is given by~\eqref{SectionIII:Q0Q1}. Let $N_i=N_{i}^{a}+N_{i}^{b}$, where
  \begin{equation}
    \label{Nia_Nib}
    N_{i}^{a}\coloneqq l_i
    \begin{pmatrix}
      0 \\
      -p(1+\tau)+p(1)-c^2 \tau
    \end{pmatrix}
    , \quad N_{i}^{b}\coloneqq l_i
    \begin{pmatrix}
      0 \\
      -\nu\frac{\tau}{1+\tau}u_x
    \end{pmatrix}.
  \end{equation}
  Denote by $\mathcal{N}_{i,4}^{a}(x,t)$ and $\mathcal{N}_{i,4}^{b}(x,t)$ the corresponding terms in $\mathcal{N}_{i,4}(x,t)$. Let us first consider $\mathcal{N}_{i,4}^{a}(x,t)$. Divide the domain of integration in $t$ into $[0,t/2]$ and $[t/2,t]$, and denote by $\mathcal{N}_{i,4,1}^{a}(x,t)$ and $\mathcal{N}_{i,4,2}^{a}(x,t)$ the corresponding terms in $\mathcal{N}_{i,4}^{a}(x,t)$. Since
  \begin{equation}
    |N_{i}^{a}(x,t)|\leq C(\delta+P(t))^2(t+1)^{-1/2}\sum_{j=1}^{2}\left( \Theta_1(x,t;\lambda_j,2\nu)+\psi_{3/2}(x,t;\lambda_j) \right),
  \end{equation}
  by~\eqref{SectionIII:Bound_of_G-G*}, \eqref{SectionIII:Lemma:Q0N}, Lemmas~\ref{LZlemma:3.2} ($\alpha=1$, $\beta=2$), \ref{LZlemma:3.3} ($\alpha=1$, $\beta=2$; Remark~\ref{Remark:LZlemma:3.3}), \ref{LZlemma:3.5} ($\alpha=1$, $\beta=1$), \ref{LZlemma:3.6} ($\alpha=1$, $\beta=1$) and~\ref{LZlemma:3.9}, we have
  \begin{align}
    |\mathcal{N}_{i,4,1}^{a}(x,t)|
    & \leq C(\delta+P(t))^2 \sum_{j,k=1}^{2}\int_{0}^{t/2}\int_{-\infty}^{\infty}(t-s)^{-1}(t+1-s)^{-1/2}e^{-\frac{(x-y-\lambda_j(t-s))^2}{C(t-s)}} \\
    & \quad \cdot (s+1)^{-1/2}\left( \Theta_1(y,s;\lambda_k,2\nu)+\psi_{3/2}(y,s;\lambda_k) \right) \, dyds \\
    & \quad +C\sum_{j=1}^{2}\int_{0}^{t/2}e^{-\frac{c^2}{\nu}(t-s)}|N_{j}^{a}(x,s)|\, ds \\
    & \leq C(\delta+P(t))^2 \sum_{j=1}^{2}\Theta_2(x,t;\lambda_j,\nu^*) \\
    & \quad +C(\delta+P(t))^2 \sum_{j\neq k}|x-\lambda_j(t+1)|^{-1}|x-\lambda_k(t+1)|^{-1/2}\chi_K(x,t;\lambda_j,\lambda_k) \\
    & \quad +C(\delta+P(t))^2(t+1)^{-1/2}\log(t+2)\sum_{j=1}^{2}\psi_{3/2}(x,t;\lambda_j) \\
    & \quad +C(\delta+P(t))^2 \sum_{j=1}^{2}\int_{0}^{t/2}e^{-\frac{c^2}{\nu}(t-s)}(s+1)^{-1/4}\psi_{3/2}(x,s;\lambda_j)\, ds \\
    & \leq C(\delta+P(t))^2 \Psi_i(x,t).
  \end{align}
  On the other hand, by integration by parts,
  \begin{align}
    \mathcal{N}_{i,4,2}^{a}(x,t)
    & =\int_{t/2}^{t}\int_{-\infty}^{\infty}(g_i-g_{i}^{*})(x-y,t-s)
    \begin{pmatrix}
      N_{1}^{a} \\
      N_{2}^{a}
    \end{pmatrix}_x
    (y,s)\, dyds \\
    & \quad +\int_{t/2}^{t}(g_i-g_{i}^{*})(x,t-s)
    \begin{pmatrix}
      \llbracket N_{1}^{a} \rrbracket \\
      \llbracket N_{2}^{a} \rrbracket
    \end{pmatrix}
    (0,s)\, ds.
  \end{align}
  Since
  \begin{equation}
    |N_{ix}^{a}(x,t)|\leq C(\delta+P(t))^2(t+1)^{-1}\sum_{j=1}^{2}\left( \Theta_1(x,t;\lambda_j,2\nu)+\psi_{3/2}(x,t;\lambda_j) \right)
  \end{equation}
  and
  \begin{equation}
    |N_{i}^{a}(\pm 0,t)|\leq C(\delta+P(t))^2(t+1)^{-3},
  \end{equation}
  by~\eqref{SectionIII:Bound_of_G-G*}, \eqref{SectionIII:Lemma:Q0N}, Lemmas~\ref{LZlemma:3.2} ($\alpha=0$, $\beta=3$), \ref{LZlemma:3.3} ($\alpha=0$, $\beta=3$; Remark~\ref{Remark:LZlemma:3.3}), \ref{LZlemma:3.5} ($\alpha=0$, $\beta=2$) and \ref{LZlemma:3.6} ($\alpha=0$, $\beta=2$), we have
  \begin{align}
    \label{SectionIII:Bound_of_Ni42a}
    \begin{aligned}
      & |\mathcal{N}_{i,4,2}^{a}(x,t)| \\
      & \leq C(\delta+P(t))^2 \\
      & \quad \cdot \sum_{j,k=1}^{2}\int_{t/2}^{t}\int_{-\infty}^{\infty}(t-s)^{-1}e^{-\frac{(x-y-\lambda_j(t-s))^2}{C(t-s)}}(s+1)^{-1}\left( \Theta_1(y,s;\lambda_k,2\nu)+\psi_{3/2}(y,s;\lambda_k) \right) \, dyds \\
      & \quad +C(\delta+P(t))^2 \sum_{j=1}^{2}\int_{t/2}^{t}(t-s)^{-1/2}(t+1-s)^{-1/2}e^{-\frac{(x-\lambda_j(t-s))^2}{C(t-s)}}(s+1)^{-3}\, ds \\
      & \leq C(\delta+P(t))^2 \sum_{j=1}^{2}\Theta_2(x,t;\lambda_j,\nu^*) \\
      & \quad +C(\delta+P(t))^2 \sum_{j\neq k}|x-\lambda_j(t+1)|^{-1}|x-\lambda_k(t+1)|^{-1/2}\chi_K(x,t;\lambda_j,\lambda_k) \\
      & \quad +C(\delta+P(t))^2(t+1)^{-1/4}\sum_{j=1}^{2}\psi_{3/2}(x,t;\lambda_j) \\
      & \quad +C(\delta+P(t))^2 \sum_{j=1}^{2}\int_{t/2}^{t}(t-s)^{-1/2}e^{-\frac{(x-\lambda_j(t-s))^2}{C(t-s)}}(s+1)^{-3}\, ds \\
      & \leq C(\delta+P(t))^2 \Psi_i(x,t)+C(\delta+P(t))^2 \sum_{j=1}^{2}\int_{t/2}^{t}(t-s)^{-1/2}e^{-\frac{(x-\lambda_j(t-s))^2}{C(t-s)}}(s+1)^{-3}\, ds.
    \end{aligned}
  \end{align}
  The last term is bounded as follows. Case (i): $|x-\lambda_j t|\leq 2(t+1)^{1/2}$. In this case, it is bounded by
  \begin{align}
    C(\delta+P(t))^2(t+1)^{-3}\int_{t/2}^{t}(t-s)^{-1/2}\, ds
    & \leq C(\delta+P(t))^2(t+1)^{-5/2} \\
    & \leq C(\delta+P(t))^2 \Psi_i(x,t).
  \end{align}
  Case (ii): $|x-\lambda_j t|>2(t+1)^{1/2}$. Let
  \begin{equation}
    A_1 \coloneqq \{ t/2\leq s\leq t \mid cs\leq |x-\lambda_j t|/2 \}, \quad A_2 \coloneqq \{ t/2\leq s\leq t \mid cs>|x-\lambda_j t|/2 \}.
  \end{equation}
  If $s\in A_1$, we have
  \begin{equation}
    |x-\lambda_j(t-s)|\geq |x-\lambda_j t|/2;
  \end{equation}
  and if $s\in A_2$, we have
  \begin{equation}
    (s+1)^{-3}\leq C|x-\lambda_j t|^{-3}.
  \end{equation}
  Thus, the last term in~\eqref{SectionIII:Bound_of_Ni42a} is bounded by
  \begin{align}
    C(\delta+P(t))^2 \sum_{j=1}^{2}\left[ (t+1)^{-5/2}e^{-\frac{(x-\lambda_j t)^2}{Ct}}+(t+1)^{-1/2}|x-\lambda_j t|^{-3} \right] \leq C(\delta+P(t))^2 \Psi_i(x,t).
  \end{align}
  Therefore, we have
  \begin{equation}
    |\mathcal{N}_{i,4,2}^{a}(x,t)|\leq C(\delta+P(t))^2 \Psi_i(x,t).
  \end{equation}
  $\mathcal{N}_{i,4}^{b}(x,t)$ can be treated in a way similar to $\mathcal{N}_{i,4,1}^{a}(x,t)$ without splitting the domain of integration in $t$.

  Let $\mathcal{N}_{i,5}(x,t)$ be the fifth term on the right-hand side of~\eqref{SectionIII:NonlinearTerms}. This can be handled in a way similar to $\mathcal{N}_{i,4}(x,t)$ except that we need to consider integrals of the form
  \begin{equation}
    \int_{0}^{t/2}\int_{0}^{\infty}e^{-\frac{|x+y|+t-s}{C}}|N_j(y,s)|\, dyds,
  \end{equation}
  \begin{equation}
    \int_{t/2}^{t}\int_{0}^{\infty}e^{-\frac{|x+y|+t-s}{C}}|N_{j}^{b}(y,s)|\, dyds, \quad \int_{t/2}^{t}\int_{0}^{\infty}e^{-\frac{|x+y|+t-s}{C}}|N_{jx}^{a}(y,s)|\, dyds
  \end{equation}
  and
  \begin{equation}
    \int_{t/2}^{t}e^{-\frac{|x|+t-s}{C}}|N_{j}^{a}(+0,s)|\, ds.
  \end{equation}
  By using
  \begin{equation}
    |N_{j}^{b}(x,t)|, \quad |N_{jx}^{a}(x,t)|\leq C(\delta+P(t))^2(t+1)^{-3/2}, \quad |N_{j}^{a}(+0,t)|\leq C(\delta+P(t))^2(t+1)^{-3},
  \end{equation}
  these are easily seen to be bounded by
  \begin{equation}
    C(\delta+P(t))^2(t+1)^{-3/2}e^{-\frac{|x|}{C}}\leq C(\delta+P(t))^2 \Psi_i(x,t).
  \end{equation}
  The last term on the right-hand side of~\eqref{SectionIII:NonlinearTerms} can be treated completely analogous to $\mathcal{N}_{i,5}(x,t)$.

  This completes the proof that $|\mathcal{N}_i(x,t)|\leq C(\delta+P(t))^2 \Psi_i(x,t)$.
\end{proof}

Next, we give a bound for $\partial_x v_i(x,t)$.

\begin{lemma}
  \label{SectionIII:Lemma:Bound_of_Derivative_of_vi}
  Under the assumptions of Theorem~\ref{SectionI:Theorem:MainTheorem:PointwiseEstimates}, there exists a positive constant $C$ such that
  \begin{equation}
    |\partial_x v_i(x,t)|\leq C\delta(t+1)^{-1}+C(\delta+P(t))^2(t+1)^{-1}
  \end{equation}
  for $(x,t)\in \mathbb{R}_* \times (0,\infty)$.
\end{lemma}

\begin{proof}
  Since
  \begin{equation}
    |\partial_x \theta_i(x,t)|\leq C\delta(t+1)^{-1},
  \end{equation}
  it is enough to show that
  \begin{equation}
    \label{SectionIII:Lemma:Proof:Bound_of_Derivative_of_vi:Bound_of_Derivative_of_ui}
    |\partial_x u_i(x,t)|\leq C\delta(t+1)^{-1}+C(\delta+P(t))^2 (t+1)^{-1}.
  \end{equation}
  By Theorem~\ref{SectionI:Theorem:MainTheorem:GlobalExistence}, \eqref{SectionIII:Lemma:Proof:Bound_of_Derivative_of_vi:Bound_of_Derivative_of_ui} clearly holds for $t\leq 2$; therefore, we assume that $t\geq 2$ in the following. We also assume that $x>0$ since the case when $x<0$ is similar.

  By the assumptions of Theorem~\ref{SectionI:Theorem:MainTheorem:PointwiseEstimates}, \eqref{SectionIII:CNS_Lag_per_Fundamental_Solution_Modified_Explicit_Formula} and~\eqref{SectionIII:Bound_of_G-G*}, the first derivative of the first term on the right-hand side of~\eqref{SectionIII:uiIntegralEquation} is bounded as follows:
  \begin{align}
    \left| \int_{0}^{\infty}\partial_x g_i(x-y,t)
    \begin{pmatrix}
      u_{01} \\
      u_{02}
    \end{pmatrix}
    (y)\, dy \right|
    & \leq C(t+1)^{-1}\sum_{j=1}^{2}||u_{0j}||_{L^1(\mathbb{R}_*)}+Ce^{-\frac{t}{C}}\sum_{j=1}^{2}||u_{0j}||_2 \\
    & \leq C\delta(t+1)^{-1}.
  \end{align}
  Here, we used the Sobolev embedding theorem. Similar calculations show that the first derivatives of the second to fourth terms on the right-hand side of~\eqref{SectionIII:uiIntegralEquation} are also bounded by $C\delta(t+1)^{-1}$.

  Next, let us consider the first derivative of the sum of the last four terms on the right-hand side of~\eqref{SectionIII:uiIntegralEquation}. Split the domain of integration in $t$ into $[0,t/2]$, $[t/2,t-1]$ and $[t-1,t]$, and denote by $A_i(x,t)$, $B_i(x,t)$ and $C_i(x,t)$ the corresponding terms.
  
  Let us first consider $A_i(x,t)$. Integration by parts gives
  \begin{align}
    A_i(x,t)
    & =\int_{0}^{t/2}\int_{-\infty}^{\infty}\partial_{x}^{2}g_i(x-y,t-s)
    \begin{pmatrix}
      N_1 \\
      N_2
    \end{pmatrix}
    (y,s)\, dyds \\
    & \quad -\int_{0}^{t/2}\int_{0}^{\infty}\partial_{x}^{2}g_{R,i}(x+y,t-s)
    \begin{pmatrix}
      N_1 \\
      N_2
    \end{pmatrix}
    (y,s)\, dyds \\
    & \quad +\int_{0}^{t/2}\int_{-\infty}^{0}\partial_{x}^{2}g_{R,i}(x-y,t-s)
    \begin{pmatrix}
      N_1 \\
      N_2
    \end{pmatrix}
    (y,s)\, dyds.
  \end{align}
  Note that
  \begin{equation}
    \label{SectionIII:Lemma:Proof:Bound_of_Derivative_of_vi:NkL1}
    ||N_k(\cdot,t)||_{L^1(\mathbb{R}_*)}\leq C(\delta+P(t))^2(t+1)^{-1/2}.
  \end{equation}
  By Theorem~\ref{SectionI:Theorem:MainTheorem:GlobalExistence}, \eqref{SectionIII:CNS_Lag_per_Fundamental_Solution_Modified_Explicit_Formula}, \eqref{SectionIII:Bound_of_G-G*}, \eqref{SectionIII:GbIntByParts}, \eqref{SectionIII:Bounds_of_Gb}, \eqref{SectionIII:Lemma:Q0N} and~\eqref{SectionIII:Lemma:Proof:Bound_of_Derivative_of_vi:NkL1},
  \begin{align}
    |A_i(x,t)|
    & \leq C(\delta+P(t))^2(t+1)^{-3/2}\int_{0}^{t/2}(s+1)^{-1/2}\, ds+C\int_{0}^{t/2}e^{-\frac{t-s}{C}}\, ds\sum_{k=1}^{2}\sup_{0\leq s\leq t/2}||N_k(\cdot,s)||_2 \\
    & \leq C(\delta+P(t))^2(t+1)^{-1}.
  \end{align}

  We next consider $B_i(x,t)$, which is defined by
  \begin{align}
    B_i(x,t)
    & =\int_{t/2}^{t-1}\int_{-\infty}^{\infty}\partial_x g_i(x-y,t-s)
    \begin{pmatrix}
      N_1 \\
      N_2
    \end{pmatrix}_x
    (y,s)\, dyds \\
    & \quad +\int_{t/2}^{t-1}\int_{0}^{\infty}\partial_x g_{R,i}(x+y,t-s)
    \begin{pmatrix}
      N_1 \\
      N_2
    \end{pmatrix}_x
    (y,s)\, dyds \\
    & \quad +\int_{t/2}^{t-1}\int_{-\infty}^{0}\partial_x g_{R,i}(x-y,t-s)
    \begin{pmatrix}
      N_1 \\
      N_2
    \end{pmatrix}_x
    (y,s)\, dyds \\
    & \quad +\int_{t/2}^{t-1}\partial_x g_{T,i}(x,t-s)
    \begin{pmatrix}
      \llbracket N_1 \rrbracket \\
      \llbracket N_2 \rrbracket
    \end{pmatrix}
    (0,s)\, ds.
  \end{align}
  Remember that $N_i=N_{i}^{a}+N_{i}^{b}$, where $N_{i}^{a}$ and $N_{i}^{b}$ are given by~\eqref{Nia_Nib}. Denote by $B_{i}^{a}(x,t)$ and $B_{i}^{b}(x,t)$ the corresponding terms in $B_i(x,t)$. Note that
  \begin{equation}
    \label{SectionIII:Lemma:Proof:Bound_of_Derivative_of_vi:Bound_of_Nixa}
    |N_{ix}^{a}(x,t)|\leq C(\delta+P(t))^2(t+1)^{-3/2}
  \end{equation}
  and
  \begin{equation}
    \label{SectionIII:Lemma:Proof:Bound_of_Derivative_of_vi:Bound_of_NiaBoundary}
    |N_{i}^{a}(\pm 0,t)|\leq C(\delta+P(t))^2(t+1)^{-3}.
  \end{equation}
  By \eqref{SectionIII:CNS_Lag_per_Fundamental_Solution_Modified_Explicit_Formula}, \eqref{SectionIII:Bound_of_G-G*}, \eqref{SectionIII:GbIntByParts}, \eqref{SectionIII:Bounds_of_Gb}, \eqref{SectionIII:Lemma:Q0N}, \eqref{SectionIII:Lemma:Proof:Bound_of_Derivative_of_vi:Bound_of_Nixa} and~\eqref{SectionIII:Lemma:Proof:Bound_of_Derivative_of_vi:Bound_of_NiaBoundary},
  \begin{align}
    |B_{i}^{a}(x,t)|
    & \leq C(\delta+P(t))^2(t+1)^{-3/2}\int_{t/2}^{t-1}(t-s)^{-1/2}\, ds+C\sum_{j=1}^{2}\sup_{t/2\leq s\leq t-1}|\partial_x N_{j}^{a}(\cdot,s)|_{\infty} \\
    & \quad +C(\delta+P(t))^2(t+1)^{-3}\int_{t/2}^{t-1}(t-s)^{-1}\, ds \\
    & \leq C(\delta+P(t))^2(t+1)^{-1}.
  \end{align}
  Next, by integration by parts,
  \begin{align}
    B_{i}^{b}(x,t)
    & =\int_{t/2}^{t-1}\int_{-\infty}^{\infty}\partial_{x}^{2}g_i(x-y,t-s)
    \begin{pmatrix}
      N_{1}^{b} \\
      N_{2}^{b}
    \end{pmatrix}
    (y,s)\, dyds \\
    & \quad -\int_{t/2}^{t-1}\int_{0}^{\infty}\partial_{x}^{2}g_{R,i}(x+y,t-s)
    \begin{pmatrix}
      N_{1}^{b} \\
      N_{2}^{b}
    \end{pmatrix}
    (y,s)\, dyds \\
    & \quad +\int_{t/2}^{t-1}\int_{-\infty}^{0}\partial_{x}^{2}g_{R,i}(x-y,t-s)
    \begin{pmatrix}
      N_{1}^{b} \\
      N_{2}^{b}
    \end{pmatrix}
    (y,s)\, dyds.
  \end{align}
  Note that
  \begin{equation}
    \label{SectionIII:Lemma:Proof:Bound_of_Derivative_of_vi:Bound_of_Nib}
    |N_{i}^{b}(x,t)|\leq C(\delta+P(t))^2(t+1)^{-3/2}
  \end{equation}
  and
  \begin{equation}
    \label{SectionIII:Lemma:Proof:Bound_of_Derivative_of_vi:Bound_of_Nixb}
    |N_{ix}^{b}(x,t)|\leq C(\delta+P(t))^2(t+1)^{-1}.
  \end{equation}
  By \eqref{SectionIII:CNS_Lag_per_Fundamental_Solution_Modified_Explicit_Formula}, \eqref{SectionIII:Bound_of_G-G*}, \eqref{SectionIII:GbIntByParts}, \eqref{SectionIII:Bounds_of_Gb}, \eqref{SectionIII:Lemma:Q0N}, \eqref{SectionIII:Lemma:Proof:Bound_of_Derivative_of_vi:Bound_of_Nib} and~\eqref{SectionIII:Lemma:Proof:Bound_of_Derivative_of_vi:Bound_of_Nixb},
  \begin{align}
    |B_{i}^{b}(x,t)|
    & \leq C(\delta+P(t))^2(t+1)^{-3/2}\int_{t/2}^{t-1}(t-s)^{-1}\, ds+C\sum_{j=1}^{2}\sum_{l=0}^{1}\sup_{t/2\leq s\leq t-1}|\partial_{x}^{l}N_{j}^{b}(\cdot,s)|_{\infty} \\
    & \leq C(\delta+P(t))^2(t+1)^{-1}.
  \end{align}
  Thus $|B_i(x,t)|\leq C(\delta+P(t))^2(t+1)^{-1}$ as desired.

  Let us next consider $C_i(x,t)$, which is defined by
  \begin{align}
    C_i(x,t)
    & =\int_{t-1}^{t}\int_{-\infty}^{\infty}\partial_x g_i(x-y,t-s)
    \begin{pmatrix}
      N_1 \\
      N_2
    \end{pmatrix}_x
    (y,s)\, dyds \\
    & \quad +\int_{t-1}^{t}\int_{0}^{\infty}\partial_x g_{R,i}(x+y,t-s)
    \begin{pmatrix}
      N_1 \\
      N_2
    \end{pmatrix}_x
    (y,s)\, dyds \\
    & \quad +\int_{t-1}^{t}\int_{-\infty}^{0}\partial_x g_{R,i}(x-y,t-s)
    \begin{pmatrix}
      N_1 \\
      N_2
    \end{pmatrix}_x
    (y,s)\, dyds \\
    & \quad +\int_{t-1}^{t}\partial_x g_{T,i}(x,t-s)
    \begin{pmatrix}
      \llbracket N_1 \rrbracket \\
      \llbracket N_2 \rrbracket
    \end{pmatrix}
    (0,s)\, ds.
  \end{align}
  Note that
  \begin{equation}
    \label{SectionIII:Lemma:Proof:Bound_of_Derivative_of_vi:Bound_of_NiBoundary}
    |N_i(\pm 0,t)|\leq C(\delta+P(t))^2(t+1)^{-5/2}.
  \end{equation}
  By \eqref{SectionIII:CNS_Lag_per_Fundamental_Solution_Modified_Explicit_Formula}, \eqref{SectionIII:Bound_of_G-G*}, \eqref{SectionIII:GbIntByParts}, \eqref{SectionIII:Bounds_of_Gb}, \eqref{SectionIII:Lemma:Q0N}, \eqref{SectionIII:Lemma:Proof:Bound_of_Derivative_of_vi:Bound_of_Nixa}, \eqref{SectionIII:Lemma:Proof:Bound_of_Derivative_of_vi:Bound_of_Nixb} and~\eqref{SectionIII:Lemma:Proof:Bound_of_Derivative_of_vi:Bound_of_NiBoundary},
  \begin{align}
    |C_i(x,t)|
    & \leq C(\delta+P(t))^2(t+1)^{-1}\int_{t-1}^{t}(t-s)^{-1/2}\, ds+C\sum_{j=1}^{2}\sup_{t-1\leq s\leq t}|\partial_x N_j(\cdot,s)|_{\infty} \\
    & \quad +C(\delta+P(t))^2(t+1)^{-5/2}\int_{t-1}^{t}(t-s)^{-1/2}\, ds \\
    & \leq C(\delta+P(t))^2(t+1)^{-1}.
  \end{align}
  This completes the proof.
\end{proof}

Let us next give a bound for $u_{xx}(x,t)$.

\begin{lemma}
  \label{SectionIII:Lemma:Bound_of_uxx}
  Under the assumptions of Theorem~\ref{SectionI:Theorem:MainTheorem:PointwiseEstimates}, there exists a positive constant $C$ such that
  \begin{equation}
    |u_{xx}(x,t)|\leq C\delta(t+1)^{-1}+C(\delta+P(t))^2(t+1)^{-1/2}
  \end{equation}
  for $(x,t)\in \mathbb{R}_* \times (0,\infty)$.
\end{lemma}

\begin{proof}
  By \eqref{SectionII:GlobalEnergyEstimate:Proposition:Energy_Estimate_of_taux:Proof:uxx}, it suffices to show that
  \begin{equation}
    \label{SectionIII:Lemma:Bound_of_uxx:Proof:Bound_of_ut}
    |u_t(x,t)|\leq C\delta(t+1)^{-1}+C(\delta+P(t))^2(t+1)^{-1/2}.
  \end{equation}
  Theorem~\ref{SectionI:Theorem:MainTheorem:GlobalExistence} shows that this holds for $t\leq 2$. Hence, let $t\geq 2$ in the following. We also assume that $x>0$ since the case when $x<0$ is similar.
  
  Since
  \begin{equation}
    \begin{dcases}
      \partial_t \tau_t-\partial_x u_t=0,                                             & x\in \mathbb{R}_*,\, t>0, \\
      \partial_t u_t -c^2 \partial_x \tau_t=\nu \partial_{x}^{2}u_t+\partial_x N_t,   & x\in \mathbb{R}_*,\, t>0, \\
      \partial_t u_t(\pm 0,t)=\llbracket c^2 \tau_t+\nu \partial_x u_t \rrbracket(0,t)
                            +\llbracket N_t \rrbracket(0,t),                          & t>0,
    \end{dcases}
  \end{equation}
  the same calculations leading to~\eqref{SectionIII:Proposition:IntegralEquation:IntegralEquation} yield
  \begin{align}
    \label{SectionIII:Lemma:Bound_of_uxx:Proof:IntegralEquationTimeDerivative}
    \begin{aligned}
      \begin{pmatrix}
        \tau_t \\
        u_t
      \end{pmatrix}
      (x,t) & =\int_{0}^{\infty}G(x-y,t)
      \begin{pmatrix}
        \tau_t \\
        u_t
      \end{pmatrix}
      (y,0)\, dy+\int_{0}^{\infty}G_R(x+y,t)
      \begin{pmatrix}
        \tau_t \\
        u_t
      \end{pmatrix}
      (y,0)\, dy \\
      & \quad +\int_{-\infty}^{0}G_T(x-y,t)
      \begin{pmatrix}
        \tau_t \\
        u_t
      \end{pmatrix}
      (y,0)\, dy+G_T(x,t)
      \begin{pmatrix}
        0 \\
        V'(0)
      \end{pmatrix} \\
      & \quad +\int_{0}^{t}\int_{0}^{\infty}G(x-y,t-s)
      \begin{pmatrix}
        0 \\
        N_{tx}
      \end{pmatrix}
      (y,s)\, dyds \\
      & \quad +\int_{0}^{t}\int_{0}^{\infty}G_R(x+y,t-s)
      \begin{pmatrix}
        0 \\
        N_{tx}
      \end{pmatrix}
      (y,s)\, dyds \\
      & \quad +\int_{0}^{t}\int_{-\infty}^{0}G_T(x-y,t-s)
      \begin{pmatrix}
        0 \\
        N_{tx}
      \end{pmatrix}
      (y,s)\, dyds \\
      & \quad +\int_{0}^{t}G_T(x,t-s)
      \begin{pmatrix}
        0 \\
        \llbracket N_t \rrbracket
      \end{pmatrix}
      (0,s)\, ds.
    \end{aligned}
  \end{align}

  Let us first consider the sum of the first four terms on the right-hand side of~\eqref{SectionIII:Lemma:Bound_of_uxx:Proof:IntegralEquationTimeDerivative}. By~\eqref{SectionI:CNSwithPointMass:Lagrangian:Perturbation} and integration by parts in $x$, this can be written as
  \begin{align}
    \bm{\mathcal{I}}_t(x,t)
    & \coloneqq \int_{0}^{\infty}\partial_x G(x-y,t)
    \begin{pmatrix}
      u_0 \\
      -p(1+\tau_0)+p(1)+\nu u_{0x}/(1+\tau_0)
    \end{pmatrix}
    (y)\, dy \\
    & \quad -\int_{0}^{\infty}\partial_x G_R(x+y,t)
    \begin{pmatrix}
      u_0 \\
      -p(1+\tau_0)+p(1)+\nu u_{0x}/(1+\tau_0)
    \end{pmatrix}
    (y)\, dy \\
    & \quad +\int_{-\infty}^{0}\partial_x G_T(x-y,t)
    \begin{pmatrix}
      u_0 \\
      -p(1+\tau_0)+p(1)+\nu u_{0x}/(1+\tau_0)
    \end{pmatrix}
    (y)\, dy \\
    & \quad -2G_R(x,t)
    \begin{pmatrix}
      V_0 \\
      0
    \end{pmatrix}.
  \end{align}
  By the assumptions of Theorem~\ref{SectionI:Theorem:MainTheorem:PointwiseEstimates}, \eqref{SectionIII:CNS_Lag_per_Fundamental_Solution_Modified_Explicit_Formula}, \eqref{SectionIII:Bound_of_G-G*}, \eqref{SectionIII:GbIntByParts} and~\eqref{SectionIII:Bounds_of_Gb},
  \begin{equation}
    |\bm{\mathcal{I}}_t(x,t)|\leq C\delta(t+1)^{-1}.
  \end{equation}
  
  Next, let us consider the sum of the last four terms on the right-hand side of~\eqref{SectionIII:Lemma:Bound_of_uxx:Proof:IntegralEquationTimeDerivative}. Split the domain of integration in $t$ into $[0,t-1]$ and $[t-1,t]$, and denote by $\bm{D}(x,t)$ and $\bm{E}(x,t)$ the corresponding terms.

  Let us first consider $\bm{D}(x,t)$. By integration by parts in $x$,
  \begin{align}
    \label{SectionIII:Lemma:Bound_of_uxx:Proof:bmD}
    \begin{aligned}
      \bm{D}(x,t)
      & =\int_{0}^{t-1}\int_{-\infty}^{\infty}\partial_x G(x-y,t-s)
      \begin{pmatrix}
        0 \\
        N_t
      \end{pmatrix}
      (y,s)\, dyds \\
      & \quad -\int_{0}^{t-1}\int_{0}^{\infty}\partial_x G_R(x+y,t-s)
      \begin{pmatrix}
        0 \\
        N_t
      \end{pmatrix}
      (y,s)\, dyds \\
      & \quad +\int_{0}^{t-1}\int_{-\infty}^{0}\partial_x G_R(x-y,t-s)
      \begin{pmatrix}
        0 \\
        N_t
      \end{pmatrix}
      (y,s)\, dyds.
    \end{aligned}
  \end{align}
  By integration by parts in $t$, the first term on the right-hand side of~\eqref{SectionIII:Lemma:Bound_of_uxx:Proof:bmD} can be written as
  \begin{align}
    \label{SectionIII:Lemma:Bound_of_uxx:Proof:bmD1st}
    \begin{aligned}
      & \int_{0}^{t-1}\int_{-\infty}^{\infty}\partial_{tx}G(x-y,t-s)
      \begin{pmatrix}
        0 \\
        N
      \end{pmatrix}
      (y,s)\, dyds \\
      & +\int_{-\infty}^{\infty}\partial_x G(x-y,1)
      \begin{pmatrix}
        0 \\
        N
      \end{pmatrix}
      (y,t-1)\, dy-\int_{-\infty}^{\infty}\partial_x G(x-y,t)
      \begin{pmatrix}
        0 \\
        N
      \end{pmatrix}
      (y,0)\, dy.
    \end{aligned}
  \end{align}
  Note that
  \begin{equation}
    \label{SectionIII:Lemma:Bound_of_uxx:Proof:Bound_of_NLinf}
    \sum_{l=0}^{1}|\partial_{x}^{l}N(x,t)|\leq C(\delta+P(t))^2(t+1)^{-1},
  \end{equation}
  \begin{equation}
    \label{SectionIII:Lemma:Bound_of_uxx:Proof:Bound_of_NL1}
    ||N(\cdot,t)||_{L^1(\mathbb{R}_*)}\leq C(\delta+P(t))^2(t+1)^{-1/2}
  \end{equation}
  and
  \begin{equation}
    \label{SectionIII:Lemma:Q1N}
    Q_1
    \begin{pmatrix}
      0 \\
      N
    \end{pmatrix}
    =
    \begin{pmatrix}
      -N/\nu \\
      0
    \end{pmatrix}
    , \quad
    \begin{pmatrix}
      0 & 0 \\
      0 & \nu
    \end{pmatrix}
    Q_1
    \begin{pmatrix}
      0 \\
      N
    \end{pmatrix}
    =
    \begin{pmatrix}
      0 \\
      0
    \end{pmatrix}.
  \end{equation}
  Thus, by Theorem~\ref{SectionI:Theorem:MainTheorem:GlobalExistence}, \eqref{SectionIII:CNS_Lag_per_Fundamental_Solution}, \eqref{SectionIII:CNS_Lag_per_Fundamental_Solution_Modified_Explicit_Formula}, \eqref{SectionIII:Bound_of_G-G*}, \eqref{SectionIII:Lemma:Q0N}, \eqref{SectionIII:Lemma:Bound_of_uxx:Proof:Bound_of_NLinf}, \eqref{SectionIII:Lemma:Bound_of_uxx:Proof:Bound_of_NL1} and~\eqref{SectionIII:Lemma:Q1N}, the second component of~\eqref{SectionIII:Lemma:Bound_of_uxx:Proof:bmD1st} is bounded by
  \begin{align}
    & C(\delta+P(t))^2(t+1)^{-1}+C(\delta+P(t))^2(t+1)^{-3/2}\int_{0}^{t/2}(s+1)^{-1/2}\, ds \\
    & +C(\delta+P(t))^2(t+1)^{-1}\int_{t/2}^{t-1}(t-s)^{-1}\, ds \\
    & +C\int_{0}^{t/2}e^{-\frac{t-s}{C}}\, ds\sup_{0\leq s\leq t/2}||N(\cdot,s)||_2+C\sum_{l=0}^{1}\sup_{t/2\leq s\leq t-1}|\partial_{x}^{l}N(\cdot,s)|_{\infty} \\
    & \leq C(\delta+P(t))^2(t+1)^{-1/2}.
  \end{align}
  The second components of the second and the third terms on the right-hand side of~\eqref{SectionIII:Lemma:Bound_of_uxx:Proof:bmD} are treated in a similar manner except for the terms involving $e^{-(|x|+t)/C}$ appearing on the right-hand side of~\eqref{SectionIII:Bounds_of_Gb}, which can be handled easily.

  Let us next consider $\bm{E}(x,t)$. By integration by parts in $x$,
  \begin{align}
    \label{SectionIII:Lemma:Bound_of_uxx:Proof:bmE}
    \bm{E}(x,t)
    & =\int_{t-1}^{t}\int_{-\infty}^{\infty}\partial_x G(x-y,t-s)
    \begin{pmatrix}
      0 \\
      N_t
    \end{pmatrix}
    (y,s)\, dyds \\
    & \quad -\int_{t-1}^{t}\int_{0}^{\infty}\partial_x G_R(x+y,t-s)
    \begin{pmatrix}
      0 \\
      N_t
    \end{pmatrix}
    (y,s)\, dyds \\
    & \quad +\int_{t-1}^{t}\int_{-\infty}^{0}\partial_x G_R(x-y,t-s)
    \begin{pmatrix}
      0 \\
      N_t
    \end{pmatrix}
    (y,s)\, dyds.
  \end{align}
  Note that
  \begin{equation}
    \label{SectionIII:Lemma:Bound_of_uxx:Proof:Bound_of_NtLinf}
    |N_t(x,t)|\leq C(\delta+P(t))^2(t+1)^{-1/2}.
  \end{equation}
  By \eqref{SectionIII:CNS_Lag_per_Fundamental_Solution_Modified_Explicit_Formula}, \eqref{SectionIII:Bound_of_G-G*}, \eqref{SectionIII:GbIntByParts}, \eqref{SectionIII:Bounds_of_Gb}, \eqref{SectionIII:Lemma:Q0N} and~\eqref{SectionIII:Lemma:Bound_of_uxx:Proof:Bound_of_NtLinf}, $|\bm{E}(x,t)|$ is bounded by
  \begin{equation}
    C(\delta+P(t))^2(t+1)^{-1/2}\int_{t-1}^{t}(t-s)^{-1/2}\, ds\leq C(\delta+P(t))^2(t+1)^{-1/2}.
  \end{equation}
  This completes the proof.
\end{proof}

Combining Lemmas~\ref{SectionIII:Lemma:Bound_of_Ii}--\ref{SectionIII:Lemma:Bound_of_uxx}, we obtain
\begin{equation}
  P(t)\leq C\delta+C(\delta+P(t))^2.
\end{equation}
Since $P(t)$ is continuous in $t$, by taking $\delta$ sufficiently small, we conclude that
\begin{equation}
  P(t)\leq C\delta.
\end{equation}
This proves Theorem~\ref{SectionI:Theorem:MainTheorem:PointwiseEstimates}.

\section*{Acknowledgements}
This work was supported by the Grant-in-Aid for JSPS Research Fellow (Grant Number 18J20574) and the JSPS Core-to-Core Program ``Foundation of a Global Research Cooperative Center in Mathematics focused on Number Theory and Geometry''. The author thanks Tatsuo Iguchi for the discussion on the subject.

\appendix
\section{Proof of the Bounds of $G_T$}
\label{AppendixA}
In this appendix, we prove~\eqref{SectionIII:Bounds_of_Gb}. The proof is basically the same as that of~\cite[Lemma~2.1]{DW18}.

For $\lambda \neq 0$ and $\mu>0$, let
\begin{equation}
  E(x,t;\lambda,\mu)\coloneqq \int_{-\infty}^{0}e^{2z}e^{-\frac{(x-z-\lambda t)^2}{\mu t}}\, dz.
\end{equation}
Using the complementary error function $\erfc(x)=2\pi^{-1/2}\int_{x}^{\infty}e^{-y^2}\, dy$, we can rewrite $E(x,t;\lambda,\mu)$ as
\begin{equation}
  E(x,t;\lambda,\mu)=\frac{\sqrt{\pi \mu t}}{2}e^{2(x-\lambda t)+\mu t}\erfc\left( \frac{x-\lambda t+\mu t}{\sqrt{\mu t}} \right).
\end{equation}

\begin{lemma}
  \label{Appendix:Gb}
  We have
  \begin{equation}
    t^{-1/2}E(x,t;\lambda,\mu)\leq C(t+1)^{-1/2}e^{-\frac{(x-\lambda t)^2}{Ct}}+Ce^{-\frac{|x|+t}{C}}.
  \end{equation}
\end{lemma}

\begin{proof}
  Case (i): $x-\lambda t+\mu t\leq 0$. Since
  \begin{equation}
    x-\lambda t=p(x-\lambda t)+(1-p)(x-\lambda t)\leq p(x-\lambda t)-(1-p)\mu t
  \end{equation}
  for $0<p<1$, by taking $p$ sufficiently small, we have
  \begin{align}
    t^{-1/2}E(x,t;\lambda,\mu)
    & \leq Ce^{-2p|x-\lambda t|-2\{ (1-p)-1/2 \} \mu t} \\
    & \leq Ce^{-2p|x|+\{ 2p|\lambda|-2(1-p)\mu+\mu \} t} \\
    & \leq Ce^{-\frac{|x|+t}{C}}.
  \end{align}
  Case (ii): $0<x-\lambda t+\mu t<Kt^{1/2}$. Here, $K$ is a positive constant. Since $|x|\leq Ct$ and
  \begin{equation}
    x-\lambda t+\mu t/2<-\mu t/2+Kt^{1/2},
  \end{equation}
  we have
  \begin{equation}
    t^{-1/2}E(x,t;\lambda,\mu)\leq Ce^{-\mu t+2Kt^{1/2}}\leq Ce^{-\frac{t}{C}}\leq Ce^{-\frac{|x|+t}{C}}.
  \end{equation}
  Case (iii): $x-\lambda t+\mu t\geq Kt^{1/2}$. Since
  \begin{equation}
    \erfc(x)=\frac{e^{-x^2}}{\sqrt{\pi}x}+O(x^{-3}e^{-x^2}) \quad \text{as $x\to +\infty$},
  \end{equation}
  by taking $K$ large enough, we have
  \begin{equation}
    t^{-1/2}E(x,t;\lambda,\mu)\leq Ce^{-\frac{(x-\lambda t)^2}{\mu t}}\frac{\sqrt{t}}{x-\lambda t+\mu t}.
  \end{equation}
  Case (iii.a): $|x-\lambda t|\leq \mu t/2-1$. In this case,
  \begin{equation}
    t^{-1/2}E(x,t;\lambda,\mu)\leq C(t+1)^{-1/2}e^{-\frac{(x-\lambda t)^2}{\mu t}}.
  \end{equation}
  Case (iii.b): $|x-\lambda t|>\mu t/2-1$. In this case,
  \begin{equation}
    t^{-1/2}E(x,t;\lambda,\mu)\leq CK^{-1}e^{-\frac{t}{C}}e^{-\frac{(x-\lambda t)^2}{Ct}}\leq C(t+1)^{-1/2}e^{-\frac{(x-\lambda t)^2}{Ct}}.
  \end{equation}
  This completes the proof.
\end{proof}

We can now prove~\eqref{SectionIII:Bounds_of_Gb}. Let $x>0$; the case when $x<0$ is similar. By~\eqref{SectionIII:CNS_Lag_per_Fundamental_Solution_Modified_Explicit_Formula}, \eqref{SectionIII:Bound_of_G-G*}, \eqref{SectionIII:GbFormula} and~Lemma~\ref{Appendix:Gb}, we obtain
\begin{equation}
  \label{Appendix:Bound_of_Gbk}
  |\partial_{x}^{k}G_T(x,t)|\leq C(t+1)^{-1/2}t^{-k/2}\left( e^{-\frac{(x-ct)^2}{Ct}}+e^{-\frac{(x+ct)^2}{Ct}} \right)+Ct^{-k/2}e^{-\frac{|x|+t}{C}}.
\end{equation}
This proves~\eqref{SectionIII:Bounds_of_Gb} when $t\geq 1$. When $t\leq 1$, use~\eqref{SectionIII:CNS_Lag_per_Fundamental_Solution_Modified_Explicit_Formula}, \eqref{SectionIII:Bound_of_G-G*}, \eqref{SectionIII:Eq:derGT} and~\eqref{Appendix:Bound_of_Gbk} ($k=0$) to obtain
\begin{equation}
  |\partial_{x}^{k}G_T(x,t)|\leq Ct^{-k/2}\left( e^{-\frac{(x-ct)^2}{Ct}}+e^{-\frac{(x+ct)^2}{Ct}} \right)+Ce^{-\frac{|x|+t}{C}}.
\end{equation}
This proves~\eqref{SectionIII:Bounds_of_Gb} when $t\leq 1$.

\section{Lemmas Used in the Proof of Lemma~\ref{SectionIII:Lemma:Bound_of_Ni}}
\label{AppendixB}
In this appendix, we collect several lemmas in~\cite{LZ97} for ease of reference. For the proof, we basically just refer to~\cite{LZ97}, but there are few places where comments are needed; this is because the solutions we are dealing with may have discontinuity across $x=0$ and integration by parts produces a boundary term.

For $\lambda \in \mathbb{R}$ and $\alpha,\mu>0$, let
\begin{equation}
  \label{Def:Theta}
  \Theta_{\alpha}(x,t;\lambda,\mu)\coloneqq (t+1)^{-\alpha/2}e^{-\frac{(x-\lambda(t+1))^2}{\mu(t+1)}}
\end{equation}
and
\begin{equation}
  \psi_{\alpha}(x,t;\lambda)\coloneqq [(x-\lambda(t+1))^2+(t+1)]^{-\alpha/2}.
\end{equation}
We note that
\begin{equation}
  |\theta_i(x,t)|\leq C\delta \Theta_1(x,t;\lambda_i,2\nu)
\end{equation}
and
\begin{equation}
  |\Theta_{\alpha}(x,t;\lambda,\mu)|\leq C\psi_{\alpha}(x,t;\lambda).
\end{equation}
We also note that for $\lambda,\lambda' \in \mathbb{R}$ ($\lambda \neq \lambda'$), we have
\begin{equation}
  \label{Appendix:DifferentThetaProduct}
  |\Theta_{\alpha}(x,t;\lambda,\mu)\Theta_{\alpha}(x,t;\lambda',\mu)|\leq Ce^{-t/C}\left( e^{-\frac{(x-\lambda(t+1))^2}{2\mu(t+1)}}e^{-\frac{(x-\lambda'(t+1))^2}{2\mu(t+1)}} \right).
\end{equation}
This follows from
\begin{equation}
  (x-\lambda (t+1))^2+(x-\lambda' (t+1))^2 \geq 2\left( \frac{\lambda-\lambda'}{2} \right)^2(t+1)^2.
\end{equation}

The following lemma is~\cite[Lemma~3.2]{LZ97}.

\begin{lemma}
  \label{LZlemma:3.2}
  Let $\lambda \in \mathbb{R}$, $\alpha \geq 0$, $\beta>0$ and $\mu>0$. Then we have
  \begin{align}
    & \left| \int_{0}^{t/2}\int_{-\infty}^{\infty}(t-s)^{-1}(t+1-s)^{-\alpha/2}e^{-\frac{(x-y-\lambda(t-s))^2}{\mu(t-s)}}\Theta_{\beta}(y,s;\lambda,\mu)\, dyds \right| \\
    & \leq
    \begin{dcases}
      C\Theta_{\gamma}(x,t;\lambda,\mu) & \text{if $\beta \neq 3$}, \\
      C\Theta_{\gamma}(x,t;\lambda,\mu)\log(t+2) & \text{if $\beta=3$},
    \end{dcases}
  \end{align}
  where $\gamma=\alpha+\min(\beta,3)-1$ and
  \begin{align}
    & \left| \int_{t/2}^{t}\int_{-\infty}^{\infty}(t-s)^{-1}(t+1-s)^{-\alpha/2}e^{-\frac{(x-y-\lambda(t-s))^2}{\mu(t-s)}}\Theta_{\beta}(y,s;\lambda,\mu)\, dyds \right| \\
    & \leq
    \begin{dcases}
      C\Theta_{\gamma}(x,t;\lambda,\mu) & \text{if $\alpha \neq 1$}, \\
      C\Theta_{\gamma}(x,t;\lambda,\mu)\log(t+2) & \text{if $\alpha=1$},
    \end{dcases}
  \end{align}
  where $\gamma=\min(\alpha,1)+\beta-1$.
\end{lemma}

For $\lambda,\lambda' \in \mathbb{R}$ and $K>0$, let
\begin{align}
  & \chi_K(x,t;\lambda,\lambda') \\
  & \coloneqq \mathrm{char}\left\{ \min(\lambda,\lambda')(t+1)+K(t+1)^{1/2}\leq x\leq \max(\lambda,\lambda')(t+1)-K(t+1)^{1/2} \right\},
\end{align}
where $\mathrm{char}\{ S \}$ is the characteristic function of a set $S$. The following lemma is~\cite[Lemma~3.3]{LZ97}.

\begin{lemma}
  \label{LZlemma:3.3}
  Let $\lambda,\lambda' \in \mathbb{R}$ ($\lambda \neq \lambda'$), $\alpha \geq 0$, $\beta \geq 1$ and $\mu>0$. Then for any $\varepsilon>0$ and $K\geq |\lambda-\lambda'|$, we have
  \begin{align}
    & \left| \int_{0}^{t}\int_{-\infty}^{\infty}(t-s)^{-1}(t+1-s)^{-\alpha/2}e^{-\frac{(x-y-\lambda(t-s))^2}{\mu(t-s)}}\Theta_{\beta}(y,s;\lambda',\mu)\, dyds \right| \\
    & \leq C[\Theta_{\gamma}(x,t;\lambda,\mu+\varepsilon)+\Theta_{\gamma}(x,t;\lambda',\mu+\varepsilon) \\
    & \phantom{\leq C} \quad +|x-\lambda(t+1)|^{-(\beta-1)/2}|x-\lambda'(t+1)|^{-(\alpha+1)/2}\chi_K(x,t;\lambda,\lambda')] \\
    & \quad +
    \begin{dcases}
      0 & \text{if $\beta \neq 3$}, \\
      C\Theta_{\gamma}(x,t;\lambda,\mu+\varepsilon)\log(t+1) & \text{if $\beta=3$},
    \end{dcases} \\
    & \quad +
    \begin{dcases}
      0 & \text{if $\alpha \neq 1$}, \\
      C\Theta_{\gamma}(x,t;\lambda',\mu+\varepsilon)\log(t+1) & \text{if $\alpha=1$},
    \end{dcases}
  \end{align}
  where $\gamma=\min(\alpha,1)+\min(\beta,3)-1$.
\end{lemma}

\begin{remark}
  \label{Remark:LZlemma:3.3}
  A closer look at the proof of~\cite[Lemma~3.3]{LZ97} shows that --- as in Lemma~\ref{LZlemma:3.2} ---  if the domain of integration in $t$ is restricted to $[0,t/2]$, we can take away the $\log(t+1)$ factor when $\alpha=1$ in Lemma~\ref{LZlemma:3.3}. Similarly, if the domain of integration in $t$ is restricted to $[t/2,t]$, we can take away the $\log(t+1)$ factor when $\beta=3$ in Lemma~\ref{LZlemma:3.3}.
\end{remark}

The following lemma is a slightly generalized version of~\cite[Lemma~3.4]{LZ97}.

\begin{lemma}
  \label{LZlemma:3.4}
  Let $\lambda,\lambda' \neq 0$ ($\lambda \neq \lambda'$), $1\leq \alpha<3$ and $\mu,\mu'>0$. Let $k\geq 0$ be an integer and $L\coloneqq \partial_t+\lambda' \partial_x -(\mu/4)\partial_{x}^{2}$. Suppose that a function $h=h(x,t)$ satisfies the following inequalities for $(x,t)\in \mathbb{R}\times (0,\infty)$:
  \begin{align}
    \label{LZlemma:3.4:Assumptions}
    \begin{aligned}
      |h(x,t)| & \leq C\Theta_{\alpha}(x,t;\lambda',\mu'), \\
      |\partial_{x}^{k}h(x,t)| & \leq C\Theta_{\alpha+k}(x,t;\lambda',\mu'), \\
      |Lh(x,t)-\partial_x F_{\alpha+1}(x,t;\lambda',\mu')| & \leq C\Theta_{\alpha+2}(x,t;\lambda',\mu'), \\
      |\partial_{x}^{k}Lh(x,t)-\partial_x F_{\alpha+k+1}(x,t;\lambda',\mu')| & \leq C\Theta_{\alpha+k+2}(x,t;\lambda',\mu'),
    \end{aligned}
  \end{align}
  where $F_{\beta}(x,t;\lambda',\mu')$ is a function with $|F_{\beta}(x,t;\lambda',\mu')|\leq C\Theta_{\beta}(x,t;\lambda',\mu')$. Then for any $\varepsilon>0$ and $K\geq |\lambda-\lambda'|$, we have
  \begin{align}
    \label{LZlemma:3.4:Bound}
    \begin{aligned}
      & \left| \int_{0}^{t}\int_{-\infty}^{\infty}\partial_x \left\{ (t-s)^{-1/2}e^{-\frac{(x-y-\lambda(t-s))^2}{\mu(t-s)}} \right\} \partial_{x}^{k}h(y,s)\, dyds \right| \\
      & \leq C(t+1)^{-k/2}[\psi_{(\alpha+1)/2}(x,t;\lambda)+\Theta_{\min(\alpha,2)}(x,t;\lambda',\mu^*+\varepsilon) \\
      & \phantom{\leq C(t+1)^{-k/2}} \quad +|x-\lambda(t+1)|^{-\alpha/2}|x-\lambda'(t+1)|^{-1/2}\chi_K(x,t;\lambda,\lambda')],
    \end{aligned}
  \end{align}
  where $\mu^*=\max(\mu,\mu')$. Furthermore, when $k=0$, we only need to assume that~\eqref{LZlemma:3.4:Assumptions} holds for $(x,t)\in \mathbb{R}_* \times (0,\infty)$; in particular, $h$ is allowed to have discontinuity across $x=0$.
\end{lemma}

\begin{proof}
  We only give a proof for the last remark on the case of $k=0$ (see~\cite[Lemma~3.4]{LZ97} for the proof when $k\geq 1$). Let
  \begin{align}
    I_1(x,t) & \coloneqq \int_{0}^{t^{1/2}}\int_{-\infty}^{\infty}\partial_x \left\{ (t-s)^{-1/2}e^{-\frac{(x-y-\lambda(t-s))^2}{\mu(t-s)}} \right\} h(y,s)\, dyds, \\
    I_2(x,t) & \coloneqq \int_{t^{1/2}}^{t}\int_{-\infty}^{\infty}\partial_x \left\{ (t-s)^{-1/2}e^{-\frac{(x-y-\lambda(t-s))^2}{\mu(t-s)}} \right\} h(y,s)\, dyds.
  \end{align}
  To bound $I_1(x,t)$, we do not need to conduct integration by parts, and the calculations in the proof of~\cite[Lemma~3.4]{LZ97} show that $I_1(x,t)$ is bounded by the right-hand side of~\eqref{LZlemma:3.4:Bound}. For $I_2(x,t)$, integration by parts yields,
  \begin{align}
    \label{LZlemma:3.4:I2_IBP}
    \begin{aligned}
      I_2(x,t)
      & \coloneqq \int_{t^{1/2}}^{t}\int_{-\infty}^{\infty}(t-s)^{-1/2}e^{-\frac{(x-y-\lambda(t-s))^2}{\mu(t-s)}}\partial_x h(y,s)\, dyds \\
      & \quad +\int_{t^{1/2}}^{t}(t-s)^{-1/2}e^{-\frac{(x-\lambda(t-s))^2}{\mu(t-s)}}\llbracket h \rrbracket(0,s)\, ds.
    \end{aligned}
  \end{align}
  The calculations in the proof of~\cite[Lemma~3.4]{LZ97} show that the first term on the right-hand side of~\eqref{LZlemma:3.4:I2_IBP} is bounded by the right-hand side of~\eqref{LZlemma:3.4:Bound}. Now, note that since $\lambda'\neq 0$, we have
  \begin{equation}
    |h(\pm 0,t)|\leq Ce^{-\frac{t}{C}}.
  \end{equation}
  Thus we only need to show that
  \begin{equation}
    I_{2b}(x,t)\coloneqq \int_{t^{1/2}}^{t}(t-s)^{-1/2}e^{-\frac{(x-\lambda(t-s))^2}{\mu(t-s)}}e^{-\frac{s}{C}}\, ds
  \end{equation}
  is bounded by the right-hand side of~\eqref{LZlemma:3.4:Bound}. Let
  \begin{equation}
    A_1\coloneqq \{ t^{1/2}\leq s\leq t \mid |\lambda|s\leq |x-\lambda t|/2 \},\quad A_2\coloneqq \{ t^{1/2}\leq s\leq t \mid |\lambda|s>|x-\lambda t|/2 \}.
  \end{equation}
  If $s\in A_1$, we have
  \begin{equation}
    |x-\lambda(t-s)|\geq |x-\lambda t|/2;
  \end{equation}
  and if $s\in A_2$, we have
  \begin{equation}
    e^{-\frac{s}{C}}\leq e^{-\frac{\sqrt{t}}{C}}e^{-\frac{|x-\lambda t|}{C}}.
  \end{equation}
  Thus
  \begin{align}
    |I_{2b}(x,t)|
    & \leq e^{-\frac{\sqrt{t}}{C}}e^{-\frac{(x-\lambda t)^2}{Ct}}\int_{t^{1/2}}^{t}(t-s)^{-1/2}\, ds \\
    & \quad +e^{-\frac{\sqrt{t}}{C}}e^{-\frac{|x-\lambda t|}{C}}\int_{t^{1/2}}^{t}(t-s)^{-1/2}\, ds \\
    & \leq C\psi_{(\alpha+1)/2}(x,t;\lambda).
  \end{align}
  This completes the proof.
\end{proof}

The following lemma is~\cite[Lemma~3.5]{LZ97}.

\begin{lemma}
  \label{LZlemma:3.5}
  Let $\lambda \in \mathbb{R}$, $\alpha,\beta \geq 0$ and $\mu>0$. Then we have
  \begin{align}
    & \left| \int_{0}^{t}\int_{-\infty}^{\infty}(t-s)^{-1}(t+1-s)^{-\alpha/2}e^{-\frac{(x-y-\lambda(t-s))^2}{\mu(t-s)}}(s+1)^{-\beta/2}\psi_{3/2}(y,s;\lambda)\, dyds \right| \\
    & \leq
    \begin{dcases}
      C(t+1)^{-\gamma/2}\psi_{3/2}(x,t;\lambda) & \text{if $\alpha \neq 1$ and $\beta \neq 3/2$}, \\
      C(t+1)^{-\gamma/2}\log(t+2)\psi_{3/2}(x,t;\lambda) & \text{if $\alpha=1$ or $\beta=3/2$},
    \end{dcases}
  \end{align}
  where $\gamma=\min(\alpha,1)+\min(\beta,3/2)-1$.
\end{lemma}

The following lemma is~\cite[Lemma~3.6]{LZ97}.

\begin{lemma}
  \label{LZlemma:3.6}
  Let $\lambda,\lambda' \in \mathbb{R}$ ($\lambda \neq \lambda'$), $\alpha,\beta \geq 0$ and $\mu>0$. Then for any $K>2|\lambda-\lambda'|$, we have
  \begin{align}
    & \left| \int_{0}^{t}\int_{-\infty}^{\infty}(t-s)^{-1}(t+1-s)^{-\alpha/2}e^{-\frac{(x-y-\lambda(t-s))^2}{\mu(t-s)}}(s+1)^{-\beta/2}\psi_{3/2}(y,s;\lambda')\, dyds \right| \\
    & \leq C(t+1)^{-\gamma/2}[\psi_{3/2}(x,t;\lambda)+\psi_{3/2}(x,t;\lambda')] \\
    & \quad +C|x-\lambda(t+1)|^{-\min(\beta,5/2)/2-1/4}|x-\lambda'(t+1)|^{-\min(\alpha,1)/2-1/2}\chi_K(x,t;\lambda,\lambda') \\
    & \quad +
    \begin{dcases}
      0 & \text{if $\alpha \neq 1$ and $\beta \neq 3/2$}, \\
      C(t+1)^{-\gamma/2}\log(t+1)\psi_{3/2}(x,t;\lambda) & \text{if $\alpha \neq 1$ and $\beta=3/2$}, \\
      C(t+1)^{-\gamma/2}\log(t+1)[\psi_{3/2}(x,t;\lambda)+\psi_{3/2}(x,t;\lambda')] & \text{if $\alpha=1$},
    \end{dcases}
  \end{align}
  where $\gamma=\min(\alpha,1)+\min(\beta,3/2)-1$.
\end{lemma}

The following lemma is a slight modification of~\cite[Lemma~3.9]{LZ97}; the proof is completely analogous.

\begin{lemma}
  \label{LZlemma:3.9}
  Let $\lambda \in \mathbb{R}$ and $\mu>0$. Then we have
  \begin{equation}
    \left| \int_{0}^{t}e^{-\frac{t-s}{\mu}}(s+1)^{-1/4}\psi_{3/2}(x,s;\lambda)\, ds \right| \leq C(t+1)^{-1/4}\psi_{3/2}(x,t;\lambda).
  \end{equation}
\end{lemma}

\providecommand{\bysame}{\leavevmode\hbox to3em{\hrulefill}\thinspace}
\providecommand{\MR}{\relax\ifhmode\unskip\space\fi MR }
\providecommand{\MRhref}[2]{%
  \href{http://www.ams.org/mathscinet-getitem?mr=#1}{#2}
}
\providecommand{\href}[2]{#2}

\end{document}